\RequirePackage{snapshot}

\newif\ifsubsections

\documentclass[]{pcmi}

\usepackage[sc]{mathpazo}          
\usepackage{eulervm}               
\usepackage[scaled=0.86]{berasans} 
\usepackage[scaled=1]{inconsolata} 
\usepackage[T1]{fontenc}

\usepackage[%
	protrusion=true,
	expansion=false,
	auto=false
	]{microtype}

\usepackage{xcolor}
\usepackage{graphicx}
\graphicspath{{./figures/}}

\ifdraft
	\definecolor{linkred}{rgb}{0.7,0.2,0.2}
	\definecolor{linkblue}{rgb}{0,0.2,0.6}
\else
	\definecolor{linkred}{rgb}{0.0,0.0,0.0}
	\definecolor{linkblue}{rgb}{0,0.0,0.0}
\fi

\usepackage{amsmath, amssymb}

\usepackage{blkarray}
\usepackage{graphicx}
\usepackage{marginnote}
\usepackage{enumitem,comment}
\usepackage{tikz}
\usetikzlibrary{matrix,backgrounds,decorations.pathreplacing}
\usepackage{rotating}
\usepackage{psfrag}
\usepackage{mathtools}
\usepackage{mathrsfs}
\usepackage{cases}
\usepackage{stackengine}
\usepackage{bm}
\usepackage{bbm}
\usepackage[multiple]{footmisc}
\usepackage[totalheight=9.6in,totalwidth=6.15in,centering]{geometry}
\usepackage{graphics,graphicx}
\usepackage{tocvsec2}

\PassOptionsToPackage{hyphens}{url} 
\usepackage[
    setpagesize=false,
    pagebackref,
	pdfpagelabels=false,
    pdfstartview={FitH 1000},
    bookmarksnumbered=false,
    linktoc=all,
    colorlinks=true,
    anchorcolor=black,
    menucolor=black,
    runcolor=black,
    filecolor=black,
    linkcolor=linkblue,
	citecolor=linkblue,
	urlcolor=linkred,
]{hyperref}
\usepackage[backrefs,msc-links,nobysame]{amsrefs}

\customizeamsrefs 

\theoremstyle{plain}
\newtheorem{theorem}[equation]{Theorem}
\newtheorem{proposition}[equation]{Proposition}
\newtheorem{lemma}[equation]{Lemma}
\newtheorem{exercise}[equation]{Exercise}

\theoremstyle{definition}
\newtheorem{definition}[equation]{Definition}
\newtheorem{example}[equation]{Example}
\newtheorem{remark}[equation]{Remark}
\newtheorem{assumption}[equation]{Assumption}

\usepackage{definitions}

\newcommand{\uno}[1]{\mathbf{1}_{#1}}
\renewcommand{\SM}{\mathcal{S}}
\renewcommand{\SN}{\bar{\mathcal{S}}}

\begin{document}

%
%
%
%
%
%

\title{From the totally asymmetric simple exclusion process to the KPZ fixed point}

%
%
\author{Jeremy Quastel \and Konstantin Matetski}
\address{University of Toronto, 40 St. George Street, Toronto, Ontario, Canada M5S 2E4}
\email{quastel@math.toronto.edu, matetski@math.toronto.edu}
%
%

\begin{abstract}
These notes are based on the article [Matetski, Quastel, Remenik, \emph {The KPZ fixed point}, 2016] and give a self-contained exposition of construction of the KPZ fixed point which is a Markov process at the centre of the KPZ universality class. Starting from the Sch\"{u}tz's formula for transition probabilities of the totally asymmetric simple exclusion process, the method is by writing them as the biorthogonal ensemble/non-intersecting path representation found by Borodin, Ferrari, Pr\"{a}hofer and Sasamoto. We derive an explicit formula for the correlation kernel which involves transition probabilities of a random walk forced to hit a curve defined by the initial data. This in particular yields a Fredholm determinant formula for the multipoint distribution of the height function of the totally asymmetric simple exclusion process with arbitrary initial condition. In the 1:2:3 scaling limit the formula leads in a transparent way to a Fredholm determinant formula for the KPZ fixed point, in terms of an analogous kernel based on Brownian motion. The formula readily reproduces known special self-similar solutions such as the Airy$_1$ and Airy$_2$ processes.
\end{abstract}  

%
%
\maketitle
\thispagestyle{empty}

%
%


\section{The totally asymmetric simple exclusion process}

The \emph{totally asymmetric simple exclusion process} (TASEP) is a basic interacting particle system studied in non-equilibrium statistical mechanics. The system consists of particles performing totally asymmetric nearest neighbour random walks on the one-dimensional integer lattice with the exclusion rule. Each particle independently attempts to jump to the neighbouring site to the right at rate $1$, the jump being allowed only if that site is unoccupied. More precisely, if we denote by $\eta \in \{0,1\}^\Z$ a particle configuration (where $\eta_x = 1$ if there is a particle at the site $x$, and $\eta_x = 0$ if the site is empty), then TASEP is a Markov process with  infinitesimal generator  acting on cylinder functions $f : \{0,1\}^\Z \to \R$ by
\[
 	(L f)(\eta) = \sum_{x \in \Z} \eta_x (1 - \eta_{x+1}) \big(f(\eta^{x, x+1}) - f(\eta)\big),
\]
where $\eta^{x, x+1}$ denotes the configuration $\eta$ with interchanged values at $x$ and $x+1$: 
\[
 \eta^{x, x+1}_y = 
\begin{cases}
 	\eta_{x+1}, &~\text{if}~ y  =x,\\
	\eta_{x}, &~\text{if}~ y  =x + 1,\\
	\eta_{y}, &~\text{if}~ y  \notin \{x, x+1\}.\\
\end{cases}
\]
See~\cite{ligg1} for the proof of the non-trivial fact that this process is well-defined.

\begin{exercise}  Prove that the following measures $\mu$ are invariant for TASEP, i.e. $\int (Lf) d\mu=0$: 
\begin{enumerate}[topsep=0pt]
 	\item the Bernoulli product measures with any density $\rho\in [0,1]$,
	\item the Dirac measure on any configuration with $\eta_x =1$ for $x\ge x_0$, $\eta_x=0$ for $x<x_0$.
\end{enumerate}
It is known~\cite{ligg1} that these are the only invariant measures.
\end{exercise}

\noindent The TASEP dynamics preserves the order of particles. Let us denote positions of particles at time $t \geq 0$ by
\[
 	\cdots <X_t(2)<X_t(1)< X_t(0)< X_t(-1)<X_t(-2)< \cdots,
\]
where $X_t(i) \in \Z$ is the position of the $i$-{th} particle. Adding $\pm\infty$ into the state space and placing a necessarily infinite number of particles at infinity allows for left- or right-finite data with no change of notation (the particles at $\pm\infty$ are playing no role in the dynamics). We follow the standard
practice of ordering particles from the right; for right-finite data the rightmost particle is labelled $1$. 

TASEP is a particular case of the \emph{asymmetric simple exclusion process} (ASEP) introduced by Spitzer in~\cite{Spitzer}. Particles in this model jump to the right with rate $p$ and to the left with rate $q$ such that $p + q = 1$, following the exclusion rule. Obviously, in the case $p = 1$ we get TASEP. In the case $p \in (0,1)$ the model becomes significantly more complicated comparing to TASEP, for example Sch\"{u}tz's formula described in Section~\ref{sec:distr} below cannot be written as a determinant which prevents the following analysis in the general case.
ASEP is important because of the weakly asymmetric limit, which means to diffusively rescale the growth process introduced below as
$\ep^{1/2} h_{\ep^{-2} t}(\ep^{-1} z)$ while at the same time taking $q-p=\mathcal{O}(\ep^{1/2})$, in order to obtain the KPZ equation~\cite{berGiaco}.

\subsection{The growth process.}
\label{sec:growth}

Of special interest in non-equilibrium physics is the \emph{growth process} associated to TASEP. More precisely, let 
\[
X^{-1}_t(u) = \min \bigl\{k \in \Z : X_t(k) \leq u\bigr\}
\]
denote the label of the rightmost particle which sits to the left of, or at, $u$ at time $t$. 
The \emph{TASEP  height function} associated to $X_t$ is given for $z\in\Z$ by 
\begin{equation}\label{defofh}
 h_t(z) = -2 \left(X_t^{-1}(z-1) - X_0^{-1}(-1) \right) - z,
\end{equation}
which fixes $h_0(0)=0$. The height function is a random walk path $h_t(z+1) = h_t(z) +\hat{\eta}_t(z)$ with $\hat{\eta}_t(z)=1$ if there is a particle at $z$ at time $t$ and $-1$ if there is no particle at $z$ at time $t$.  We can also easily extend the height function to a continuous function of $x\in \R$ by linearly interpolating between the integer points. 

\begin{exercise} Show that the dynamics of $h_t$ is that local max's become local min's at rate $1$; i.e. if \break$h_t(z) = h_t(z\pm 1) +1$ then $h_t(z)\mapsto h_t(z)-2$ at rate $1$, the rest of the height function remaining unchanged (see the figure below).   What happens if we consider ASEP?
\end{exercise}

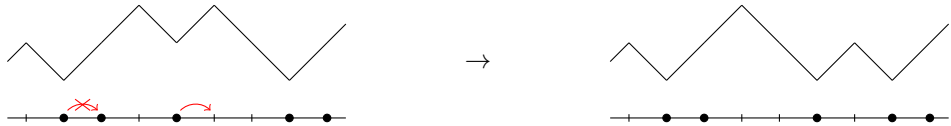
\begin{figure}
\centering
\begin{minipage}{.4\textwidth}
\centering
\begin{tikzpicture}
	\draw (0.25, 0) -- (4.75, 0);
	\foreach \x in {1,...,9} {
		\draw (0.5 * \x, -0.05) -- (0.5 * \x, 0.05);
	}
	\draw[fill=black] (1,0) circle (0.05);
	\draw[fill=black] (1.5,0) circle (0.05);
	\draw[fill=black] (2.5,0) circle (0.05);
	\draw[fill=black] (4,0) circle (0.05);
	\draw[fill=black] (4.5,0) circle (0.05);
	
	\draw (0.25, 0.75) -- (0.5, 1);
	\draw (0.5, 1) -- (1, 0.5);
	\draw (1, 0.5) -- (1.5, 1);
	\draw (1.5, 1) -- (2, 1.5);
	\draw (2, 1.5) -- (2.5, 1);
	\draw (2.5, 1) -- (3, 1.5);
	\draw (3, 1.5) -- (3.5, 1);
	\draw (3.5, 1) -- (4, 0.5);
	\draw (4, 0.5) -- (4.5, 1);
	\draw (4.5, 1) -- (4.75, 1.25);
	
	\draw[->, draw=red] (1.05, 0.1) to[bend left=45] (1.45, 0.1);
	\draw[draw=red] (1.15, 0.25) to (1.35, 0.1);
	\draw[draw=red] (1.15, 0.1) to (1.35, 0.25);
	\draw[->, draw=red] (2.55, 0.1) to[bend left=45] (2.95, 0.1);
\end{tikzpicture}
\end{minipage}
\begin{minipage}{.1\textwidth}
\centering
\begin{tikzpicture}
	\node (0,0) {$\rightarrow$};
\end{tikzpicture}
\end{minipage}
\begin{minipage}{.4\textwidth}
\centering
\begin{tikzpicture}
	\draw (0.25, 0) -- (4.75, 0);
	\foreach \x in {1,...,9} {
		\draw (0.5 * \x, -0.05) -- (0.5 * \x, 0.05);
	}
	\draw[fill=black] (1,0) circle (0.05);
	\draw[fill=black] (1.5,0) circle (0.05);
	\draw[fill=black] (3,0) circle (0.05);
	\draw[fill=black] (4,0) circle (0.05);
	\draw[fill=black] (4.5,0) circle (0.05);
	
	\draw (0.25, 0.75) -- (0.5, 1);
	\draw (0.5, 1) -- (1, 0.5);
	\draw (1, 0.5) -- (1.5, 1);
	\draw (1.5, 1) -- (2, 1.5);
	\draw (2, 1.5) -- (2.5, 1);
	\draw (2.5, 1) -- (3, 0.5);
	\draw (3, 0.5) -- (3.5, 1);
	\draw (3.5, 1) -- (4, 0.5);
	\draw (4, 0.5) -- (4.5, 1);
	\draw (4.5, 1) -- (4.75, 1.25);
\end{tikzpicture}
\end{minipage}
\caption{Evolution of TASEP and its height function.}
\end{figure}

Two standard examples of initial data for TASEP are the \emph{step initial data} (when $X_0(k) = -k$ for $k \geq 1$) and \emph{$d$-periodic initial data} (when $X_0(k) = -d(k-1)$ for $k \in \Z$) with $d \geq 2$. Analysis of TASEP with one of this initial data is much easier than in the general case. In particular the results presented in Sections~\ref{sec:exact} and \ref{sec:123} below were known from~\cite{borFerPrahSasam,bfp,ferrariMatr} and served as a starting point for our work.

\section{Distribution function of TASEP}
\label{sec:distr}

If there are a finite number of particles, we can alternatively denote their positions 
\[
\vec x \in \Omega_N = \bigl\{x_N < \cdots < x_1 \bigr\} \subset \Z^N,
\] 
where $\Omega_N$ is called  the \emph{Weyl chamber}.
The transition probabilities for TASEP with a finite number of particles was first obtained in~\cite{MR1468391} using \emph{(coordinate) Bethe ansatz}. 

\begin{proposition}[Sch\"{u}tz's formula]\label{prop:Schuetz}
The transition probability for $2 \le N <\infty$ TASEP particles has a determinantal form
\begin{equation}\label{eq:Green}
	\pp \bigl(X_t = \vec x\, |\, X_0 = \vec y\bigr)=\det\big[F_{i - j}(x_{N+1 - i}-y_{N+1 - j},t)\big]_{1\leq i,j\leq N}
\end{equation}
with $\vec x, \vec y \in \Omega_N$,  and
\begin{equation}\label{eq:defF}
F_{n}(x, t)=\frac{(-1)^n}{2\pi \I} \oint_{\Gamma_{0,1}} dw\,\frac{(1-w)^{-n}}{w^{x-n+1}}e^{t(w-1)},
\end{equation}
where $\Gamma_{0,1}$ is any simple loop oriented anticlockwise which includes $w=0$ and $w=1$.
\end{proposition}

In the rest of this section we provide a proof of this result using Bethe ansatz and in Section~\ref{ss:schutz_check} we show that Sch\"{u}tz's formula can alternatively be easily checked to satisfy the Kolmogorov forward equation. 

\subsection{Proof of Sch\"{u}tz's formula using Bethe ansatz.}
\label{schutzba}

In this section we will prove Proposition~\ref{prop:Schuetz} following the argument of~\cite{MR2824604}. We will consider $N \geq 2$ particles in TASEP and derive the master (Kolmogorov forward) equation for the process $X_t  = \bigl(X_t(1), \cdots, X_t(N)\bigr) \in \Omega_N$, where $\Omega_N$ is the Weyl chamber defined above. For a function $F : \Omega_N \to \R$ we introduce the operator
\[
\big(\CL^{(N)} F\big)(\vec x) = -\sum_{k=1}^N \1{x_k - x_{k+1} > 1} \bigl(\nablam_{\!k} F\bigr)(\vec x),
\]
where $x_{N+1} = -\infty$ and $\nablam_{\!k}$ is the discrete derivative
\begin{equation}\label{eq:L}
 	\nablam f(z) = f(z) - f(z-1), \qquad f : \Z \to \R,
\end{equation}
acting on the $k$-th argument of $F$. One can see that this is the infinitesimal generator of TASEP in the variables $X_t$. Thus, if 
\[
	P^{(N)}_t(\vec y, \vec x) = \pp \bigl(X_t = \vec x\, |\, X_0 = \vec y\bigr) 
\]
is the transition probability of $N$ particles of TASEP from $\vec y \in \Omega_N$ to $\vec x \in \Omega_N$, then \emph{the master equation} (=\emph{Kolmogorov forward equation}) is
\begin{equation}\label{eq:forward}
	\frac{d}{dt} P^{(N)}_t(\vec y, \cdot) = \CL^{(N)} P^{(N)}_t(\vec y, \cdot), \qquad P^{(N)}_0(\vec y, \cdot) = \delta_{\vec y, \cdot}.
\end{equation}

The idea of~\cite{Bethe} was to rewrite \eqref{eq:forward} as a differential equation with constant coefficients and boundary conditions, i.e. if $u^{(N)}_t : \Z^N \to \R$ solves
\begin{equation}\label{eq:master}
	\frac{d}{dt} u^{(N)}_t = -\sum_{k=1}^N \nablam_{\!k} u^{(N)}_t, \qquad u^{(N)}_0(\vec x) = \delta_{\vec y, \vec x},
\end{equation}
with the \emph{boundary conditions}
\begin{equation}\label{eq:boundary}
	 \nablam_{\!k} u^{(N)}_t (\vec x) = 0, \quad \text{when} \quad x_{k} = x_{k+1} + 1,
\end{equation}
then for $\vec x, \vec y \in \Omega_N$ one has 
\begin{equation}\label{eq:P_u}
	P^{(N)}_t(\vec y, \vec x) = u^{(N)}_t(\vec x).
\end{equation}

\begin{exercise}
 Prove this  by induction on $N \geq 1$.
\end{exercise}

\noindent The strategy is now to find a general solution to the master equation \eqref{eq:master} and then  a particular one which satisfies the boundary and initial conditions.  The method is known as \emph{(coordinate) Bethe ansatz}.

\subsubsection*{Solution to the master equation.}

For a fixed $\vec y \in \Z^N$, we are going to find a solution to the equation \eqref{eq:master}. For this we will consider indistinguishable particles, so that the state space $\{x_1, \cdots, x_N\} \subset \Z$ of the system is given by
\[
	\sum_{\sigma \in \S_N} u^{(N)}_t(\vec x_\sigma),
\]
where $\S_N$ is the symmetric group and $\vec x_\sigma = \bigl(x_{\sigma(1)}, \cdots, x_{\sigma(N)}\bigr)$. With this in mind we define the generating function 
\[
 \phi^{(N)}_t(\vec w) = \frac{1}{|\S_N|}\sum_{\vec x \in \Z^N} \sum_{\sigma \in \S_N} \vec w^{\vec x_\sigma} u^{(N)}_t(\vec x_\sigma),
\]
where $\vec w \in \C^{N}$, $\vec w^{\vec x} = z_1^{x_1} \cdots z_N^{x_N}$ and $|\S_N| = N!$. Since we would like the identity \eqref{eq:P_u} to hold, it is reasonable to assume that $\bigl|u^{(N)}_t(\vec x)\bigr| \leq \min_{i} \frac{t^{x_i - y_i}}{(x_i - y_i)!}$ which guarantees locally absolute convergence of the sum above and all the following computations. Then \eqref{eq:master} yields
\begin{align*}
 	\frac{d}{dt} \phi^{(N)}_t(\vec w) &= \frac{1}{|\S_N|} \sum_{\vec x \in \Z^N} \sum_{\sigma \in \S_N} \vec w^{\vec x_\sigma} \frac{d}{dt} u^{(N)}_t(\vec x_\sigma) \\
	&=- \frac{1}{|\S_N|} \sum_{\vec x \in \Z^N} \sum_{\sigma \in \S_N} \vec w^{\vec x_\sigma} \sum_{k=1}^N \nablam_{\!k} u^{(N)}_t(\vec x_\sigma) \\
	&= -\frac{1}{|\S_N|} \sum_{k=1}^N \sum_{\vec x \in \Z^N} \sum_{\sigma \in \S_N} \vec w^{\vec x_\sigma} \nablam_{\!k} u^{(N)}_t(\vec x_\sigma) \\
	&= \frac{1}{|\S_N|} \sum_{\vec x \in \Z^N} \sum_{\sigma \in \S_N} \vec w^{\vec x_\sigma} u^{(N)}_t(\vec x_\sigma) \sum_{k=1}^N \bigl(w_{\sigma(k)} - 1\bigr) \\
	&= \phi^{(N)}_t(\vec w) \sum_{k=1}^N \eps(w_k),
\end{align*}
where $\eps(w) = w-1$ for $w \in \C$. From the last identity we conclude that
\[
 	\phi^{(N)}_t(\vec w) = C(\vec w) \prod_{k=1}^N e^{\eps(w_k) t},
\]
for a function $C : \C^N \to \C$ which is independent of $t$, but can depend on $\vec y$. Then  Cauchy's integral theorem gives a solution to the master equation
\begin{align}\label{eq:gen_sol}
 	u^{(N)}_t(\vec x) = \frac{1}{(2\pi \I)^N} \sum_{\sigma \in \S_N} \oint_{\Gamma_0} d \vec w\, \frac{\phi^{(N)}_t(\vec w)}{\vec w_\sigma^{\vec x+1}} = \frac{1}{(2\pi \I)^N} \sum_{\sigma \in \S_N} \oint_{\Gamma_0} d\vec w\, C(\vec w) \prod_{k=1}^N \frac{e^{\eps(w_k) t}}{w_{\sigma(k)}^{x_k +1}},
\end{align}
where $\vec x + 1 = \bigl(x_1 + 1, \cdots, x_N+1\bigr)$ and $\Gamma_0$ is a contour in $\C^N$ around the origin. Our next goal it to find $C$ and $\Gamma_0$ such that this solution satisfied the initial and boundary conditions for \eqref{eq:master}. 

\subsubsection*{Satisfying the boundary conditions.}

We are going to find functions $C$ and a contour $\Gamma_0$ such that the solution \eqref{eq:gen_sol} satisfies the boundary conditions \eqref{eq:boundary}. We will look for a solution in a more general form than \eqref{eq:gen_sol}. More precisely, we will consider functions $C_\sigma(z)$ depending on $\sigma \in \S_N$, which gives us the \emph{Bethe ansatz solution}
\begin{equation}\label{eq:Bethe_u}
 	u^{(N)}_t(\vec x) = \frac{1}{(2\pi \I)^N} \sum_{\sigma \in \S_N} \oint_{\Gamma_0} d\vec w\, C_\sigma(\vec w) \prod_{k=1}^N \frac{e^{\eps(w_k) t}}{z_{\sigma(k)}^{x_k +1}}.
\end{equation}
In the case $x_{k} = x_{k+1} + 1$, the boundary condition \eqref{eq:boundary} yields
\begin{align*}
 	 \nablam_{\!k} u^{(N)}_t (\vec x) &=- \frac{1}{(2\pi \I)^N} \sum_{\sigma \in \S_N} \oint_{\Gamma_0} d \vec w \prod_{i \neq k, k+1} \frac{C_\sigma(\vec w)}{w_{\sigma(i)}^{x_i+1}} \frac{1 - w_{\sigma(k)}^{-1}}{w_{\sigma(k)}^{x_k} w_{\sigma(k+1)}^{x_{k+1}+1}} \prod_{i = 1}^N e^{\eps(w_i) t}\\
	 &= -\frac{1}{(2\pi \I)^N} \sum_{\sigma \in \S_N} \oint_{\Gamma_0} d \vec w \prod_{i \neq k, k+1} \frac{C_\sigma(\vec w)}{w_{\sigma(i)}^{x_i+1}} \frac{f(w_{\sigma(k)})}{(w_{\sigma(k)} w_{\sigma(k+1)})^{x_k}} \prod_{i = 1}^N e^{\eps(w_i) t} = 0.
\end{align*}
In particular, this identity holds if for $f(w) = 1 - w^{-1}$ we have
\[
 	\sum_{\sigma \in \S_N} \frac{C_\sigma(\vec w) f(w_{\sigma(k)})}{(w_{\sigma(k)} w_{\sigma(k+1)})^{x_k}} = 0,
\]
for all $\vec w \in \C^N$. Let $T_k \in \S_N$ be the transposition $(k, k+1)$, i.e. it interchanges the elements $k$ and $k+1$. Then the last identity holds if we have
\[
	C_\sigma(\vec x) f(w_{\sigma(k)}) + C_{T_k \sigma}(\vec w) f(w_{\sigma(k + 1)}) = 0.
\]
In particular, one can see that the following functions satisfy this identity
\begin{equation}\label{eq:C_sigma}
 	C_\sigma(\vec w) = \sgn(\sigma) \prod_{i=1}^N f(w_{\sigma(i)})^i\, \psi(\vec w),
\end{equation}
for any function $\psi : \C^N \to \R$. Thus we need to find a specific function $\psi$ so that the initial condition in \eqref{eq:master} is satisfied.

\subsubsection*{Satisfying the initial condition.}

Since the equation \eqref{eq:master} preserves the Weyl chamber, it is sufficient to check the initial condition only for $\vec x, \vec y \in \Omega_N$. Combining \eqref{eq:Bethe_u} with \eqref{eq:C_sigma}, the initial condition at $t = 0$ is given by
\begin{equation}\label{eq:u_init}
 	\frac{1}{(2\pi \I)^N} \sum_{\sigma \in \S_N} \oint_{\Gamma_0} d\vec w \frac{C_\sigma(\vec w)}{w_{\sigma}^{\vec x +1}} = \delta_{\vec y, \vec x}.
\end{equation}
If $\mathrm{id} \in \S_N$ is the identity permutation and $C_{\mathrm{id}}(\vec w) = \vec w^{\vec y}$ then obviously
\[
 	\frac{1}{(2\pi \I)^N} \oint_{\Gamma_0} d \vec w \frac{C_{\mathrm{id}}(\vec w)}{w^{\vec x +1}} = \delta_{\vec y, \vec x}.
\]
For this to hold we need to choose the function $\psi$ in \eqref{eq:C_sigma} to be
\[
	\psi(\vec w) = \prod_{i=1}^N f(w_{i})^{-i} w_i^{y_i}.
\]
Thus, a candidate for the solution is given by
\[
 	u^{(N)}_t(\vec x) = \frac{1}{(2\pi \I)^N} \sum_{\sigma \in \S_N} \sgn(\sigma) \oint_{\Gamma_0} d \vec w\, \prod_{k=1}^N \frac{f(w_{k})^{k - \sigma(k)} e^{\eps(w_k) t}}{w_{\sigma(k)}^{x_k - y_{\sigma(k)} +1}},
\]
which can be written as Sch\"{u}tz's formula \eqref{eq:Green}. It is obvious that the contour $\Gamma_{0}$ should go around $0$ and $1$, since otherwise the determinant in \eqref{eq:Green} will vanish when $\vec x$ and $\vec y$ are far enough. 

In order to complete the we still need to prove that this solution satisfies the initial condition. To this end we notice that for $n \geq 0$ we have
\[
 	F_{-n}(x, 0) =\frac{(-1)^n}{2\pi \I} \oint_{\Gamma_{0}} dw\,\frac{(1-w)^{n}}{w^{x+n+1}},
\]
which in particular implies that $F_{-n}(x, 0) = 0$ for $x < -n$ and $x > 0$, and $F_{0}(x, 0) = \delta_{x, 0}$. In the case $x_N < y_N$, we have $x_N < y_k$ for all $k = 1, \ldots, N-1$, and $x_N - y_{N+1 - j} < 1-j$ since $y \in \Omega_N$. This yields $F_{1 - j}(x_{N}-y_{N+1 - j},0) = 0$ and the determinant in \eqref{eq:Green} vanishes, because the matrix contains a row of zeros. If $x_N \geq y_N$, then we have $x_k > y_N$ for all $k = 1, \ldots, N-1$, and all entries of the first column in the matrix from \eqref{eq:Green} vanish, except the first entry which equals $\delta_{x_N, y_N}$. Repeating this argument for $x_{N-1}$, $x_{N-2}$ and so on, we obtain that the matrix is upper-triangular with delta-functions at the diagonal, which gives us the claim.

\begin{remark}
 	Similar computations lead to the distribution function of ASEP~\cite{MR2824604}. Unfortunately, this distribution function doesn't have the determinantal form as \eqref{eq:Green} which makes its analysis significantly more complicated.
\end{remark}

\subsection{Direct check of Sch\"{u}tz's formula.}
\label{ss:schutz_check}

We will show that the determinant in \eqref{eq:Green}  satisfies the master equation \eqref{eq:master} with the boundary conditions \eqref{eq:boundary}, providing an alternate proof to the one in Section \ref{schutzba}. To this end we will use only the following properties of the functions $F_{n}$, which can be easily proved,
\begin{equation}\label{eq:F_props}
\partial_t F_{n}(x,t) = -\nablam F_{n}(x,t), \qquad F_{n}(x,t) = -\nablap F_{n+1}(x,t),
\end{equation}
where $\nablam$ has been defined in \eqref{eq:L} and $\nablap f(x) = f(x+1) - f(x)$. Furthermore, it will be convenient to define the vectors
\begin{equation}\label{eq:propsF}
H_i(x, t) = \bigl[F_{i - 1}(x-y_{N},t), \cdots, F_{i - N}(x-y_{1},t) \bigr]'.
\end{equation}
Then, denoting by $u_t^{(N)}(\vec x)$ the right-hand side of \eqref{eq:Green}, we can write
\begin{align*}
\partial_t u_t^{(N)}(\vec x) &= \sum_{k = 1}^N \det \bigl[\cdots, \partial_t H_k(x_{N+1 - k}, t), \cdots\bigr]\\
&= -\sum_{k = 1}^N \det \bigl[\cdots, \nablam H_k(x_{N+1 - k}, t), \cdots\bigr]\\
&= -\sum_{k=1}^N \nablam_{\!k} \det \bigl[F_{i - j}(x_{N+1 - i}-y_{N+1 - j},t)\bigr]_{1\leq i,j\leq N},
\end{align*}
where the operators in the first and second sums are applied only to the $k$-th column, and where we made use of the first identity in \eqref{eq:propsF} and multi-linearity of determinants. Here, $\nablam_{\!k}$ is as before the operator $\nablam$ acting on $x_k$. 

Now, we will check the boundary conditions \eqref{eq:boundary}. If $x_{k} = x_{k+1} + 1$, then using again multi-linearity of determinants and the second identity in \eqref{eq:propsF} we obtain
\begin{align*}
\nablam_{\!k} \det \bigl[F_{i - j}(x_{N+1 - i}\; -\; &y_{N+1 - j},t)\bigr] \\
&= \det \bigl[\cdots, \nablam H_{N+1 - k}(x_{k}, t), H_{N - k}(x_{k+1}), \cdots\bigr]\\
&=  \det \bigl[\cdots, \nablap H_{N+1 - k}(x_{k} - 1, t), H_{N - k}(x_{k+1}), \cdots\bigr]\\
&=  \det \bigl[\cdots, \nablap H_{N+1 - k}(x_{k + 1}, t), H_{N - k}(x_{k+1}), \cdots\bigr]\\
&=  \det \bigl[\cdots, -H_{N - k}(x_{k + 1}, t), H_{N - k}(x_{k+1}), \cdots\bigr].
\end{align*}
The latter determinant vanishes, because the matrix has two equal columns. A proof of the initial condition was provided at the end of the previous section.

\section{Determinantal point processes}

In this section we provide some results on determinantal point processes, which can be found e.g. in~\cite{Bor11,borodinRains,johansson}. These processes were studied first in~\cite{Macchi1975} as `fermion' processes and the name 'determinantal' was introduced in~\cite{borOlsh00}.

\begin{definition}\label{def:DPP}
Let $\Xf$ be a discrete space and let $\mu$ be a measure  on $\Xf$. A \emph{determinantal point process} on the space $\Xf$ with {\rm correlation kernel} $\CK:\Xf\times\Xf\to \C$ is a signed\footnote{ In our analysis of TASEP we will be using only a counting measure $\mu$ assigning a unit mass to each element of $\Xf$. However, a determinantal point process can be defined in full generality on a locally compact Polish space with a Radon measure (see~\cite{BHKPV}). Moreover, in contrast to the usual definition we define the measure $\mathcal{W}$ to be signed rather than a probability measure. This fact will be crucial in Section~\ref{eq:random_walks} below, and we should also note that all the properties of determinantal point processes which we will use don't require $\mathcal{W}$ to be positive.
} measure $\mathcal{W}$ on $2^\Xf$ (the power set of $\Xf$), integrating to $1$ and such that for any points $x_1, \cdots, x_n \in \Xf$ one has the identity
\begin{equation}\label{eq:DPP}
	\sum_{\substack{Y \subset \Xf: \\ \{x_1,\ldots,x_n\}\subset Y}} \hspace{-0.5cm}\mathcal{W}(Y) = \det \bigl[\CK(x_i,x_j)\bigr]_{1 \leq i,j \leq n}\, \prod_{k=1}^n \mu(x_k),
\end{equation}
where the sum runs over finite subsets of $\Xf$.
\end{definition}

\noindent The determinants on the right-hand side are called \emph{$n$-point correlation functions} or \emph{joint intensities} and denoted by
\begin{equation}\label{eq:corr}
	\varrho^{(n)}(x_1,\ldots,x_n) = \det \bigl[\CK(x_i,x_j)\bigr]_{1 \leq i,j \leq n}.
\end{equation}
One can easily see that these functions have the following properties: they are symmetric under permutations of arguments and vanish if $x_i = x_j$ for $i \neq j$.

\begin{exercise}  
In the case that $\mathcal{W}$ is a positive measure,  show that if $K$ is the kernel of the orthogonal projection onto a subset of dimension $n$, then the number of points in $\Xf$ is almost surely equal to $n$.
\end{exercise}
 
Usually, it is non-trivial to show that a process is determinantal. Below we provide several examples of determinantal point processes
(these ones are not signed).

\begin{example}[Non-intersecting random walks]\label{ex:RWs}
Let $X_i(t)$, $1 \leq i \leq n$, be independent time-homogeneous Markov chains on $\Z$ with one step transition probabilities $p_t(x,y)$ satisfying the identity \break$p_t(x,x-1) + p_t(x,x+1) = 1$ (i.e. every time each random walk makes one unit step either to the left or to the right). Let furthermore $X_i(t)$ be reversible with respect to a probability measure $\pi$ on $\Z$, i.e. $\pi(x) p_t(x,y) = \pi(y) p_t(y, x)$ for all $x, y \in \Z$ and $t \in \N$. Then, conditioned on the events that the values of the random walks at times $0$ and $2t$ are fixed, i.e. $X_i(0) = X_i(2t) = x_i$ for all $1 \leq i \leq n$ where each $x_i$ is even, and no two of them intersect on the time interval $[0, 2t]$, the configuration of mid-positions $\{X_i(t) : 1 \leq i \leq n\}$ is a determinantal point process on $\Z$ with respect to the measure $\pi$, i.e.
\begin{equation}\label{eq:RWs}
\pp\bigl[ X_i(t) = z_i, 1 \leq i \leq n \bigr] = \det \bigl[\CK(z_i,z_j)\bigr]_{1 \leq i,j \leq n} \prod_{k=1}^n \pi(z_k),
\end{equation}
where the probability is conditioned by the described event (assuming of course that its probability is non-zero). Here, the correlation kernel $\CK$ is given by
\begin{equation}\label{eq:RWsKernel}
\CK(u, v) = \sum_{i=1}^n \psi_i(u) \phi_i(v),
\end{equation}
where the functions $\psi_i$ and $\phi_i$ are defined by
\[
\psi_i(u) = \sum_{k=1}^n \left(A^{-\frac{1}{2}}\right)_{i, k} \frac{p_t(x_k, u)}{\pi(u)}, \qquad \phi_i(v) = \sum_{k=1}^n \left(A^{-\frac{1}{2}}\right)_{i, k} \frac{p_t(x_k, v)}{\pi(v)},
\]
with the matrix $A$ having the entries $A_{i, k} = \frac{p_{2t}(x_i, x_k)}{\pi(x_k)}$. Invertibility of the matrix $A$ follows from the fact that the probability of the condition is non-zero and Karlin-McGregor formula (see Exercise~\ref{ex:KMcG} below). This result is a particular case of a more general result of~\cite{johansson} and it can be obtained from Karlin-McGregor formula similarly to~\cite[Cor.~4.3.3]{BHKPV}.
\end{example}

\begin{exercise}
Prove that the mid-positions $\{X_i(t) : 1 \leq i \leq n\}$ of the random walks defined in the previous example form a determinantal process with the correlation kernel \eqref{eq:RWsKernel}.
\end{exercise}

\begin{exercise}[Karlin-McGregor formula \cite{karlinMcGregor}]\label{ex:KMcG}
Let $X_i$, $1 \leq i \leq n$, be i.i.d. (time-inhomogeneous) Markov chains on $\Z$ with transition probabilities $p_{k, \ell}(s,t)$ satisfying $p_{k, k+1}(t, t+1) + p_{k, k-1}(t, t+1) = 1$ for all $k$ and $t > 0$. Let us fixed initial states $X_i(0) = k_i$ for $k_1 < k_2 < \cdots < k_n$ such that each $k_i$ is even. Then the probability that at time $t$ the Markov chains are at the states $\ell_1 < \ell_2 < \cdots < \ell_n$, and that no two of the chains intersect up to time $t$, equals $\det \bigl[p_{k_i,\ell_j}(0, t)\bigr]_{1 \leq i, j \leq n}$. 

Hint (this idea is due to S.R.S. Varadhan): for a permutation $\sigma \in \S_n$ and $0 \leq s \leq t$, define the process
\begin{equation}
 M_\sigma(s) = \prod_{i=1}^n \pp\bigl( X_{i}(t) = \ell_{\sigma(i)} \big| X_i(s)\bigr),
\end{equation}
which is a martingale with respect to the filtration generated by the Markov chains $X_i$. This implies that the process $M = \sum_{\sigma \in \S_n} \sgn(\sigma) M_\sigma$ is also a martingale. Obtain the Karlin-McGregor formula by applying the optional stopping theorem to $M$ for a suitable stopping time.
\end{exercise}

\begin{example}[Gaussian unitary ensemble]
The most famous example of determinantal point processes is the \emph{Gaussian unitary ensemble} (GUE) introduced by Wigner. Let us define the $n\times n$ matrix $A$ to have i.i.d. standard complex Gaussian entries and let $H = \frac{1}{\sqrt 2} (A + A^*)$. Then the eigenvalues $\lambda_1 > \lambda_2 > \cdots > \lambda_n$ of $H$ form a determinantal point process on $\R$ with the correlation kernel
\[
 \CK(x,y) = \sum_{k = 0}^{n-1} H_k(x) H_k(y),
\]
with respect to the Gaussian measure $d \mu(x) = \frac{1}{\sqrt{2 \pi}} e^{-x^2 / 2} dx$, where $H_k$ are Hermite polynomials which are orthonormal on $L^2(\R, \mu)$. A proof of this result can be found in~\cite[Ch. 3]{mehta}.
\end{example}

\begin{example}[Aztec diamond tilings]
The Aztec diamond is a diamond-shaped union of lattice squares (see Figure~\ref{fig:aztec}). Let's now color some squares in gray following the pattern of a chess board and so that all the bottom left squares are colored. It is easy to see that the Aztec diamond can be perfectly covered by domino tilings, which are $2 \times 1$ or $1 \times 2$ rectangles, and the number of tilings growth exponentially in the width of the diamond. Let's draw a tiling uniformly from all possible tilings and let's mark gray left squares of horizontal dominos and gray bottom squares of vertical dominos. This random set is a determinantal point process on the lattice $\Z^2$~\cite{MR2118857}.
\end{example}

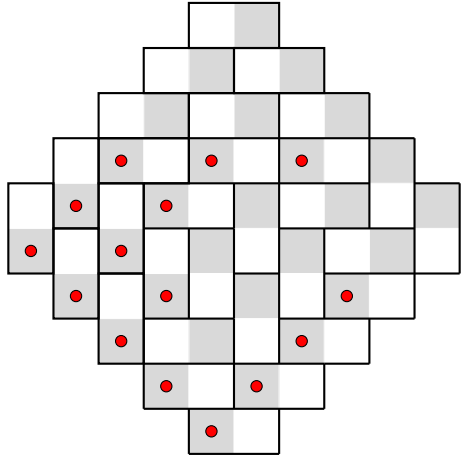
\begin{figure}
\centering
\begin{tikzpicture}[scale=1.5]
	\foreach \x in {0,...,5} {
		\foreach \y in {0,...,4} {
			\filldraw[fill=gray!30!white, draw=white] (-2 + 0.4 * \x + 0.4 * \y, -0.4 + 0.4 * \x - 0.4 * \y) rectangle (-2 + 0.4 * \x + 0.4 * \y + 0.4, 0.4 * \x - 0.4 * \y);
		}
	}

	\draw[thick] (0, 0) -- (-0.8, 0);
	\draw[thick] (-0.8, 0.4) -- (-0.8, 0);
	\draw[thick] (0, 0.4) -- (0, 0);
	\draw[thick] (0, 0.4) -- (-0.8, 0.4);
	\draw[thick] (0, 0) -- (0, -0.8);
	\draw[thick] (-0.4, 0) -- (-0.4, -0.8);
	\draw[thick] (-0.8, 0) -- (-0.8, -0.8);
	\draw[thick] (0, -0.8) -- (-0.8, -0.8);
	\draw[thick] (-0.8, -1.2) -- (-0.8, -0.8);
	\draw[thick] (0, -1.2) -- (0, -0.8);
	\draw[thick] (0, -1.2) -- (-0.8, -1.2);
	\draw[thick] (-0.8, -1.2) -- (-0.8, -1.6);
	\draw[thick] (0, -1.2) -- (0, -1.6);
	\draw[thick] (0, -1.6) -- (-0.8, -1.6);
	\draw[thick] (0.8, -1.2) -- (0.8, -1.6);
	\draw[thick] (0, -1.6) -- (0.8, -1.6);
	\draw[thick] (0, -1.2) -- (0.8, -1.2);
	\draw[thick] (0.4, -1.6) -- (0.4, -2);
	\draw[thick] (-0.4, -1.6) -- (-0.4, -2);
	\draw[thick] (-0.4, -2) -- (0.4, -2);
	\draw[thick] (0.4, -1.2) -- (0.4, 1.2);
	\draw[thick] (0, -0.4) -- (0.4, -0.4);
	\draw[thick] (0, 0.4) -- (0.4, 0.4);
	\draw[thick] (-0.4, 0.4) -- (-0.4, 1.2);
	\draw[thick] (-0.4, 0.8) -- (0.4, 0.8);
	\draw[thick] (-0.4, 1.2) -- (0.4, 1.2);
	\draw[thick] (0.4, -0.8) -- (1.6, -0.8);
	\draw[thick] (0.8, -1.2) -- (1.2, -1.2);
	\draw[thick] (1.2, -1.2) -- (1.2, -0.8);
	\draw[thick] (0.8, -0.8) -- (0.8,0);
	\draw[thick] (0.4, 0) -- (1.6,0);
	\draw[thick] (0.8, -0.4) -- (2,-0.4);
	\draw[thick] (1.6, -0.8) -- (1.6,0.8);
	\draw[thick] (2, -0.4) -- (2,0.4);
	\draw[thick] (1.2, 0) -- (1.2,1.2);
	\draw[thick] (1.6, 0.4) -- (2,0.4);
	\draw[thick] (0.4, 0.4) -- (1.2,0.4);
	\draw[thick] (0.4, 0.8) -- (1.6,0.8);
	\draw[thick] (0.4, 1.2) -- (1.2,1.2);
	\draw[thick] (-0.4, 2) rectangle (0.4,1.6);
	\draw[thick] (0, 1.6) rectangle (0.8,1.2);
	\draw[thick] (0, 1.6) rectangle (-0.8,1.2);
	\draw[thick] (-0.4, 1.2) rectangle (-1.2,0.8);
	\draw[thick] (-0.4, 0.8) rectangle (-1.2,0.4);
	\draw[thick] (-1.2, 0.4) rectangle (-0.8,-0.4);
	\draw[thick] (-1.2, -0.4) rectangle (-0.8,-1.2);
	\draw[thick] (-1.6, 0.8) rectangle (-1.2,0);
	\draw[thick] (-1.6, 0) rectangle (-1.2,-0.8);
	\draw[thick] (-2, 0.4) rectangle (-1.6,-0.4);
	
	\draw[fill=red] (-1.8,-0.2) circle (0.05);
	\draw[fill=red] (-1.4,0.2) circle (0.05);
	\draw[fill=red] (-1,0.6) circle (0.05);
	\draw[fill=red] (-1,-0.2) circle (0.05);
	\draw[fill=red] (-0.2,0.6) circle (0.05);
	\draw[fill=red] (0.6,0.6) circle (0.05);
	\draw[fill=red] (-0.6,0.2) circle (0.05);
	\draw[fill=red] (-1.4,-0.6) circle (0.05);
	\draw[fill=red] (-1,-1) circle (0.05);
	\draw[fill=red] (-0.6,-1.4) circle (0.05);
	\draw[fill=red] (-0.2,-1.8) circle (0.05);
	\draw[fill=red] (-0.6,-0.6) circle (0.05);
	\draw[fill=red] (0.2,-1.4) circle (0.05);
	\draw[fill=red] (0.6,-1) circle (0.05);
	\draw[fill=red] (1,-0.6) circle (0.05);
\end{tikzpicture}
\caption{Aztec diamond tiling.}
\label{fig:aztec}
\end{figure}

\subsection{Probability of an empty region.}

A useful property of determinantal point processes is that the `probability' (recall that the measure in Definition~\ref{def:DPP} is signed) of having an empty region is given by a Fredholm determinant.

\begin{lemma}\label{lem:GapProbability}
Let $\mathcal{W}$ be a determinantal point process on a discrete set $\Xf$ with a measure $\mu$ and with a correlation kernel $\CK$. Then for a Borel set $B \subset \Xf$ one has
\[
	\sum_{X \subset \Xf \setminus B} \hspace{-0.2cm} \mathcal{W}(X) = \det(I - \CK)_{\ell^2(B, \mu)},
\]
where the latter is the Fredholm determinant defined by
\begin{equation}\label{eq:Fredholm}
 \det(I - \CK)_{\ell^2(B, \mu)} = \sum_{n\geq 0}\frac{(-1)^n}{n!} \int_{B^n} \det \bigl[\CK(y_i,y_j)\bigr]_{1 \leq i,j \leq n}\, d\mu(y_1) \cdots d\mu(y_n).
\end{equation}
\end{lemma}

\begin{proof}
Using Definition~\ref{def:DPP} and the correlation functions \eqref{eq:corr} we can write
\begin{align*}
\sum_{X \subset \Xf \setminus B} \hspace{-0.2cm} \mathcal{W}(X) &= \sum_{X \subset \Xf} \mathcal{W}(X)\, \prod_{x \in X}\left((1-\1{B}(x)\right) \\
&= \sum_{n\geq 0} \frac{(-1)^n}{n!} \sum_{X \subset \Xf} \mathcal{W}(X)\, \sum_{\substack{x_1,\ldots,x_n \in X \\ x_i \neq x_j}}\prod_{k=1}^n\1{B}(x_{k})\\
&= \sum_{n\geq 0} \frac{(-1)^n}{n!} \sum_{\substack{y_1,\ldots,y_n \in B \\ y_i \neq y_j}} \sum_{X \subset \Xf} \mathcal{W}(X)\, \sum_{x_1,\ldots,x_n \in X}\prod_{k=1}^n\1{x_{k} = y_k}\\
&= \sum_{n\geq 0} \frac{(-1)^n}{n!} \sum_{\substack{y_1,\ldots,y_n \in B \\ y_i \neq y_j}} \sum_{\substack{X \subset \Xf \\ \{y_1,\ldots,y_n\} \subset X}} \hspace{-0.5cm} \mathcal{W}(X)\\
&=\sum_{n\geq 0}\frac{(-1)^n}{n!}\sum_{\substack{y_1,\ldots,y_n \in B \\ y_i \neq y_j}} \hspace{-0.3cm} \varrho^{(n)}(y_1,\ldots,y_n) \prod_{k=1}^n \mu(y_k)\\
&= \sum_{n\geq 0}\frac{(-1)^n}{n!}\int_{B^n} \det \bigl[\CK(y_i,y_j)\bigr]_{1 \leq i,j \leq n}\, d\mu(y_1) \cdots d\mu(y_n)\\
&= \det(I - \CK)_{\ell^2(B, \mu)},
\end{align*}
which is exactly our claim. Note, that the condition $y_i \neq y_j$ cabe be omitted, because $\varrho^{(n)}$ vanishes on the diagonals.
\end{proof}

\begin{exercise}
 Prove that if $\Xf$ is finite and $\mu$ is the counting measure, then the Fredholm determinant \eqref{eq:Fredholm} coincides with the usual determinant.
\end{exercise}

\subsection{$\mathbf L$-ensembles of signed measures.}

A more restrictive definition of a determinantal process was introduced in~\cite{borodinRains}. In order to simplify our notation, we take the measure $\mu$ in this section to be the counting measure and we will skip it in notation below.

With the notation of Definition~\ref{def:DPP}, let us be given a function $L:\Xf\times\Xf\to\C$. For any finite subset $X = \{x_1, \cdots, x_n\} \subset \Xf$ we define a symmetric minor $L_X = \bigl[L(x_i,x_j)\bigr]_{x_i,x_j \in X}$. Then one can define a (signed) measure on $\Xf$, called the \emph{$L$-ensemble}, by
\begin{equation}\label{eq:LEnsemble}
\mathcal{W}(X)=\frac{\det(L_X)}{\det(1+L)_{\ell^2(\Xf)}},
\end{equation}
for $X \subset \Xf$, if the Fredholm determinant $\det(1+L)_{\ell^2(\Xf)}$ is non-zero (recall the definition \eqref{eq:Fredholm}).

\begin{exercise}
 Check that the measure $\mathcal{W}$ defined in \eqref{eq:LEnsemble} integrates to $1$.
\end{exercise}

The requirement $\det(1+L)_{\ell^2(\Xf)} \neq 0$ guarantees that there exists a unique function\break$(1+L)^{-1} : \Xf \times \Xf \to \C$ such that $(1+L)^{-1} * (1 + L) = 1$, where $*$ is the convolution on $\Xf$ and $1 : \Xf \times \Xf \to \{0,1\}$ is the identity function non-vanishing only on the diagonal. Furthermore, it was proved in~\cite{Macchi1975} that the $L$-ensemble is a determinantal point process: 

\begin{proposition}\label{prop:Macchi}
The measure $\mathcal{W}$ defined in \eqref{eq:LEnsemble} is a determinantal point process with correlation kernel $\CK=L(1+L)^{-1} = 1 - (1+L)^{-1}$.
\end{proposition}

\begin{example}[Non-intersecting random walks]
It is not difficult to see that the distribution of the mid-positions $\{X_i(t) : 1 \leq i \leq n\}$ of the random walks from Example~\ref{ex:RWs} is the $L$-ensemble with the function 
\[
 	L(u, v) = \sum_{i=1}^N p_t(u, x_i) p_t(x_i, v).
\]
The correlation kernel $\CK$ can be computed from Proposition~\ref{prop:Macchi} and it coincides with \eqref{eq:RWs}.
\end{example}

\begin{exercise}
 Perform the computations from the previous example.
\end{exercise}

\subsection{Conditional $\mathbf L$-ensembles.}

An $L$-ensemble can be conditioned by fixing certain values of the determinantal process. More precisely, consider a nonempty subset $\Zf\subset\Xf$ and a given $L$-ensemble on $\Xf$. We define a measure on $2^\Zf$, called \emph{conditional $L$-ensemble}, in the following way:
\begin{equation}\label{eq:LCond}
\mathcal{W}(Y)=\frac{\det(L_{Y\cup\Zf^c})}{\det(1_\Zf+L)},
\end{equation}
for any $Y \subset \Zf$, where $1_\Zf(x,y) = 1$ if and only if $x = y \in \Zf$, and $1_\Zf(x,y) = 0$ otherwise.

\begin{exercise}
 Prove that the measure $\mathcal{W}$ defined in \eqref{eq:LCond} integrates to $1$.
\end{exercise}

\noindent Roughly speaking the definition \eqref{eq:LCond} means that we restrict the $L$-ensemble by the condition that the values $\Zf^c$ are fixed. The following result is a generalisation of Proposition~\ref{prop:Macchi} and its proof be found in~\cite[Prop.~1.2]{borodinRains}:

\begin{proposition}\label{prop:ConditionalLensembles}
The conditional $L$-ensemble is a determinantal point process on $\Zf$ with correlation kernel
\begin{equation}\label{eq:K_LCond}
 \CK=1_\Zf-(1_\Zf+L)^{-1}\big|_{\Zf\times\Zf},
\end{equation}
where $F \big|_{\Zf\times\Zf}$ means restriction of the function $F$ to the set $\Zf\times\Zf$.
\end{proposition}

\section{Biorthogonal representation of the correlation kernel}

The formula \eqref{eq:Green} is not suitable for asymptotic analysis of TASEP, because the size of the matrix goes to $\infty$ as the number of particles $N$ increases. To overcome this problem, the authors of~\cite{borFerPrahSasam} (and in its preliminary version~\cite{sasamoto}) wrote it as a Fredholm determinant, which can be then subject to asymptotic analysis.

In order to state this result, we need to make some definitions. For an integer $M \geq 1$, a fixed vector $\vec a \in\R^M$ and indices $n_1<\ldots<n_M$ we introduce the projections
\begin{equation}\label{eq:defChis}
\chi_{\vec a}(n_j,x)=\1{x> a_j}, \hspace{1cm} \bar\chi_{\vec a}(n_j,x)=\1{x\leq a_j},
\end{equation}
acting on $x \in \Z$, which also regard as multiplication operators acting on $\ell^2 \bigl(\{n_1,\ldots,n_M\}\times\Z\bigr)$.
We will use the same notation if $a$ is a scalar, writing 
\begin{equation}\label{eq:defChisScalar}
\chi_a(x)=1-\bar\chi_a(x)=\1{x>a}.
\end{equation}
Then from~\cite{borFerPrahSasam} we have the following result:

\begin{theorem}\label{thm:BFPS}
Suppose that TASEP starts with particles labeled $X_0(1)>X_0(2)>\ldots >X_0(N)$ and let $1\leq n_1<n_2<\dotsm<n_M\leq N$ and $\vec a \in\R^M$ for some $1\leq M \leq N$. Then for $t>0$ we have
\begin{equation}\label{eq:extKernelProbBFPS}
	\pp \bigl(X_t(n_j) > a_j,~j=1,\ldots,M\bigr)=\det \bigl(I-\bar\chi_{\vec a}K_t\bar\chi_{\vec a}\bigr)_{\ell^2(\{n_1,\ldots,n_M\}\times\Z)}
\end{equation}
where the kernel $K_t$ is given by
\begin{equation}\label{eq:Kt}
K_t(n_i,x_i;n_j,x_j)=-\phi^{n_j - n_i}(x_i,x_j)\1{n_i<n_j}+\sum_{k=1}^{n_j}\Psi^{n_i}_{n_i-k}(x_i)\Phi^{n_j}_{n_j-k}(x_j),
\end{equation}
and where $\phi(x,y)=\1{x>y}$ and
\begin{equation}\label{eq:defPsi}
\Psi^n_k(x)=\frac1{2\pi\I}\oint_{\Gamma_0}dw\,\left(\frac{1-w}{w}\right)^{k} \frac{e^{t(w-1)}}{w^{x - X_0(n-k) + 1}}.
\end{equation}
Here, $\Gamma_0$ is any simple loop, oriented counterclockwise, which includes the pole at $w=0$ but does not include $w=1$.
The functions $\Phi_k^{n}$, $0 \leq k < n$, are defined implicitly by the following two properties:
\begin{enumerate}[label={\normalfont (\arabic{*})}]
\item the biorthogonality relation, for $0 \leq k, \ell < n$,
\[
 \sum_{x\in\Z}\Psi_k^{n}(x)\Phi_\ell^{n}(x)=\1{k=\ell},
\]
\label{ortho}
\item   the spanning property 
\[
 {\mathrm{span}} \bigl\{\Phi^n_k : 0 \leq k < n\bigr\} = {\mathrm{span}}\bigl\{x^k : 0 \leq k < n\bigr\},
\]
which in particular implies that the function $\Phi^n_k$ is a polynomial of degree at most $n-1$.
\end{enumerate} 
\end{theorem}

\begin{remark}
 The problem with this result is that the functions $\Phi^n_k$ are not given explicitly. In special cases of periodic and step initial data exact integral expressions of these functions had been found in \cite{borFerPrahSasam}, \cite{ferrariMatr} and \cite{bfp} (the latter is for the discrete time TASEP, which can be easily adapted for continuous time). More precisely, for the step initial data $X_0(i)=-i$, $i\geq1$, we have
\[
	\Phi^n_k(x) = \frac{1}{2 \pi i} \oint_{\Gamma_0} dv\, \frac{(1-v)^{x+n}}{v^{k+1}} e^{t v},
\]
and in the case of periodic initial data $X_0(i)=-d i$, $i\geq1$, with $d \geq 2$,
\[
	\Phi^n_k(x) = \frac{1}{2 \pi \I} \oint_{\Gamma_{0}} dv\, \frac{(1 - d v) (2(1-v))^{x + dn - 1}}{v (2^d (1-v)^{d-1} v)^{k}} e^{t v}.
\]
 The key new result in~\cite{KPZ} is an  expression for the functions $\Phi^n_k$, and therefore the kernel $K_t$, for \emph{arbitrary} initial data.
\end{remark}

\begin{remark}
 The functions $F$ from \eqref{eq:defF} and $\Psi$ from \eqref{eq:defPsi} are obviously related by the identity
\begin{equation}\label{eq:FandPhi}
\Psi^N_{k}(x) =  (-1)^k F_{-k} (x - y_{N - k }, t),
\end{equation}
for $0 \leq k \leq N$, so that all properties of $F$ can be translated to $\Psi$. Moreover, one can see that if $n \leq 0$, then the function inside the integral in \eqref{eq:defF} has the only pole at $w=0$, which yields 
\[
F_{n+1}(x,t)=-\sum_{y < x} F_n(y,t).
\]
Writing this relation in terms of the functions $\Psi^N_{k}$, we get
\begin{equation}\label{eq:PsiRecursion}
\Psi^{N}_{N - k}(x) = \sum_{y < x} \Psi^{N+1}_{N +1 - k}(x).
\end{equation}
\end{remark}

In the next section we provide a proof of this result following~\cite{borFerPrahSasam}. The main idea is to rewrite the problem in terms of non-intersecting random walks (or \emph{vicious random walks}) whose configurations form a \emph{Gelfand-Tsetlin pattern}\footnote{This property relates TASEP to random matrices, see~\cite{ferrariMatr}.}. The distribution of these random walks forms a determinantal point process whose correlation kernel is \eqref{eq:Kt}.

\subsection{Non-intersecting random walks.}
\label{eq:random_walks}

Our aim in this section is to rewrite Sch\"{u}tz's formula \eqref{eq:Green} in a form involving transition probabilities of non-intersecting random walks.

We start with rewriting the transition probabilities \eqref{eq:Green} in the following way:

\begin{proposition}\label{p:transitions}
For $\vec x, \vec y \in \Omega_N$, one has the following identity:
\begin{equation}\label{eq:transitions}
	\pp \big(X_t = \vec x\, |\, X_0 = \vec y\big) = \sum_{\substack{\bz \in \GT_N: \\ z_1^{n} = x_n}} \det \Big[\Psi^N_{N-j}\big(z_{i}^N\big)\Big]_{1 \leq i,j \leq N},
\end{equation}
where the functions $\Psi$ are defined in \eqref{eq:defPsi}, where the sum runs over the domain $\GT_N$ of triangular arrays given by a Gelfand-Tsetlin pattern 
\[
	\GT_N = \Big\{\bz = (z_i^n)_{n, i} \colon z_i^n \in \Z,\; 1 \leq i \leq n < N,\; z_i^{n+1} < z_i^{n} \leq z_{i+1}^{n+1} \Big\},
\]
with fixed values $z_1^{n} = x_n$ for all $n=1, \cdots, N$ (see Figure~\ref{fig:GT} for a graphical representation of $\GT_4$).
\end{proposition}
{
\def\Num{4}
\def\StartX{4}
\def\StartY{1}
\begin{figure}
\centering
\begin{tikzpicture}
	\foreach \x [evaluate=\x as \evalX using int(\Num + 1 - \x)] in {1,...,\Num} {
    		\foreach \y [evaluate=\y as \evalY using int(\Num + 1 - \y)] in {\x,...,\Num} {
			\ifthenelse{\y=\Num}{} {
				\draw (\StartX - \y + 1 + 2 * \x,\StartY + \y - 1) edge (\StartX - \y + 1 + 2 * \x - 1,\StartY + \y - 1 + 1);
				\node at (\StartX - \y + 1 + 2 * \x - 1.2, \StartY + \y - 1 + 0.5) {\begin{rotate}{-45}$<$\end{rotate}};

				\ifthenelse{\x=\Num}{} {
					\draw  (\StartX - \y + 1 + 2 * \x, \StartY + \y - 1) edge (\StartX - \y + 1 + 2 * \x + 1, \StartY + \y - 1 + 1);
					\node at (\StartX - \y + 1 + 2 * \x, \StartY + \y - 1 + 0.4) {\begin{rotate}{45}$\leq$\end{rotate}};
				}
			}
			\ifthenelse{\x=1} {
				\fill[black] (\StartX - \y + 1 + 2 * \x,\y) circle (3pt) node [left = 0.2cm] {$x_\y = z^\y_\x$};
			}{
				\fill[black] (\StartX - \y + 1 + 2 * \x,\y) circle (3pt) node [left = 0.2cm] {$z^\y_\x$};
			}
		}
	}
\end{tikzpicture}
\caption{The Gelfand-Tsetlin pattern $\GT_4$ with fixed values $z^n_1 = x_n$.}
\label{fig:GT}
\end{figure}
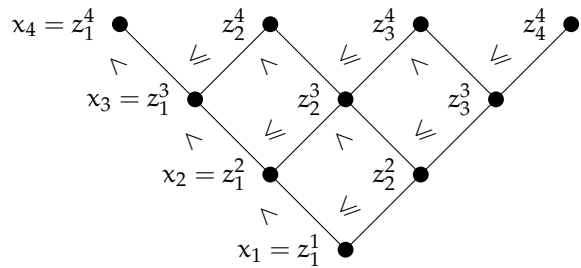
}

\begin{proof}
This decomposition is obtained using only the identity
\begin{equation}\label{eq:F_sum}
F_{n+1}(x,t)=\sum_{y \geq x} F_n(y,t),
\end{equation}
which is the integrated form of the second equality in \eqref{eq:F_props} combined with the fact that the convergence $\lim_{y\to+\infty}F_n(y,t)=0$ holds fast enough. From Sch\"{u}tz's formula \eqref{eq:Green} we have
\begin{equation}\label{eq:P_det}\arraycolsep=5pt\renewcommand\arraystretch{1.2}
\pp \bigl(X_t = \vec x \,|\, X_0 = \vec y\bigr) = 
\det\left[\begin{array}{ccc} F_0(z_1^N-y_N,t)&
\cdots & F_{-N+1}(z_1^N-y_1,t) \\ \vdots & \ddots & \vdots \\
F_{N-1}(z_1^1-y_N,t)&\cdots & F_0(z_1^1-y_1,t)\end{array}\right],
\end{equation}
where we renamed the variables $z^n_1 = x_n$. Applying the property \eqref{eq:F_sum} twice to each entry of the last row we can rewrite it as
\begin{align}\arraycolsep=5pt\renewcommand\arraystretch{1.2}
	\Big[\begin{array}{ccc} F_{N-1}(z_1^1-y_N,t)&\;\cdots\; & F_0(z_1^1-y_1,t)\end{array}\Big] &= \sum_{z_2^2 \geq z_1^1} \Big[\begin{array}{ccc} F_{N-2}(z_2^2-y_N,t)&\;\cdots\; & F_{-1}(z_2^2-y_1,t)\end{array}\Big] \nonumber\\
	&= \sum_{z_2^2 \geq z_1^1}\sum_{z_3^3\geq z_2^2} \Big[\begin{array}{ccc} F_{N-3}(z_3^3-y_N,t)&\;\cdots\; & F_{-2}(z_3^3-y_1,t)\end{array}\Big].\label{eq:last_row}
\end{align}
Applying furthermore the identity \eqref{eq:F_sum} to the penultimate row in \eqref{eq:P_det} we obtain
\[
	\sum_{z_2^3 \geq z_1^2} \Big[\begin{array}{ccc} F_{N-3}(z_2^3-y_N,t)&\;\cdots\; & F_{-2}(z_2^3-y_1,t)\end{array}\Big].
\]
Combining these two identities with multilinearity of determinant, the right-hand side of \eqref{eq:P_det} equals
\[\arraycolsep=5pt\renewcommand\arraystretch{1.2}
\sum_{z_2^2 \geq z_1^1}\sum_{z_3^3\geq z_2^2}\sum_{z_2^3
\geq z_1^2} \det\left[\begin{array}{ccc} F_0(z_1^N-y_N,t)& \cdots &
F_{-N+1}(z_1^N-y_1,t) \\ \vdots & \ddots & \vdots \\
F_{N-3}(z_1^3-y_N,t)&\cdots & F_{-2}(z_1^3-y_1,t) \\
F_{N-3}(z_2^3-y_N,t)&\cdots & F_{-2}(z_2^3-y_1,t) \\
F_{N-3}(z_3^3-y_N,t)&\cdots & F_{-2}(z_3^3-y_1,t) \end{array}\right].
\]
The determinant is antisymmetric in the variables $z_2^3$ and $z_3^3$ (i.e. it changes sign if we swap $z_2^3$ and $z_3^3$), therefore the contribution of the symmetric part of the summation domain $\bigl\{z_3^3 \geq z_2^2\bigr\} \cap \bigl\{z_2^3 \geq z_1^2\bigr\}$ is zero. Since the symmetric part of this domain is $\big\{z_3^3 \geq z_2^2\big\} \cap \big\{z_2^3 \geq z_2^2\big\}$, we are left with the sum over $\big\{z_3^3 \geq z_2^2\big\} \cap \big\{z_2^3 \in[z_1^2,z_2^2)\big\}$. We iterate the same procedure for $k=3,\ldots,N-1$, applying \eqref{eq:F_sum} to the last $k$ rows and removing the sums over symmetric domains, and we get the formula
\[
\pp \bigl(X_t = \vec x \,|\, X_0 = \vec y\bigr) = \sum_{\substack{\bz \in \GT_N: \\ z_1^{n} = x_n}} \det \Big[F_{1-j} \big(z_{i}^N - y_{N + 1 -j},t\big)\Big]_{1 \leq i,j \leq N}.
\]
Now, we can use the identity \eqref{eq:FandPhi} to get
\begin{align*}
 \det \Big[F_{1-j} \big(z_{i}^N - y_{N + 1 -j},t\big)\Big]_{1 \leq i,j \leq N} &= \det \Big[(-1)^{j-1} \Psi^N_{j-1} \big(z_{i}^N\big)\Big]_{1 \leq i,j \leq N}\\
 &= (-1)^{(1 + 2 + \cdots + N) - N} \det \Big[\Psi^N_{j-1}\big(z_{i}^N\big)\Big]_{1 \leq i,j \leq N},
\end{align*}
and we change the order of the columns of the matrix inside the determinant
\[
 \det \Big[\Psi^N_{j-1}\big(z_{i}^N\big)\Big]_{1 \leq i,j \leq N} = (-1)^{\lfloor N/2 \rfloor} \det \Big[\Psi^N_{N-j}\big(z_{i}^N\big)\Big]_{1 \leq i,j \leq N}.
\]
It is not difficult to see that $(1 + 2 + \cdots + N) - N + \lfloor N/2 \rfloor$ is an even integer so that the power of $-1$ equals $1$. Hence, combining these identities we get
\[
  \det \Big[F_{1-j}\big(z_{i}^N - y_{N + 1 -j},t\big)\Big]_{1 \leq i,j \leq N} = \det \Big[\Psi^N_{N-j}\big(z_{i}^N\big)\Big]_{1 \leq i,j \leq N},
\]
which gives exactly our claim \eqref{eq:transitions}.
\end{proof}

The weight of a configuration $\bz \in \GT_N$ in \eqref{eq:transitions} can be written as
\begin{equation}\label{eq:pointMeasure}
\mathcal{W}_N(\bz) = \left(\prod_{n=1}^N \det \Big[\phi \big(z_i^{n-1}, z_j^{n}\big)\Big]_{1\leq i,j \leq n}\right)\, \det \Big[\Psi^N_{N-j}\big(z_{i}^N\big)\Big]_{1 \leq i,j \leq N},
\end{equation}
where $\phi(x,y)=\1{x>y}$ and where we have introduced new values $z_n^{n-1}= +\infty$ (so that $\phi(z_n^{n-1},y)=\nobreak1$ for all $y \in \Z$). The determinant $\det \big[\phi(z_i^{n-1}, z_j^{n})\big]_{1\leq i,j \leq n}$ is the indicator function for the inequalities in $\GT_N$ between the levels $n-1$ and $n$ to hold. More precisely, if we define the space of integer-valued triangular arrays
\begin{equation}\label{eq:LambdaDef}
	\Lambda_N = \bigl\{\bz = (z_i^n)_{n, i} \colon z_i^n \in \Z,\; 1 \leq i \leq n \leq N \bigr\},
\end{equation}
then for $\bz = (z_i^n)_{n, i} \in \Lambda_N$ we have
\begin{equation}\label{eq:detIndicator}
\prod_{n=1}^N \det \big[\phi(z_i^{n-1}, z_j^{n})\big]_{1\leq i,j \leq n} = \1{\bz \in \GT_N}.
\end{equation}

\begin{exercise}
 Prove that the identity \eqref{eq:detIndicator} holds.
\end{exercise}

\noindent Hence, the identities \eqref{eq:transitions} and \eqref{eq:pointMeasure} yield
\begin{equation}\label{eq:prob_as_sum}
	\pp \bigl(X_t = \vec x \,|\, X_0 = \vec y\bigr) = \sum_{\substack{\bz \in \Lambda_N: \\ z_1^{n} = x_n}} \mathcal{W}_N(\bz),
\end{equation}
where the sum runs over the set $\Lambda_N$ of integer-valued triangular arrays with fixed boundary values $z_1^{n}=x_n$ for all $n = 1, \ldots, N$.

The variables \mbox{$\{z_i^n : i=1,\ldots,n\}$} in \eqref{eq:pointMeasure} can be interpreted as the positions of particles labelled by $i=1,\ldots,n$ at time $n$, so that $z_k^k,\ldots,z_k^n$ is the trajectory of particle $k$ with the transition kernel $\phi$ (which can be made a probability kernel by multiplying by a power of $2$). At time $n$ there are $n$ particles at positions $z_1^n,\ldots,z_n^n$, which make geometric jumps to the left at time $n+1$ conditioned by non-intersection (they are called \emph{vicious random walks}). Furthermore, a new $(n+1)$-st particle is added at position $z_{n+1}^{n+1} \geq z_n^n$ at time $n+1$.

The measure $\mathcal{W}_N$ on $\Lambda_N$ is not a probability measure, because the contribution coming from the functions $\Psi$ can give a negative value. We will show that this is a determinantal measure, which in particular means that the probability that the sites $z_{k_i}^{n_i}$ for $i=1,\ldots,M$ are occupied by the respective walks is proportional to
\begin{equation}\label{eq:vicious_det}
	\det \Bigl[\CK_t \big( (n_i, k_i, z_{k_i}^{n_i}), (n_j, k_j, z_{k_j}^{n_j}) \big)\Bigr]_{1 \leq i, j \leq M},
\end{equation}
for a correlation kernel $\CK_t$ which will be obtained below.

\subsection{The correlation kernel of the signed measure.}

We will prove in this section that the measure $\mathcal{W}_N$ defined in \eqref{eq:pointMeasure} on triangular arrays $\Lambda_N$ is a determinantal point process and will find its correlation kernel, so that in particular the property \eqref{eq:vicious_det} holds. Note that we will consider the measure $\mathcal{W}_N$ on the whole space $\Lambda_N$, and only after having found the correlation kernel will we will fix the boundary values as in \eqref{eq:prob_as_sum}.

\subsubsection*{The measure $\mathcal{W}_N$ is a conditional $\mathbf L$-ensemble.} It will be more convenient to write the values of a triangular array as a one-dimensional array. More precisely, if fix some point configuration $(z^n_i)_{n, i} \in \Lambda_N$, then every value $z^n_i = z$ can be identified with the triplet $(n, i, z)$, where $n \in \{1, \cdots, N\}$ and $i \in \{1, \cdots, n\}$. So that the values $(z^n_i)_{n, i}$ can be written as a one-dimensional array parametrized by $(n, i)$, e.g. in the case $N = 3$ we have
\begin{center}
\begin{tikzpicture}[
array/.style={matrix of nodes,nodes={draw, minimum width=10mm, minimum height=5mm},column sep=-\pgflinewidth, row sep=0.5mm, nodes in empty cells,
row 1/.style={nodes={draw=none, fill=none, minimum size=5mm}, font=\scriptsize, color=blue}}]
\centering
\matrix[array] (array) {
$(1,1)$ & $(2,1)$ & $(2,2)$ & $(3,1)$ & $(3,2)$ & $(3,3)$\\
  $z^1_1$ & $z^2_1$ & $z^2_2$ & $z^3_1$ & $z^3_2$ & $z^3_3$\\};
\end{tikzpicture}
\end{center}
\noindent With this idea in mind, the point process which we are going to define has the domain $\Zf$ of all triplets $(n, i, z)$, where $n \in \{1, \cdots, N\}$, $i \in \{1, \cdots, n\}$ and $z \in \Z$. In fact, we need to define a slightly larger domain $\Xf = \{1,2,\ldots,N\} \cup \Zf$, so that the numbers $\{1,2,\ldots,N\}$ will refer to either the values $z^{n-1}_{n}$ or the initial values $y_{n}$ of TASEP, and the determinantal point process will be conditioned by $\Zf^c = \Xf \setminus \Zf = \{1,2,\ldots,N\}$. Our aim is to define a function $L : \Xf \times \Xf \to \R$ such that for every set $Y \subset \Zf$ one has
\[
 \mathcal{W}_N(Y) = \det \big(L_{Y\cup\Zf^c}\big),
\]
(where we use the notation from \eqref{eq:LEnsemble}) which means that $\mathcal{W}_N$ is a conditional $L$-ensemble. As we will see below, this corresponds to fixing the initial values $y_i$ of TASEP and the `infinities' $z^{n-1}_n$.

Using the equivalence between point configurations $(z^n_i)_{n, i} \in \Lambda_N$ and one-dimensional arrays described in the previous paragraph, every point configuration from $\Xf$ can be obviously identified with an array as well (by adding new boxes indexed by $1$, $2$, $\ldots$, $N$). In the previous example we will have the array
\begin{center}
\begin{tikzpicture}[
array/.style={matrix of nodes,nodes={draw, minimum width=10mm, minimum height=7mm},column sep=-\pgflinewidth, row sep=0.5mm, nodes in empty cells,
row 1/.style={nodes={draw=none, fill=none}, font=\scriptsize, color=blue}}]
\matrix[array] (array) {
$1$ & $2$ & $3$ & $(1,1)$ & $(2,1)$ & $(2,2)$ & $(3,1)$ & $(3,2)$ & $(3,3)$\\
$*^{\vphantom{1}}_{\vphantom{1}}$ & $*^{\vphantom{1}}_{\vphantom{1}}$ & $*^{\vphantom{1}}_{\vphantom{1}}$ & $z^1_1$ & $z^2_1$ & $z^2_2$ & $z^3_1$ & $z^3_2$ & $z^3_3$\\};
\end{tikzpicture}
\end{center}
where by $*$ we mean the values indexed by $1$, $2$, $\cdots$, $N$. This means that $L_{\{\bz\}\cup\Zf^c}$ (recall the notation from \eqref{eq:LEnsemble}) is in fact a function of two arguments each of which is either  $(n, i)$, such that $n \in \{1, \cdots, N\}$ and $i \in \{1, \cdots, n\}$, or $k \in \{1, \cdots, N\}$. So we can identify $L_{\{\bz\}\cup\Zf^c}$ with a square matrix of size $N(N+3) / 2$. Now, we are going to define this matrix. 

\subsubsection*{Writing the $\mathbf L$-function in a matrix form.} Let us denote by $\CM_{n, m}$ the set of all $n \times m$ matrices with real entries. Then for $1 \leq n < m \leq N$ we define the function $W_{[n,m)} : \Lambda_N \to \CM_{n, m}$ such that for a fixed particles configuration $\bz = (z_i^n) \in \Lambda_N$ the matrix $W_{[n,m)}(\bz)$ is given by
\[
 \Big[W_{[n,m)} \big(\bz\big)\Big]_{i,j} = \phi^{m - n}(z_i^n,z_j^{m}) \1{n < m},\qquad 1\leq i\leq n,\,1\le j \leq m.
\]
Similarly we define the function $\Psi^{(N)} : \Lambda_N \to \CM_{N, N}$ to have the entries
\[
\Big[\Psi^{(N)}(\bz)\Big]_{i,j} = \Psi^N_{N-j}(z_i^N),\qquad 1\leq i,j \leq N,
\]
where the functions $\phi$ and $\Psi^N_{N-j}$ are as in the statement of Theorem~\ref{thm:BFPS}. Finally, for $0 \leq m < N$ we define the function $E_m : \Lambda_N \to \CM_{N, m+1}$ by the entries
\[
\Big[E_{m}(\bz)\Big]_{i,j}=
\left\{\begin{array}{ll} 
\phi(z_{m+1}^m,z_j^{m+1}),\quad& \textrm{if}~ i=m+1,\;1\leq j \leq m+1,\\
0,&\textrm{otherwise}.
\end{array}\right.
\]
In fact, the function $E_{m}$ should also have the values of $z_{m+1}^m$ as arguments, but we prefer not to indicate this, since we will always fix these values to be infinities. With these objects at hand we identity the function $L_{\{\bz\} \cup \Zf^c}$ (recall, that the configuration $\bz \in \Lambda_N$ has been fixed) with the following square matrix in block form:
\begin{equation}\label{eq:MatrixL}
\arraycolsep=5pt\renewcommand\arraystretch{1.2}
L_{\{\bz\} \cup \Zf^c}=\left[\begin{array}{cccccc}
0 & E_0 & E_1 & E_2 &\ldots & E_{N-1} \\
0 & 0 & -W_{[1,2)} & 0 & \cdots & 0 \\
0 & 0 & 0 & -W_{[2,3)} & \ddots & \vdots \\
\vdots & \vdots & \vdots & \ddots & \ddots &0    \\
0 & 0 & 0 & 0 &\cdots & -W_{[N-1,N)} \\
\Psi^{(N)} & 0 & 0 & 0 & \cdots & 0
\end{array}
\right] (\bz),
\end{equation}
where each block takes $\bz$ as an argument and gives a usual matrix. The first $N$ columns (resp. rows) of the matrix \eqref{eq:MatrixL} are parametrized by the values $\{1,2,\ldots,N\}$ and the succeeding columns (resp. rows) are parametrized by the pairs $\{(n, i) : 1 \leq n \leq N,\, 1 \leq i \leq n\}$ which have the lexicographic order.

\begin{example}
In the case $N = 2$, a configuration $\bz \in \Lambda_2$ contains $3$ values $\{z_1^1\} \cup \{z_2^1, z_2^2\}$, and the matrix $L_{\{\bz\} \cup \Zf^c}$ is given by
\[
L_{\{\bz\} \cup \Zf^c}=
\begin{tikzpicture}[baseline=0pt,
array/.style={matrix of nodes,nodes={draw=none, minimum width=17mm, minimum height=7mm}, row sep=0.5mm, nodes in empty cells}]

\matrix[array, left delimiter={[},right delimiter={]}, decoration=brace] (array) {
$0$ 				& $0$  				& $\phi(z_{1}^0,z_1^{1})$ 	& $0$ 				& $0$ \\
$0$ 				& $0$ 				& $0$ 				& $\phi(z_{2}^1,z_1^{2})$ 	& $\phi(z_{2}^1,z_2^{2})$ \\
$0$ 				& $0$ 				& $0$ 				& $-\phi(z_1^1,z_1^{2})$ 	& $-\phi(z_1^1,z_2^{2})$ \\
$\Psi^2_{1}(z_1^2)$ 	& $\Psi^2_{0}(z_1^2)$ 	& $0$ 				& $0$ 				& $0$ \\
$\Psi^2_{1}(z_2^2)$ 	& $\Psi^2_{0}(z_2^2)$ 	& $0$ 				& $0$ 				& $0$ \\};

\node[xshift=-2em] at (array-1-1.west) {\scriptsize\color{blue}$1$};
\node[xshift=-2em] at (array-2-1.west) {\scriptsize\color{blue}$2$};
\node[xshift=-2em] at (array-3-1.west) {\scriptsize\color{blue}$(1,1)$};
\node[xshift=-2em] at (array-4-1.west) {\scriptsize\color{blue}$(2,1)$};
\node[xshift=-2em] at (array-5-1.west) {\scriptsize\color{blue}$(2,2)$};

\node[yshift=1em] at (array-1-1.north) {\scriptsize\color{blue}$1$};
\node[yshift=1em] at (array-1-2.north) {\scriptsize\color{blue}$2$};
\node[yshift=1em] at (array-1-3.north) {\scriptsize\color{blue}$(1,1)$};
\node[yshift=1em] at (array-1-4.north) {\scriptsize\color{blue}$(2,1)$};
\node[yshift=1em] at (array-1-5.north) {\scriptsize\color{blue}$(2,2)$};

\begin{scope}[on background layer]
\fill[red!20] (array-4-1.north west) rectangle (array-5-2.south east);
\fill[green!20] (array-1-3.north west) rectangle (array-2-3.south east);
\fill[yellow!20] (array-1-4.north west) rectangle (array-2-5.south east);
\fill[blue!20] (array-3-4.north west) rectangle (array-3-5.south east);
\end{scope}

\draw[<-,shorten <=1pt, blue,rounded corners] (array-5-1.south east)
 	|-+(0,-0.4)
    	node[below] {\scriptsize\color{blue}$\Psi^N$};
\draw[<-,shorten <=1pt, blue,rounded corners] (array-2-3.south)
 	|-+(0,-0.25)
	--+(0.3,-0.25)
	|-+(0.3,-2.8)
    	node[below] {\scriptsize\color{blue}$E_0$};
\draw[<-,shorten <=1pt, blue,rounded corners] (array-3-4.south east)
 	|-+(0,-2)
    	node[below] {\scriptsize\color{blue}$-W_{[1,2)}$};
\draw[<-,shorten <=1pt, blue,rounded corners] (array-1-5.south east)
 	--+(0.23,0)
	|-+(0.23,-1.8)
	--+(-0.3,-1.8)
	|-+(-0.3,-3.55)
    	node[below] {\scriptsize\color{blue}$E_1$};
\end{tikzpicture},
\]
where the `infinities' $z^0_1$ and $z^1_2$ are fixed. In particular, it follows from the definition of the function $\phi$ in Theorem~\ref{thm:BFPS} that $\phi(z_{1}^0,z) = \phi(z_{2}^1,z) = 1$ for any $z \in \Z$.
\end{example}

\subsubsection*{The $\mathbf L$-function defines the measure $\mathcal{W}_N$.} It is not difficult to see (recall the definition \eqref{eq:pointMeasure}) that one has the identity $\mathcal{W}(\bz) = C \det (L_{\{\bz\} \cup \Zf^c})$, where $C \in \{\pm 1\}$. To see this, we define the square matrices 
\[
 \Big[T_{m}\Big]_{i, j} = \phi(z_i^m,z_j^{m+1}),\qquad 1\le i, j \leq m+1.
\]
One can see that the first row of $T_{m}$ coincides with the only non-zero row of $E_{m}(\bz)$ and the other rows of $T_{m}$, with $m \geq 1$, form a matrix coinciding with $-W_{[m,m+1)} \big(\bz\big)$. Then the matrix \eqref{eq:MatrixL} is obtained by rows permutations from the following one:
\[
 \arraycolsep=5pt\renewcommand\arraystretch{1.2}
\left[\begin{array}{cccccc}
0 & T_0 & 0 & 0 &\ldots & 0 \\
0 & 0 & T_1 & 0 & \cdots & 0 \\
0 & 0 & 0 & T_2 & \ddots & \vdots \\
\vdots & \vdots & \vdots & \ddots & \ddots &0    \\
0 & 0 & 0 & 0 &\cdots & T_{N-1} \\
\Psi^{(N)}(\bz) & 0 & 0 & 0 & \cdots & 0
\end{array}
\right],
\]
(one does it by swapping the first row of $T_2$ with the previous two rows, then the first row of $T_3$ with the previous three rows, and so on). The determinant of the latter matrix, and hence of $L_{\{\bz\} \cup \Zf^c}$, equals exactly $\mathcal{W}(\bz)$ defined in \eqref{eq:pointMeasure}.

\subsubsection*{The correlation kernel.} One can see that the minors \eqref{eq:MatrixL} uniquely define the function $L : \Xf \times \Xf \to \R$. For example, in the case $N=2$ above, one has
\begin{align*}
 L\big((1), (1,1, z)\big) = \phi(z_{1}^0, z), &\qquad L\big((1,1,y), (2,2, z)\big) = -\phi(y,z),\\
 L\big((2,1, z), (2)\big) = \Psi^2_{0}(z), &\qquad L\big((2,1, y), (1,1, z)\big) = 0.
\end{align*}
Hence, the point measure $\mathcal{W}_N$ is a conditional $L$-ensemble with this function $L$. By Proposition~\ref{prop:ConditionalLensembles}, the point-measure $\mathcal{W}_N$ on $\Zf$ is determinantal with correlation kernel $\CK : \Zf \times \Zf \to \R$ given by
\begin{equation}\label{eq:KDef}
 \CK=1_\Zf-(1_\Zf+L)^{-1}\big|_{\Zf\times\Zf}.
\end{equation}
In fact, we can compute the inverse of the operator above. To this end, we identify the function $L$ with a function on $\Lambda_N$ and with values in $\CM_{N(N+3)/2, N(N+3)/2}$ so that the identities \eqref{eq:MatrixL} hold. Then $L$ can be written is the block form
\[
L=\left[
    \begin{array}{cc}
      0 & B \\
      C & D_0 \\
    \end{array}
  \right]
\]
with the blocks 
\[
 B=[E_0,\ldots,E_{N-1}] \;  : \; \Lambda_N \to \CM_{N, N(N+1)/ 2},\qquad  C=[0,\ldots,0, (\Psi^{(N)})']' \;:\; \Lambda_N \to \CM_{N(N+1)/ 2, N},
\]
and $D_0 : \Lambda_N \to \CM_{N(N+1)/ 2, N(N+1)/ 2}$ given by
\[
 D_0 = \arraycolsep=5pt\renewcommand\arraystretch{1.2}
\left[\begin{array}{ccccc}
0 & -W_{[1,2)} & 0 & \cdots & 0 \\
0 & 0 & -W_{[2,3)} & \ddots & \vdots \\
 \vdots & \vdots & \ddots & \ddots &0    \\
0 & 0 & 0 &\cdots & -W_{[N-1,N)} \\
 0 & 0 & 0 & \cdots & 0
\end{array}
\right].
\]
Defining furthermore the function $D=1 + D_0$ which can be written as
\[
 D = \arraycolsep=5pt\renewcommand\arraystretch{1.2}
\left[\begin{array}{ccccc}
1 & -W_{[1,2)} & 0 & \cdots & 0 \\
0 & 1 & -W_{[2,3)} & \ddots & \vdots \\
 \vdots & \vdots & \ddots & \ddots &0    \\
0 & 0 & 0 &\cdots & -W_{[N-1,N)} \\
 0 & 0 & 0 & \cdots & 1
\end{array}
\right],
\]
the result of~\cite[Lem.~1.5]{borodinRains} yields an expression for the correlation kernel $\CK$. For the sake of completeness we provide a proof here. As for the function $L$, we will identify the functions $B$, $C$ and $D$ with the functions on $\Xf \times \Xf$. Moreover, we will denote for clarity the convolutions over the values in $\Zf$ by $\star$, and the convolution over $\{1, \cdots, N\}$ we will write as a usual product.

\begin{lemma}
 The operators $D$ and $M=B \star D^{-1} \star C$ are invertible, and the correlation kernel $\CK$ defined in \eqref{eq:KDef} can be written as
\begin{equation}\label{eq:KAfterInversion}
\CK= 1_\Zf - D^{-1} + D^{-1} \star C M^{-1} B \star D^{-1}.
\end{equation}
\end{lemma}

\begin{proof}
The claim will follow if we show that one has
\begin{equation}\label{eq:LInverse}
 (1_\Zf+L)^{-1} = \arraycolsep=5pt\renewcommand\arraystretch{1.2}
 \left[
    \begin{array}{cc}
      - M^{-1} & M^{-1} B \star D^{-1} \\
      D^{-1} \star C M^{-1} & \quad D^{-1} - D^{-1} \star C M^{-1} B \star D^{-1} \\
    \end{array}
  \right],
\end{equation}
where $M$ is as in the statement of this lemma. The easiest way is to check that the matrix on the right-hand side multiplied by $1_\Zf+L$ is the identity matrix, i.e.
\begin{align*}
(1_\Zf+L)&\arraycolsep=5pt\renewcommand\arraystretch{1.2}
\left[
    \begin{array}{cc}
      - M^{-1} & M^{-1} B \star D^{-1} \\
      D^{-1} \star C M^{-1} & \quad D^{-1} - D^{-1} \star C M^{-1} B \star D^{-1} \\
    \end{array}
  \right]\\
  &= \left[
    \begin{array}{cc}
      0 & B \\
      C & D \\
    \end{array}
  \right]  \left[
    \begin{array}{cc}
      - M^{-1} & M^{-1} B \star D^{-1} \\
      D^{-1} \star C M^{-1} & \quad D^{-1} - D^{-1} \star C M^{-1} B \star D^{-1} \\
    \end{array}
  \right]\\
   &= \left[
    \begin{array}{cc}
      B \star D^{-1} \star C M^{-1} & \quad B \star D^{-1} - B \star D^{-1} C \star M^{-1} B \star D^{-1} \\
      - C  M^{-1} + C M^{-1} & \quad CM^{-1} B \star D^{-1} + 1 -  C M^{-1} B \star D^{-1} \\
    \end{array}
  \right]\\
  &= \left[
    \begin{array}{cc}
      M M^{-1} & \quad B \star D^{-1} - M M^{-1} B \star D^{-1} \\
      - C M^{-1} + C M^{-1} & \quad CM^{-1} B \star D^{-1} + 1 - C M^{-1} B \star D^{-1} \\
    \end{array}
  \right] = \left[
    \begin{array}{cc}
      1 & 0 \\
      0 & 1 \\
    \end{array}
  \right],
\end{align*}
which means that \eqref{eq:LInverse} indeed holds. Taking the restriction $(1_\Zf+L)^{-1}\big|_{\Zf\times\Zf}$ we get the right-bottom block of the matrix in \eqref{eq:LInverse} which combined with the definition \eqref{eq:KDef} gives exactly the expression on the right-hand side of \eqref{eq:KAfterInversion}.
\end{proof}

The inverse of the matrix $D$ can be computed very easily
\[
D^{-1} = (1 + D_0)^{-1} = \sum_{k \geq 0} D_0^{\star k} =\arraycolsep=5pt\renewcommand\arraystretch{1.2}
\left[\begin{array}{cccc}
1 & W_{[1,2)} & \cdots & W_{[1,N)} \\
0& 1  & \ddots & \vdots \\
\vdots & \ddots  & \ddots & W_{[N-1,N)} \\
0& 0 &  0 &1
\end{array}
\right],
\]
so that the submatrix of $1 - D^{-1}$ with rows $(n, \cdot)$ and columns $(m, \cdot)$ is $W_{[n, m)} \1{n < m}$. 

\begin{exercise}
 Prove that the inverse of $D$ is indeed given by the matrix above.
\end{exercise}

\noindent Moreover, we can easily compute
\[\arraycolsep=5pt\renewcommand\arraystretch{1.2}
D^{-1} \star C = \left[\begin{array}{c}
W_{[1,N)} \star \Psi^{(N)} \\
\vdots \\
W_{[N-1,N)} \star \Psi^{(N)} \\
\Psi^{N}
\end{array}\right],
\]
as well as
\[
B \star D^{-1} = \arraycolsep=5pt\renewcommand\arraystretch{1.2} \left[\begin{array}{cccc}
E_0 & E_0 \star W_{[1,2)}+E_1 & \cdots & \sum_{k=1}^{N-1} E_{k-1} \star W_{[k,N)}+ E_{N-1}
\end{array}\right].
\]
Therefore the $(n,m)$-block of the correlation kernel $\CK$ is given by
\begin{equation}\label{eq:K_interm}
\bigl[\CK \bigr]_{(n, \cdot), (m, \cdot)} = -W_{[n,m)} \1{n < m}+W_{[n,N)} \star \Psi^{(N)} M^{-1} \left(\sum_{k=1}^{m-1} E_{k-1} \star W_{[k,m)}+ E_{m-1}\right).
\end{equation}
It follows from the property \eqref{eq:PsiRecursion} that for $\bz \in \Lambda_N$ one has
\[
\left[W_{[n,N)} \star \Psi^{(N)}(\bz)\right]_{i,j} = \left(\phi^{N - n} * \Psi^N_{N-j}\right)(z_i^n)=\Psi^{n}_{n-j}(z_i^n),
\]
and it remains to evaluate the last part of \eqref{eq:K_interm}. For the $N\times m$ matrix in the bracket in \eqref{eq:K_interm} we have
\[
\left[ \left(\sum_{k=1}^{m-1} E_{k-1} \star W_{[k,m)}+ E_{m-1} \right)(\bz)\right]_{i,j} =
\left\{\begin{array}{ll}
\phi^{m + 1}(z_i^{i-1},z_j^m),\quad & 1\leq i \leq m,\\
0, & m < i \leq N,
\end{array}\right.
\]
and we arrive at the expression
\[
 \bigl[\CK(\bz)\bigr]_{(n, i), (m, j)} = - \phi^{m - n}(z_i^n,z_j^{m}) \1{n < m} + \sum_{\ell, k = 1}^m \Psi^{n}_{n-\ell}(z_i^n)\, \bigl[M^{-1}\bigr]_{\ell, k}\, \phi^{m + 1}(z_k^{k-1},z_j^m),
\]
where the matrix $M$ is given by
\begin{equation}\label{eq:MatrixM}
 	\bigl[M\bigr]_{i, j} = 	
\begin{cases}
 	\left(\phi^{N - i} * \Psi^N_{N- j}\right)(z_i^{i-1}),&  i < j,\\
	1,& i = j,\\
	0,& i > j.
\end{cases}
\end{equation}

\subsubsection*{Biorthogonalization of the correlation kernel.} The functions $\phi^{m + 1}(z_i^{i-1},x)$ with $i =1, \ldots, m$ form a basis of ${\rm span}\bigl\{1,x,\ldots,x^{m-1}\bigr\}$ (by considering $z_i^{i-1}$ to be a fixed value). Since by assumption the functions $\bigl\{\Phi^m_{m-1}(x),\ldots,\Phi^m_0(x)\bigr\}$ form a basis of this space as well, we can define a matrix $A_m \in \CM_{m, m}$ which does a change of basis to $\bigl\{\Phi^m_{m-1}(x),\ldots,\Phi^m_0(x)\bigl\}$, namely
\[
	\phi^{m + 1}(z_i^{i-1},x)=\sum_{\ell=1}^m \bigl[A_m\bigr]_{i,\ell}\, \Phi^m_{m-\ell}(x).
\]
We convolve this equation with $\Psi^m_{m-j}(x)$ and obtain, using the biorthogonality assumption,
\[
\bigl[A_m\bigr]_{i,j}= \left (\phi^{m  - i + 1}*\Psi^m_{m-j}\right)(z_i^{i-1}).
\]
In particular, we have $A_N=M$, and since $M$ is invertible, the two properties from Theorem~\ref{thm:BFPS} indeed define the functions $\Phi^n_k$ uniquely. Thus we obtain
\begin{equation}\label{eq:some_expression}
\sum_{k=1}^m \bigl[M^{-1}\bigr]_{\ell, k}\, \phi^{m + 1}(z_k^{k-1},z_j^m) = \sum_{k=1}^m \bigl[A_N^{-1}\bigr]_{\ell, k}\, \sum_{i=1}^m \bigl[A_m\bigr]_{k,i}\, \Phi^m_{m-i}(z_j^m),
\end{equation}
and our claim is now to prove that 
\begin{equation}\label{eq:AOrth}
 \sum_{k=1}^m \bigl[A_N^{-1}\bigr]_{\ell, k} \bigl[A_m\bigr]_{k,i} = \delta_{\ell, i},
\end{equation}
for all $1 \leq \ell \leq N$ and $1 \leq i \leq m$. We notice that 
\[
\bigl[A_m\bigr]_{k, i}= \left (\phi^{m  - k + 1}*\Psi^m_{m-i}\right)(z_k^{k-1}) = \left (\phi^{N - k + 1}*\Psi^N_{N-i}\right)(z_k^{k-1}) = \bigl[A_N\bigr]_{k, i},
\]
for $1 \leq k, i \leq m$. Thus, using the fact that $A_N$ is upper-triangular (which follows from \eqref{eq:MatrixM}), we obtain for $1 \leq i \leq m$:
\[
 \sum_{k=1}^m \bigl[A_N^{-1}\bigr]_{\ell, k} \bigl[A_m\bigr]_{k,i} = \sum_{k=1}^m \bigl[A_N^{-1}\bigr]_{\ell, k} \bigl[A_N\bigr]_{k,i} = \sum_{k=1}^N \bigl[A_N^{-1}\bigr]_{\ell, k} \bigl[A_N\bigr]_{k,i} = \delta_{\ell, i},
\] 
which is exactly our claim. This gives that the right-hand side of \eqref{eq:some_expression} is equal to $\Phi^m_{m-\ell}(z_j^m)$, and the $(n,m)$-th block of the correlation kernel $\CK$ is given by
\[
 \bigl[\CK(\bz)\bigr]_{(n, i), (m, j)} = - \phi^{m - n}(z_i^n,z_j^{m}) \1{n < m} + \sum_{\ell = 1}^m \Psi^{n}_{n-\ell}(z_i^n)\, \Phi^m_{m-\ell}(z_j^m),
\]
which, if we go to the function $\CK : \Zf \times \Zf \to \R$, is equivalent to 
\[
\CK \bigl((n, i, x), (m, j, y)\bigr)=-\phi^{m - n}(x,y) \1{m > n} + \sum_{\ell = 1}^m\Psi^{n}_{n-\ell}(x)\, \Phi^{m}_{m-\ell}(y).
\]
Since we are interested in the distribution of the particles $z^n_1$ (see Figure~\ref{fig:GT}), the kernel \eqref{eq:Kt} is obtained by fixing the values $i = j = 1$ in the above expression.

\begin{exercise} Follow the proof of Lemma~\ref{lem:GapProbability} to show that the identity \eqref{eq:extKernelProbBFPS} holds.
\end{exercise}

\section{Explicit formulas for the correlation kernel}
\label{sec:exact}

Only for a few special cases of initial data (step, see e.g.~\cite{dimers}; and periodic~\cite{borFerPrahSasam,bfp,bfs}) were the correlation kernels \eqref{eq:Kt} known, and hence only for those choices asymptotics could be performed in the TASEP and related cases, leading to the Tracy-Widom $F_{\text{GUE}}$ and $F_{\text{GOE}}$ one-point distributions, and then later to the Airy processes for multipoint distributions. Below we provide formulas for arbitrary initial data and their extensions as $N \to \infty$.

\subsection{Finite initial data.}

We conjugate the kernels from Theorem~\ref{thm:BFPS} by powers of $2$: 
\begin{equation}\label{eq:defQ}
  Q(x,y)=\frac{1}{2^{x-y}}\1{x>y},
\end{equation}
as well as
\begin{equation}\label{eq:defPsi}
\Psi^n_k(x)=\frac1{2\pi\I}\oint_{\Gamma_0}dw\,\frac{(1-w)^k}{2^{x-X_0(n-k)}w^{x+k+1-X_0(n-k)}} e^{t(w-1)}.
\end{equation}
Then the functions $\Phi_k^{n}(x)$, $k=0,\ldots,n-1$, are defined implicitly by
\begin{enumerate}[label={\normalfont (\arabic{*})}]
\item the biorthogonality relation $\sum_{x\in\Z}\Psi_k^{n}(x)\Phi_\ell^{n}(x)=\1{k=\ell}$;
\label{ortho}
\smallskip
\item $2^{-x}\Phi^n_k(x)$ with $0 \leq k < n$ form a basis of ${\mathrm{span}}\bigl\{x^k : 0 \leq k < n\bigr\}$.
\end{enumerate} 
We are interested in computing the kernel,
\begin{equation}\label{eq:Kt_new}
K_t(n_1,x_1;n_2,x_2)=-Q^{n_2-n_1}(x_1,x_2)\1{n_1<n_2}+\sum_{k=1}^{n_2}\Psi^{n_1}_{n_1-k}(x_1)\Phi^{n_2}_{n_2-k}(x_2),
\end{equation}
for $n_1,n_2\in\Z_{\geq1}$ and $x_1,x_2\in\Z$. The initial data $X_0$ appears in a simple way in the functions $\Psi_k^n$, which can be computed explicitly. We note that $Q$ is the transition matrix of a geometric random walk on $\Z$ and has a left inverse 
\begin{equation}\label{eq:Qinv}
Q^{-1}=I+2\nablap, 
\end{equation}
where we recall that $\nablap f(x) = f(x+1) - f(x)$. Moreover, for all $m,n\in \Z_{\geq 0}$ we have the identities
\[ Q^{n-m}\Psi^n_{n-k}=\Psi^m_{m-k}, \hspace{1cm} \Psi^{n}_k=e^{-\frac{t}{2} (I+\nablam)}Q^{-k}\delta_{X_0(n-k)},\]
where $\delta_x(y) = \uno{x = y}$ is the Kronecker's delta and $\nablam f(x) = f(x) - f(x-1)$.
  
\begin{exercise} 
Prove that these identities hold.
\end{exercise}

\noindent In order to give exact formulas for the functions $\Phi^n_k$ from Theorem~\ref{thm:BFPS} we need to define the functions $h^n_k : \Z_{\geq 0} \times \Z \to \R$ as solutions to the initial-boundary value problem for the backwards heat equation
\begin{subnumcases}{\label{bhe}}
(\Qt)^{-1}h^n_k(\ell,z)=h^n_k(\ell+1,z), &  $0 \leq \ell<k,\,z \in \Z$;\label{bhe1}\\ 
h^n_k(k,z)=2^{z-X_0(n-k)}, & $z \in \Z$;\label{bhe2}\\ 
h^n_k(\ell,X_0(n-\ell))= 0, & $0 \leq \ell<k$;\label{bhe3}
\end{subnumcases} 
for $n\geq1$ and $0\le k <n$ and where $Q^*$ is the  adjoint of $Q$.

\begin{exercise} 
Show that the dimension of $\ker (Q^*)^{-1}$ is $1$ and use this to show existence and uniqueness of the solution to the equation \eqref{bhe}.
\end{exercise}

\begin{remark} 
The expression $\Qt h^{n}_{k}(k,z)$ is divergent so you can't write  $\Qt h^{n}_k(\ell+1,z)=h^{n}_k(\ell,z)$.
\end{remark}

\noindent With these functions at hand we are ready to give exact formulas for $\Phi^n_k$.

\begin{theorem}\label{thm:h_heat}
The functions $\Phi^n_k$ from Theorem~\ref{thm:BFPS} are given by
\begin{equation}\label{eq:defPhink}
 \Phi^n_k(z) =\sum_{y\in \Z} h^{n}_k(0,y)\,e^{\frac{t}{2} (I+\nabla^-)}(y,z).
\end{equation}
\end{theorem}

\begin{proof}
Before proving \eqref{eq:defPhink} we need to prove that $2^{-x}h^{n}_{k}(0,x)$ is a polynomial of degree $k$. We proceed by induction.
Note first that, by \eqref{bhe2}, $2^{-x}h^{n}_{k}(k,x)$ is a polynomial of degree 0.
Assume now that $\tilde h^n_k(\ell,x)=2^{-x}h^{n}_{k}(\ell,x)$ is a polynomial of degree $k-\ell$ for some $0<\ell\leq k$.
By \eqref{bhe1} and \eqref{eq:Qinv} we have
\begin{equation}\label{eq:tildeh-eq}
  \tilde h^{n}_{k}(\ell,y)=2^{-y}(\Qt)^{-1}h^{n}_{k}(\ell-1,y)=\tilde h^{n}_{k}(\ell-1,y-1)-\tilde h^{n}_{k}(\ell-1,y)
\end{equation}
Taking $x\geq X_0(n-\ell+1)$ and summing gives $\tilde h^{n}_{k}(\ell-1,x)=-\sum_{y=X_0(n-\ell+1)+1}^{x}2^{-y}h^{n}_{k}(\ell,y)$ thanks to \eqref{bhe3}, which by the inductive hypothesis is a polynomial of degree $k-\ell+1$.
The function\break$\tilde h^n_k(\ell-1,x)\big|_{x\geq X_0(n-\ell+1)}$ has a unique polynomial extension to all $\Z$, which by uniqueness of solutions of \eqref{bhe} and \eqref{eq:tildeh-eq} shows that $\tilde h^{n}_{k}(\ell-1,\cdot)$ is a polynomial of degree $k-\ell+1$ as needed.

Now we check the biorthogonality relation of \eqref{eq:defPhink} with the functions $\Psi^n_k$. We have
\begin{align*}
	\sum_{z\in\Z}\Psi^n_\ell(z)\Phi^n_k(z) &=\sum_{z_1,z_2\in\Z}\sum_{z\in\Z}e^{-\frac{t}{2} (I+\nabla^-)}(z,z_1)Q^{-\ell}(z_1,X_0(n-\ell))h^{n}_{k}(0,z_2)e^{\frac{t}{2} (I+\nabla^-)}(z_2,z)\\
	&=\sum_{z\in\Z}Q^{-\ell}(z,X_0(n-\ell))h^{n}_{k}(0,z)
	=(\Qt)^{-\ell}h^{n}_{k}(0,X_0(n-\ell)),
\end{align*}
where in the first equality we have used the fact that $2^{-x}h^n_k(0,x)$ is a polynomial together with the fact that the $z_1$ sum is finite to apply Fubini. For $0 \leq \ell\leq k$, we use all equations in \eqref{bhe} to get
\[
(\Qt)^{-\ell}h^{n}_{k}(0,X_0(n-\ell))=h^{n}_{k}(\ell,X_0(n-\ell))=\1{k=\ell}.
\]
For $\ell>k$, we use \eqref{bhe1} and $2^z\in\ker{(\Qt)^{-1}}$ to obtain
\[
(\Qt)^{-\ell}h^{n}_{k}(0,X_0(n-\ell))=(\Qt)^{-(\ell-k-1)}(\Qt)^{-1}h^{n}_{k}(k,X_0(n-\ell))=0.
\] 
Since $2^{-x}\Phi^n_k(x)$ is a polynomial of degree $k$ in $x$, this one is as well.
\end{proof}

\subsection{Correlation kernel as a transition probability.} 
\label{subsec:hitting}

In this section we will perform summation in \eqref{eq:Kt_new} and obtain a formula for the correlation kernel involving hitting probabilities of a random walk. We start with noting that it is sufficient to fund the kernel \eqref{eq:Kt_new} with $n_1 = n_2$.

\begin{exercise}  Show that the kernel \eqref{eq:Kt_new} can be recovered from the kernel  $K^{(n)}_t(x_1,x_2)=K_t(n,x_1;n,x_2)$ by
\begin{equation}
K_t(n_i,\cdot;n_j,\cdot)=Q^{n_j-n_i} \bigl(-\1{n_i<n_j}+K^{(n_j)}_t\bigr).
\end{equation}
\end{exercise}

\noindent From the exercise, we can restrict our discussion to the kernel $K^{(n)}_t$. Using the functions $h^n_k$ let us define
\begin{equation}
	G_{0,n}(z_1,z_2)=\sum_{k=0}^{n-1}Q^{n-k}(z_1,X_0(n-k))h^{n}_{k}(0,z_2),
\end{equation}
so that the correlation kernel $K_t^{(n)}$ equals
\begin{equation}\label{eq:Kt-decomp}
	K_t^{(n)}=e^{-\frac{t}{2} (I+\nabla^-)} Q^{-n}G_{0,n}e^{\frac{t}{2} (I+\nabla^-)}.
\end{equation}
Next, we note that the functions $h^n_k$ can be written as hitting probabilities of a random walk. More precisely, let $\Qt$ (the adjoint of $Q$) be the transition kernel of the random walk $B_m^*$ with Geom$\bigl[\frac12\bigr]$ jumps (strictly) to the right. Then for $0\leq\ell\leq k\leq n-1$ we define stopping times
\[
\tau^{\ell,n}=\min \bigl\{m\in\{\ell,\ldots,n-1\}\!:\,B^*_m> X_0({n-m})\bigr\},
\]
with the convention that $\min\emptyset=\infty$. For $z\leq X_0(n-\ell)$ the function $h^n_k$ can be written as
\begin{equation}
h^n_k(\ell,z)=\pp_{B^*_{\ell-1}=z}\big(\tau^{\ell,n}=k\big).
\end{equation}
  
\begin{exercise} 
Prove this identity.
\end{exercise}

\begin{exercise} Suppose $X$ 
 is a random variable taking values in $\mathbb{N}$
with the \emph{memoryless property},i.e. for each pair of numbers $m,n\in\mathbb{N}$ one has
$\pp(X\ge m+n\mid X>n) = \pp(X\ge m)$.
Show that $X$ has a geometric distribution.
\end{exercise}

From the memoryless property of the geometric distribution we get for all $y > X_0(n-k)$,
\begin{equation}\label{eq:memoryless}
\pp_{B^*_{-1}=z}\big(\tau^{0,n}=k,\,B^*_k=y\big)=2^{X_0(n-k)-y}\pp_{B^*_{-1}=z}\big(\tau^{0,n}=k\big),
\end{equation}
and as a consequence, for $z_2\leq X_0(n)$, we have
\begin{equation}\label{eq:G-formula}
   G_{0,n}(z_1,z_2)=\pp_{B^*_{-1}=z_2}\big(\tau^{0,n}<n,\,B^*_{n-1}=z_1\big),
\end{equation}
which is the probability for the walk starting at $z_2$ at time $-1$ to end up at $z_1$ after $n$ steps, having hit the curve $\big(X_0(n-m)\big)_{m=0,\ldots,n-1}$ in between.

\begin{exercise} 
Show that the identities \eqref{eq:memoryless} and \eqref{eq:G-formula} indeed hold.
\end{exercise}

\noindent The next step is to obtain an expression along the lines of \eqref{eq:G-formula} which holds for all $z_2$, and not just for $z_2\leq X_0(n)$. To this end, we need to define analytic extensions of the kernels $Q^n$. More precisely, for each fixed $y_1$, $2^{-y_2}Q^n(y_1,y_2)$ extends in $y_2$ to a  polynomial $2^{-y_2}\bar{Q}^{(n)}(y_1,y_2)$ of degree $n-1$ with 
\begin{equation}
\bar{Q}^{(n)}(y_1,y_2)= \frac{1}{2\pi \I} \oint_{\Gamma_0} dw\,\frac{(1+w)^{y_1 - y_2 -1}}{2^{y_1-y_2} w^n},
\end{equation}
so that for $y_1-y_2\geq n$ we have $\bar{Q}^{(n)}(y_1,y_2)=Q^n(y_1,y_2)$. Furthermore, it is easy to prove that
$Q^{-1}\bar{Q}^{(n)}=\bar{Q}^{(n)}Q^{-1}=\bar{Q}^{(n-1)}$ for $n>1$, but $Q^{-1}\bar{Q}^{(1)}=\bar{Q}^{(1)}Q^{-1}=0$, and
$\bar{Q}^{(n)}\bar{Q}^{(m)}$ is divergent (so the $\bar{Q}^{(n)}$ are no longer a group like $Q^n$).

\begin{exercise} 
Prove that these properties of $\bar{Q}^{(n)}$ indeed hold.
\end{exercise}

\noindent Let $B_m$ be now a random walk with transition matrix $Q$ (that is, $B_m$ has Geom$\bigl[\tfrac{1}{2}\bigr]$ jumps strictly to the left) for which we define the stopping time
\begin{equation}\label{eq:deftau}
 \tau= \min\bigl\{ m\ge 0: B_m> X_0(m+1)\bigr\}.
\end{equation}
Using this stopping time and the extension of $Q^m$ we obtain:

\begin{lemma}\label{lem:G0n-formula}
For all $z_1,z_2\in\Z$ we have the identity
\begin{equation}
G_{0,n}(z_1,z_2) = \1{z_1>X_0(1)} \bar{Q}^{(n)}(z_1,z_2) + \1{z_1\leq X_0(1)}\ee_{B_0=z_1}\!\left[ \bar{Q}^{(n - \tau)}(B_{\tau}, z_2)\1{\tau<n}\right].
\end{equation}
\end{lemma}

\begin{proof}
For $z_2\leq X_0(n)$, the expression in \eqref{eq:G-formula} can be written as
\begin{multline}\label{eq:G0napp1}
G_{0,n}(z_1,z_2)=\pp_{B^*_{-1}=z_2}\big(\tau^{0,n}\le n-1,\,B^*_{n-1}=z_1\big)=\pp_{B_{0}=z_1}\big(\tau\leq n-1,B_{n}=z_2\big)\\
 =\sum_{k=0}^{n-1}\sum_{z>X_0(k+1)}\pp_{B_{0}=z_1}\big(\tau=k,\,B_{k}=z\big)Q^{n-k}(z,z_2)
 =\ee_{B_0=z_1}\!\left[Q^{n-\tau}\big(B_{\tau},z_2\big)\1{\tau<n}\right].
\end{multline}
The last expectation is straightforward to compute if $z_1>X_0(1)$, and we get
\begin{equation}
G_{0,n}(z_1,z_2)=\1{z_1>X_0(1)}Q^n(z_1,z_2)+\1{z_1\leq X_0(1)}\ee_{B_0=z_1}\!\left[Q^{n-\tau}\big(B_{\tau},z_2\big)\1{\tau<n}\right]
\end{equation}
for all $z_2\leq X_0(n)$.
Let us now denote
\[
\wt G_{0,n}(z_1,z_2)=\1{z_1>X_0(1)}\bar{Q}^{(n)}(z_1,z_2)+\1{z_1\leq X_0(1)}\ee_{B_0=z_1}\!\left[\bar{Q}^{(n-\tau)}\big(B_{\tau},z_2\big)\1{\tau<n}\right].
\]
We claim that $\wt G_{0,n}(z_1,z_2)=G_{0,n}(z_1,z_2)$ for all $z_2\leq X_0(n)$.
To see this,  note that 
\[
\P_{X_0(1)}\bar{Q}^{(n)}\bar\P_{X_0(n)}=\P_{X_0(1)}{Q^n}\bar\P_{X_0(n)},
\]
thanks to the properties proved in the exercise above. For the other term, the last equality in \eqref{eq:G0napp1} shows that we only need to check that $\P_{X_0(k+1)}\bar{Q}^{(k+1)}\bar\P_{X_0(n)}=\P_{X_0(k+1)}{Q^{k+1}}\bar\P_{X_0(n)}$ for $k=0,\ldots,n-1$, which follows again from the same fact proved in the exercise.
To complete the proof, recall that, by Theorem~\ref{eq:extKernelProbBFPS}, $K^{(n)}_t$ satisfies the following:  for every fixed $z_1$, $2^{-z_2}K_t(z_1,z_2)$ is a polynomial of degree $n-1$ in $z_2$.
It is easy to check that this implies that $G_{0,n}=Q^nR^{-1}K_tR$ satisfies the same.
Since $\bar{Q}^{(k)}$ also satisfies this property for each $k=0,\ldots,n$, we deduce that $2^{-z_2}\wt G_{0,n}(z_1,z_2)$ is a polynomial in $z_2$.
Since it coincides with $2^{-z_2}G_{0,n}(z_1,z_2)$ at infinitely many $z_2$'s, we deduce that $\wt G_{0,n}=G_{0,n}$.  
\end{proof}

In order to have a lighter notation, we define the following kernels
\begin{align}
 \SM_{-t,-n}(z_1,z_2) &:= (e^{-\frac{t}2 \nabla^-} Q^{-n})^*(z_1,z_2)  = \frac{1}{2\pi\I} \oint_{\Gamma_0}dw\, \frac{(1-w)^{n}}{2^{z_2-z_1} w^{n +1 + z_2 - z_1}}e^{t(w-1 / 2)},\label{def:sm}\\
 \SN_{-t,n} (z_1,z_2) &:=  \bar{Q}^{(n)}e^{\frac{t}2 \nabla^-} (z_1,z_2)= \frac{1}{2 \pi \I} \oint_{\Gamma_{0}} dw\,\frac{(1-w)^{z_2-z_1 + n - 1}}{2^{z_1-z_2} w^{n}} e^{t(w - 1 / 2)}.\label{def:sn}
\end{align}

\begin{exercise}  
Show that these operators indeed are given by the contour integrals.
\end{exercise}

\noindent Furthermore, we define the following function
\begin{equation}
\bar{\SN}_{-t,n}^{\epi(X_0)}(z_1,z_2) = \ee_{B_0=z_1}\!\left[ \SN_{-t,n - \tau}(B_{\tau}, z_2)\1{\tau<n}\right],
\end{equation}
where the superscript epi refers to the fact that $\tau$ (defined in \eqref{eq:deftau}) is the hitting time of the epigraph of the curve $\big(X_0(k+1)+1\big)_{k=0,\ldots,n-1}$ by the random walk $B_k$. With these operators at hand we have the following formula for TASEP with general right-finite initial data:

\begin{theorem}[TASEP formula for right-finite initial data]\label{thm:tasepformulas}
 Assume that initial values satisfy $X_0(j)=\infty$ for $j\le 0$. Then for  $1\leq n_1<n_2<\dotsm<n_M$ and $t>0$ we have
\begin{equation}\label{eq:extKernelProb}
  \pp\!\left(X_t(n_j)>a_j,~j=1,\ldots,M\right)=\det \bigl(I-\bar\chi_a K_t\bar\chi_a\bigr)_{\ell^2(\{n_1,\ldots,n_M\}\times\Z)},
\end{equation}
where the kernel $K_t$ is given by
\begin{equation}\label{eq:Kt-2}
K_t(n_i,\cdot;n_j,\cdot)=-Q^{n_j-n_i}\1{n_i<n_j}+(\SM_{-t,-{n}_i})^*\bar{\SN}_{-t,n_j}^{\epi(X_0)}.
\end{equation}
\end{theorem}

\begin{proof}
If $X_0(1)<\infty$ then we are in the setting of the above sections.
Formulas \eqref{eq:extKernelProb} and \eqref{eq:Kt-2} follow directly from the above definitions together with \eqref{eq:Kt-decomp} and Lemma~\ref{lem:G0n-formula}.
If $X_0(j)=\infty$ for $j=1,\ldots,\ell$ and $X_0(\ell+1)<\infty$ then 
\[
\pp_{X_0}\big(X_t(n_j)>a_j,~j=1,\ldots,M\big)=\det\!\big(I-\bar\chi_aK^{(\ell)}_t\bar\chi_a\big)_{\ell^2(\{n_1,\ldots,n_M\}\times\Z)}
\]
 with the correlation kernel
 \[
 K^{(\ell)}_t(n_i,\cdot;n_j,\cdot)=-Q^{n_j-n_i}\1{n_i<n_j}+(\SM_{-t,-{n}_i+\ell})^*{\SN}_{-t,n_j-\ell}^{\epi(\theta_\ell X_0)},
 \]
 where $\theta_\ell X_0(j) = X_0(\ell + j)$. Using now the fact that $Q^{\ell}\bar{\SN}_{-t,n_j-\ell}^{\epi(\theta_\ell X_0)}={\SN}_{-t,n_j}^{\epi(X_0)}$ and \eqref{def:sm} we conclude that \eqref{eq:Kt-2} still holds in this case.
\end{proof}

\begin{remark}  One can write a formula for initial data which are not right finite~\cite{KPZ}, but they are a bit more cumbersome.  In practice,
one cuts of the data very far to the right, and uses the formula above.  Since one has exact formulas, one can check that cutting off has a small effect.
\end{remark}

For some special initial data one can get simpler expressions for the correlation kernel~\cite{KPZ} and we can recover the formulas from \cite{borFerPrahSasam}, \cite{ferrariMatr} and \cite{bfp}. We leave these computations as exercises below.

\begin{exercise}
Consider TASEP with step initial data, $X_0(i) = -i$ for $i \geq 1$ and show that
\[
K_t(n_i,z_1;n_j,z_2)=-Q^{n_j-n_i}(z_1, z_2)\uno{n_i<n_j} + \frac{1}{(2\pi\I)^2} \oint_{\Gamma_0}dw \oint_{\Gamma_0}dv\, \frac{(1-w)^{n_i} (1-v)^{n_j + z_2}}{2^{z_1-z_2} w^{n_i + z_1 +1} v^{n_j}} \frac{e^{t(w + v-1)}}{1 - v - w}.
\]
\end{exercise}

\begin{exercise}$\!\!\!\!{}^*$
Consider TASEP with periodic initial data $X_0(i)=2i$, $i\in\Z$ and show that 
\[K^{(n)}_t(z_1,z_2)=-\frac1{2\pi\I}\oint_{1+\Gamma_{0}}dv\, \frac{v^{z_2+2n}}{2^{z_1-z_2}(1-v)^{z_1+2n+1}}\, e^{t(1-2v)}.\]
(Hint:  approximate by finite periodic initial data $X_0(i) = 2(N-i)$ for $i=1,\ldots,2N$.)
\end{exercise}

\subsection{Path integral formulas.}

In addition to the extended kernel formula \eqref{eq:extKernelProbBFPS}, one has a \emph{path integral formula}
\begin{equation}
\det \left(I-K^{(n_m)}_{t} \bigl(I-Q^{n_1-n_m}\P_{a_1}Q^{n_2-n_1}\P_{a_2}\dotsm Q^{n_m-n_{m-1}}\P_{a_m}\bigr)\right)_{L^2(\Z)},\label{eq:path-int-kernel-TASEPgem}
\end{equation}
where as before $K^{(n)}_t(z_1,z_2)=K_t(n,z_1;n,z_2)$. Such formulas were first obtained in \cite{prahoferSpohn} for the Airy$_2$ process (see \cite{prolhacSpohn} for the proof), and later was extended to the Airy$_1$ process in \cite{quastelRemAiry1} and then to a very wide class of processes in \cite{bcr}. We provide this result below in full generality.

For $t_1<t_2<\dotsm<t_n$ we consider an extended kernel $K^\uptext{ext}$ given as follows: for $1\leq i,j\leq n$ and $x,y\in X$ (here $(X,\mu)$ is a given measure space),
\begin{equation}\label{eq:generalExt}
  K^\uptext{ext}(t_i,x;t_j,y)=
  \begin{dcases*}
    \mathcal{W}_{t_i,t_j}K_{t_j}(x,y), & if $i\geq j$,\\
    -\mathcal{W}_{t_i,t_j}(I-K_{t_j})(x,y), & if $i<j$.
  \end{dcases*}
\end{equation}
Additionally, we are considering multiplication operators $\Ml_{t_i}$ acting on a measurable function $f$ on $X$ as $\Ml_{t_i}f(x)=\varphi_{t_i}(x)f(x)$ for some measurable function $\varphi_{t_i}$ defined on $X$.
$M$ will denote the diagonal operator acting on functions $f$ defined on $\{t_1,\ldots,t_n\}\!\times\!X$ as $Mf(t_i,\cdot)=\Ml_{t_i}f(t_i,\cdot)$.

We provide below all the assumptions from \cite{bcr} except their Assumption 2(iii) which has to be changed in our case.

\begin{assumption}\label{assum:1}
There are integral operators $Q_t$ on $L^2(X)$ such that the following hold:
  \begin{enumerate}[label=(\roman*)]
  \item The integral operators $Q_{t_i}\mathcal{W}_{t_i,t_j}$, $Q_{t_i}K_{t_i}$,
  $Q_{t_i}\mathcal{W}_{t_i,t_j}K_{t_j}$ and $Q_{t_j}\mathcal{W}_{t_j,t_i}K_{t_i}$ for $1\leq i<j\leq n$
  are all bounded operators mapping $L^2(X)$ to itself.
  \item The operator $K_{t_1}-\bar Q_{t_1}\mathcal{W}_{t_1,t_2}\bar Q_{t_2}\cdots\mathcal{W}_{t_{n-1},t_n}\bar Q_{t_n}
  \mathcal{W}_{t_n,t_1}K_{t_1}$, where $\bar Q_{t_i}=I-Q_{t_i}$, is a bounded operator mapping
  $L^2(X)$ to itself.
  \end{enumerate}
\end{assumption}

\begin{assumption}\label{assum:2}
  For each $i\leq j\leq k$ the following hold:
    \begin{enumerate}[label=(\roman*)]
    \item \emph{Right-invertibility}: $\mathcal{W}_{t_i,t_j}\mathcal{W}_{t_j,t_i}K_{t_i}=K_{t_i}$;
    \item \emph{Semigroup property}: $\mathcal{W}_{t_i,t_j}\mathcal{W}_{t_j,t_k}=\mathcal{W}_{t_i,t_k}$;
    \item \emph{Reversibility relation}: $\mathcal{W}_{t_i,t_j}K_{t_{j}}\mathcal{W}_{t_j,t_{i}}=K_{t_{i}}$, for all $t_i<t_j$.
    \end{enumerate}
 \end{assumption}
 
 \begin{assumption}\label{assum:3}
  One can choose multiplication operators $V_{t_i}$, $V'_{t_i}$, $U_{t_i}$ and $U'_{t_i}$
  acting on $\CM(X)$, for $1\leq i\leq n$, in such a way that:
  \begin{enumerate}[label=(\roman*)]
  \item $V_{t_i}'V_{t_i}Q_{t_i}=Q_{t_i}$  and $K_{t_i}U_{t_i}'U_{t_i}=K_{t_i}$, for all $1\leq i\leq n$.
  \item The operators $V_{t_i}Q_{t_i}K_{t_i}V_{t_i}'$, $V_{t_i}Q_{t_i}\mathcal{W}_{t_i,t_j}V_{t_j}'$,
    $V_{t_i}Q_{t_i}\mathcal{W}_{t_i,t_j}K_{t_j}V_{t_j}'$ and
    $V_{t_j}Q_{t_j}\mathcal{W}_{t_j,t_i}K_{t_i}V_{t_i}'$ preserve $L^2(X)$ and are trace class in $L^2(X)$, for all $1\leq i<j\leq n$.
  \item The operator
    $U_{t_{i}}\!\left[\mathcal{W}_{t_{i},t_1}K_{t_1}-\bar{Q}_{t_{i}}\mathcal{W}_{t_{i},t_{i+1}}\dotsm
      \bar{Q}_{t_{n-1}}\mathcal{W}_{t_{{n-1}},t_{n}}\bar{Q}_{t_{n}}\mathcal{W}_{t_{n},t_1}K_{t_1}\right]U_{t_1}'$ preserves
    $L^2(X)$ and is trace class in $L^2(X)$, for all $1\leq i\leq n$, where
    $\bar{Q}_{t_i}=I-Q_{t_i}$.
  \end{enumerate}
\end{assumption}

We are assuming here that $\mathcal{W}_{t_i,t_j}$ is invertible for all $t_i\leq t_j$, so that $\mathcal{W}_{t_j,t_i}$ is defined as a proper operator\footnote{This is just for simplicity; it is possible to state a version of Theorem~\ref{thm:alt-extendedToBVP} asking instead that the product $K_{t_j}\mathcal{W}_{t_j,t_i}$ be well defined.}.
Moreover, we assume that it satisfies
\begin{equation}
  \mathcal{W}_{t_j,t_i}K_{t_{i}}=K^\uptext{ext}(t_j,\cdot;t_i,\cdot)
\end{equation}
for all $t_i\geq t_j$, and that the multiplication operators $U_{t_i},U_{t_i}'$ introduced in Assumption~\ref{assum:3} satisfy Assumption~\ref{assum:3}(iii) with the operator in that assumption replaced by
\[U_{t_i}\left[\mathcal{W}_{t_{i},t_{i+1}}\oM_{t_{i+1}}
    \dotsm\mathcal{W}_{{t_{n-1}},{t_{n}}}\oM_{{t_{n}}}K_{t_n}-\mathcal{W}_{t_{i},{t_1}}\oM_{{t_1}}\mathcal{W}_{{t_1},{t_2}}\oM_{{t_2}}\dotsm\mathcal{W}_{{t_{n-1}},{t_{n}}}\oM_{{t_{n}}}K_{t_n}\right]U_{t_i}'.\]

\begin{theorem}\label{thm:alt-extendedToBVP}
  Under the assumptions above, we have the identity
  \begin{equation}
    \det\!\big(I-\Ml K^\uptext{ext}\big)_{L^2(\{t_1,\dots,t_n\}\times X)}
    =\det\!\big(I-K_{t_n}+K_{t_n}\mathcal{W}_{t_n,t_1}\oM_{t_1}\mathcal{W}_{t_1,t_2}\oM_{t_2}\dotsm\mathcal{W}_{t_{n-1},t_n}\oM_{t_n}\big)_{L^2(X)},
  \end{equation}
  where $\oM_{t_i}=I-\Ml_{t_i}$.
\end{theorem}

\begin{proof}
  The proof is a minor adaptation of the arguments in \cite[Thm. 3.3]{bcr}, and we will use throughout it all the notation and conventions of that proof.
  We will just sketch the proof, skipping several technical details (in particular, we will completely omit the need to conjugate by the operators $U_{t_i}$ and $V_{t_i}$, since this aspect of the proof can be adapted straightforwardly from \cite{bcr}).

  In order to simplify notation throughout the proof we will replace subscripts of the form $t_i$ by $i$, so for example $\mathcal{W}_{i,j}=\mathcal{W}_{t_i,t_j}$.
  Let $\sK=\Ml K^\uptext{ext}$. Then $\sK$ can
  be written as
  \begin{equation}
    \sK=\sQ \bigl(\sW^{-}\sK^{\rm d}+\sW^{+}(\sK^{\rm d}-I)\bigr)\qquad\text{with}\qquad
    \sK^{\rm d}_{ij}=K_{i}\uno{i=j},\quad \sQ_{i,j}=\Ml_{{i}}\uno{i=j},
  \end{equation}
  where $\sW^{-}$, $\sW^{+}$ are lower triangular, respectively strictly upper triangular, and defined by
  \begin{equation*}
    \sW^{-}_{ij} = \mathcal{W}_{i,j}\uno{i\geq j},\qquad
    \sW^{+}_{ij}=\mathcal{W}_{{i},{j}}\uno{i < j}.
  \end{equation*}
  The key to the proof in \cite{bcr} was to observe that $\big[(I+\sW^{+})^{-1})\big]_{i,j}=\uno{j=i}-\mathcal{W}_{{i},{i+1}}\uno{j=i+1}$, which then implies that $\big[(\sW^{-}+\sW^{+})\sK^{\rm d}(I+\sW^{+})^{-1}\big]_{i,j}=\mathcal{W}_{{i},{1}}K_{1}\uno{j=1}$.
  The fact that only the first column of this matrix has non-zero entries is what ultimately allows one to turn the Fredholm determinant of an extended kernel into one of a kernel acting on $L^2(X)$.
  However, the derivation of this last identity uses $\mathcal{W}_{i,{j-1}}K_{j-1}\mathcal{W}_{{j-1},j}=\mathcal{W}_{i,j}K_{j}$, which is a consequence of Assumptions~\ref{assum:2}(ii) and (iii), and thus is not available to us. In our case we may proceed similarly by observing that
  \begin{equation}
    \big[(\sW^{-})^{-1}\big]_{i,j}=\uno{j=i}-\mathcal{W}_{{i},{i-1}}\uno{j=i-1},
  \end{equation}
  as can be checked directly using Assumption~\ref{assum:2}(ii).
  Now using the identity \[\mathcal{W}_{i,{j+1}}K_{j+1}\mathcal{W}_{{j+1},j}=\mathcal{W}_{i,j}K_{j},\] which follows from Assumption~\ref{assum:2}(ii) and (iii), we get
  \begin{equation}
      \big[(\sW^{-}+\sW^{+})\sK^{\rm d}(\sW^{-})^{-1}\big]_{i,j}
      =\mathcal{W}_{i,j}K_{j}-\mathcal{W}_{i,{j+1}}K_{j+1}\mathcal{W}_{{j+1},j}\uno{j<n}
      =\mathcal{W}_{{i},{n}}K_{n}\uno{j=n}.\label{eq:T-T+}
  \end{equation}
  Note that now only the last column of this matrix has non-zero entries, which accounts for the difference between our result and that of \cite{bcr}. 
  To take advantage of \eqref{eq:T-T+} we write 
  \[
  I-\sK=(I+\sQ\sW^+)\big[I-(I+\sQ\sW^+)^{-1}\sQ(\sW^-+\sW^+)\sK^{\rm d}(\sW^-)^{-1}\sW^-\big].
  \]
  Since $\sQ\sW^+$ is strictly upper triangular, $\,\det(I+\sQ\sW^+)=1$, which in particular shows that $I+\sQ\sW^+$ is invertible. Thus by the cyclic property of the Fredholm determinant, $\,\det(I-\sK)=\det(I-\wt\sK)$ with
  \[\wt\sK=\sW^-(I+\sQ\sW^+)^{-1}\sQ(\sW^-+\sW^+)\sK^{\rm d}(\sW^-)^{-1}.\]
  Since only the last column of $(\sW^-+\sW^+)\sK^{\rm d}(\sW^-)^{-1}$ is non-zero, the same holds for $\wt\sK$, and thus $\,\det(I-\sK)=\det(I-\wt\sK_{n,n})_{L^2(X)}$.

  Our goal now is to compute $\wt\sK_{n,n}$.
  From \eqref{eq:T-T+} and Assumption~\ref{assum:2}(ii) we get, for $0\leq k\leq n-i$,
  \begin{multline*}
    \left[(\sQ\sW^+)^k\sQ(\sW^-+\sW^+)\sK^{\rm d}(\sW^-)^{-1}\right]_{i,n}\\
    =\qquad\smashoperator{\sum_{i<\ell_1<\dots<\ell_k\leq n}}\qquad
    \Ml_{{i}}\mathcal{W}_{i,{\ell_1}}\Ml_{{\ell_1}}\mathcal{W}_{{\ell_{1}},{\ell_2}} \dotsm
    \Ml_{{\ell_{k-1}}}\mathcal{W}_{{\ell_{k-1}},{\ell_k}}\Ml_{{\ell_k}}\mathcal{W}_{{\ell_k},{n}}K_{n},
  \end{multline*}
  while for $k>n-i$ the left-hand side above equals 0 (the case $k=0$ is interpreted as
  $\Ml_{i}\mathcal{W}_{i,n}K_{n}$). 
  As in \cite{bcr} this leads to
  \begin{equation}
    \wt\sK_{i,n}
    =\sum_{j=1}^i\sum_{k=0}^{n-j}(-1)^{k}\quad\smashoperator{\sum_{j=\ell_0<\ell_1<\dots<\ell_k\leq
        n}}\quad\mathcal{W}_{i,j}\Ml_{j}\mathcal{W}_{{j},{{\ell_1}}}\Ml_{{\ell_1}}\mathcal{W}_{{\ell_1},{\ell_2}}
    \Ml_{{\ell_{k-1}}}\mathcal{W}_{{\ell_{k-1}},{\ell_k}}\Ml_{{\ell_k}}\mathcal{W}_{{\ell_k},{n}}K_{n}.
  \end{equation}
  Replacing each $\Ml_{\ell}$ by $I-\oM_{\ell}$ except for the first one and simplifying as in \cite{bcr} leads to
  \[\wt\sK_{i,n}=\mathcal{W}_{{i},{i+1}}\oM_{{i+1}}\mathcal{W}_{{i+1},{i+2}}\oM_{{i+2}}
    \dotsm\mathcal{W}_{{{n-1}},{{n}}}\oM_{{{n}}}K_{n}-\mathcal{W}_{{i},{1}}\oM_{{1}}\mathcal{W}_{{1},{2}}\oM_{{2}}\dotsm\mathcal{W}_{{{n-1}},{{n}}}\oM_{{{n}}}K_{n}.\]
  Setting $i=n$ yields $\wt\sK_{n,n}=K_{n}-\mathcal{W}_{{n},{1}}\oM_{{1}}\mathcal{W}_{{1},{2}}\oM_{{2}}\dotsm\mathcal{W}_{{{n-1}},{{n}}}\oM_{{{n}}}K_{n}$
  and then an application of the cyclic property of the determinant gives the result.
\end{proof}

\subsection{Proof of the TASEP path integral formula.}\label{app:proofPathIntTASEP}

To obtain the path integral version \eqref{eq:path-int-kernel-TASEPgem} of the TASEP formula we use Theorem~\ref{thm:alt-extendedToBVP}.
Recall from \eqref{eq:PsiRecursion} that $Q^{n-m}\Psi^{n}_{n-k}=\Psi^{m}_{m-k}$.
Then we can write
\begin{equation}\label{eq:checkQKK}
  Q^{n_j-n_i}K^{(n_j)}_t=\sum_{k=0}^{n_j-1}Q^{n_j-n_i}\Psi^{n_j}_{k}\otimes \Phi^{n_j}_{k}=\sum_{k=0}^{n_j-1}\Psi^{n_i}_{n_i-n_j+k}\otimes \Phi^{n_j}_{k}=K_t(n_i,\cdot;n_j,\cdot)+Q^{n_j-n_i}\uno{n_i<n_j}.
\end{equation}
This means that the extended kernel $K_t$ has exactly the structure specified in \eqref{eq:generalExt}, taking\break$t_i=n_i$, $K_{t_i}=K^{(n_i)}_t$, $\mathcal{W}_{t_i,t_j}=Q^{n_j-n_i}$ and $\mathcal{W}_{t_i,t_j}K_{t_j}=K_t(n_i,\cdot;n_j,\cdot)$.
It is not hard to check that Assumptions 1 and 3 of \cite[Thm. 3.3]{bcr} hold in our setting.
The semigroup property (Assumption~2(ii)) is trivial in this case, while the right-invertibility condition (Assumption~2(i)) 
\[
Q^{n_j-n_i}K_t(n_j,\cdot;n_i,\cdot)=K_t^{(n_i)}
\] 
for $n_i\leq n_j$ follows similarly to \eqref{eq:checkQKK}.
However, Assumption 2(iii) of \cite{bcr}, which translates into $Q^{n_j-n_i}K_t^{(n_j)}=K_t^{(n_i)}Q^{n_j-n_i}$ for $n_i\leq n_j$, does not hold in our case (in fact, the right hand side does not even make sense as the product is divergent, as can be seen by noting that $\Phi^{(n)}_0(x)=2^{x-X_0(n)}$; alternatively, note that the left hand side depends on the values of $X_0(n_{i+1}),\ldots,X_0(n_j)$ but the right hand side does not), which is why we need Theorem~\ref{thm:alt-extendedToBVP}.
To use it, we need to check that
\begin{equation}
  Q^{n_j-n_i}K_t^{(n_j)}Q^{n_i-n_j}=K_t^{(n_i)}.\label{eq:secondExtKernAssmp}
\end{equation}
In fact, if $k\geq0$ then \eqref{bhe1} together with the easy fact that $h^n_k(\ell,z)=h^{n-1}_{k-1}(\ell-1,z)$ imply that 
\[
(\Qt)^{n_i-n_j}h^{n_j}_{k+n_j-n_i}(0,z)=h^{n_j}_{k+n_j-n_i}(n_j-n_i,z)=h^{n_i}_{k}(0,z), 
\]
so that
$(\Qt)^{n_i-n_j}\Phi^{n_j}_{k+n_j-n_i}=\Phi^{n_i}_k$.
On the other hand, if $n_i-n_j\leq k<0$ then we have 
\[
(\Qt)^{n_i-n_j}h^{n_j}_{k+n_j-n_i}(0,z)=(\Qt)^kh^{n_j}_{k+n_j-n_i}(k+n_j-n_i,z)=0
\]
 thanks to \eqref{bhe2} and the fact that $2^z\in\ker(\Qt)^{-1}$, which gives $(\Qt)^{n_i-n_j}\Phi^{n_j}_{k+n_j-n_i}=0$.
Therefore, proceeding as in \eqref{eq:checkQKK}, the left hand side of \eqref{eq:secondExtKernAssmp} equals
\begin{equation}
  \sum_{k=0}^{n_j-1}Q^{n_j-n_i}\Psi^{n_j}_{k}\otimes(\Qt)^{n_i-n_j}\Phi^{n_j}_{k}=\sum_{k=0}^{n_j-1}\Psi^{n_i}_{n_i-n_j+k}\otimes(\Qt)^{n_i-n_j}\Phi^{n_j}_{k}
  =\sum_{k=0}^{n_i-1}\Psi^{n_i}_k\otimes\Phi^{n_i}_{k}
\end{equation}
as desired.

\section{The KPZ fixed point}\label{sec:123}

In this section we will take the KPZ scaling limit of the TASEP growth process, by using the formula from Theorem~\ref{thm:tasepformulas}, and will get a complete characterisation of the limiting Markov process, called the \emph{KPZ fixed point}. We start with introducing the topology and some operators.

\subsection{State space and topology.}
\label{UC} 

The state space on which we define the KPZ fixed point is the following:

\begin{definition}[$\UC$ functions]
We define $\UC$ as the space of upper semicontinuous functions\break$\fh\!:\R \to [-\infty,\infty)$ with $\fh(\fx)\le C(1+|\fx|)$  for some $C<\infty$.
\end{definition}

\begin{example}
The $\UC$ function $\mathfrak{d}_\fu(\fu) = 0$, $\mathfrak{d}_\fu(\fx) = -\infty$ for $\fx\neq \fu$, is known as a \emph{narrow wedge at $\fu$}.
These arise naturally as $\mathfrak{d}_0$ is clearly the limit  of the TASEP height function $h(x)=-|x|$ under the rescaling $\ep^{1/2}h(\ep^{-1}x)$.
\end{example}

We will endow this space with the topology of local $\UC$ convergence.
This is the natural topology for lateral growth, and will allow us to compute the KPZ limit in all cases of interest\footnote{Actually the bound $\fh(\fx)\le C(1+|\fx|)$ which we are imposing here and in~\cite{fixedpt} on $\UC$ functions is not as general as possible, but makes the arguments a bit simpler and it suffices for most cases of interest (see also~\cite[Foot. 9]{fixedpt}).}.
In order to define this topology, recall that $\fh$ is upper semicontinuous ($\UC$) if and only if its \emph{hypograph} 
\[\hypo(\fh) = \{(\fx,\fy): \fy\le \fh(\fx)\}\]
is closed in $[-\infty,\infty)\times \R$.
Slightly informally, local $\UC$ convergence can be defined as follows:

\begin{definition}[Local $\UC$ convergence]
We say that $(\fh_\ep)_\ep\subseteq\UC$ \emph{converges locally in $\UC$} to $\fh\in\UC$ if there is a $C>0$ such that $\fh_\ep(\fx) \le C(1+|\fx|)$ for all $\ep>0$ and for every $M\geq1$ there is a $\delta=\delta(\ep,M)>0$ going to 0 as $\ep\to0$ such that the hypographs $\mathfrak{H}_{\ep,M}$ and $\mathfrak{H}_{M}$ of $\fh_\ep$ and $\fh$ restricted to $[-M,M]$ are $\delta$-close in the sense that 
\[\cup_{(\ft,\fx)\in\mathfrak{H}_{\ep,M}}B_{\delta}((\ft,\fx))\subseteq\mathfrak{H}_{M}\qqand\cup_{(\ft,\fx)\in\mathfrak{H}_{M}}B_{\delta}((\ft,\fx))\subseteq\mathfrak{H}_{\ep,M}.\]
\end{definition}

We will also use an analogous space $\LC$, made of lower semicontinuous functions:

\begin{definition}[$\LC$ functions and local convergence]
We define $\LC=\big\{\fg\!: -\fg\in\UC\!\big\}$ and endow this space with the topology of \emph{local $\LC$ convergence} which is defined analogously to local $\UC$ convergence, now in terms of \emph{epigraphs}, 
\[\epi(\fg) = \{(\fx,\fy): \fy\ge \fg(\fx)\}.\]
Explicitly, $(\fg_\ep)_{\ep} \subset \LC$ converges locally in $\LC$ to $\fg \in \LC$ if and only if $-\fg_\ep\rightarrow-\fg$ locally in $\UC$.
\end{definition}

\begin{exercise}
Show that if $\fh$ is locally H\"older $\beta\in (0,1)$ then convergence in $\UC$ or $\LC$ implies uniform convergence on compact sets.
\end{exercise}
\begin{exercise}$\!\!\!\!{}^*$
Show that if $\fh_0\in \LC$ then the inviscid solution given by the Hopf-Lax formula
\[ h(t,x)=\sup_y\left\{ h_0(y)- t^{-1}(x-y)^2\right\}\]
is continuous in the $\UC$ topology.
\end{exercise}

\subsection{Auxiliary operators.}

In order to state our main result we need to introduce several operators, which will appear in the explicit Fredholm determinant formula for the fixed point. Our basic building block is the following (almost) group of operators:

\begin{definition}\label{def:groupS}
For $\fx,\ft\in\R^2\setminus \{\fx<0, \ft= 0\}$ let us define the operator
\begin{equation}\label{eq:groupS}
\fT_{\ft,\fx}=\exp \bigl\{ \fx\partial^2+\tfrac{\ft}3\tts\partial^3 \bigr\},
\end{equation}
which satisfies the identity
\begin{equation}\label{eq:groupS2}
\fT_{\fs,\fx}\fT_{\ft,\fy}=\fT_{\fs+\ft,\fx+\fy}
\end{equation}
as long as all subscripts avoid the region $\{\fx<0, \ft= 0\}$.
\end{definition}

For $\ft>0$ the operator $\fT_{\ft,\fx}$ acts on nice functions by convolution with the kernel
\begin{equation}\label{eq:fTdef}
\fT_{\ft,\fx}(z)=\frac1{2\pi\I} \int_{\langle}\, dw\,e^{\frac{\ft}3 w^3+\fx  w^2-z w} = \ft^{-1/3} e^{\frac{2 \fx^3}{3\ft^2} -\frac{z\fx}{\ft} }\,\Ai \bigl(-\ft^{-1/3} z+\ft^{-4/3}\fx^2\bigr),
\end{equation} 
where ${\langle}~$ is the positively oriented contour going in straight lines from $e^{-\I\pi/3}\infty$ to $e^{\I\pi/3}\infty$ through $0$ and $\Ai$ is the Airy function 
\[
\Ai(z)= \frac1{2\pi\I} \int_{\langle}dw\, e^{\frac{1}3 w^3-z w}.
\]
When $\ft<0$ we have the identity $\fT_{\ft,\fx}=(\fT_{-\ft,\fx})^*$, which in particular yields
\begin{equation}
  (\fT_{\ft,\fx})^*\fT_{\ft,-\fx}=\fI.
\end{equation}

\begin{exercise}
 Prove that this identity indeed holds.
\end{exercise}

\begin{definition}[Hit operators]\label{def:hit}
For $\fg\in \LC$ let us define
\begin{equation}
\bar{\fT}^{\epi(\fg)}_{\ft,\fx}(v,u)=\ee_{\fB(0)=v}\big[\fT_{\ft,\fx-\ftau}(\fB(\ftau),u)\1{\ftau<\infty}\big]=\int_{0}^\infty\,\pp_{\fB(0)=v}(\ftau\in d\fs)\,\fT_{\ft,\fx-\fs}(\fB(\ftau),u)
\end{equation}
where $\fB(x)$ is a Brownian motion with diffusion coefficient $2$ and $\ftau$ is the hitting time of the epigraph of the function $\fg$.
\end{definition}

Note that for $v\ge\fg(0)$ we trivially have the identity
\begin{equation}\label{eq:epiabove}
\bar{\fT}^{\epi(\fg)}_{\ft,\fx}(v,u)=\fT_{\ft,\fx}(v,u).
\end{equation}
If $\fh\in \UC$, there is a similar operator $\bar{\fT}^{\hypo(\fh)}_{\ft,\fx}$, except that now $\ftau$ is the hitting time of the  hypograph of $\fh$ and $\bar{\fT}^{\hypo(\fh)}_{\ft,\fx}(v,u)=\fT_{\ft,\fx}(v,u)$ for $v\le\fh(0)$.

One way to think of $\bar{\fT}^{\epi(\fg)}_{\ft,\fx}(v,u)$ is as a sort of asymptotic transformed transition `probability' for the Brownian motion $\fB$ to go from $v$ to $u$ hitting the epigraph of $\fg$ (note that $\fg$ is not necessarily continuous, so hitting $\fg$ is not the same as hitting $\epi(\fg)$; in particular, $\fB(\ftau)\geq\fg(\ftau)$ and in general the equality need not hold).
To see what we mean, write
\begin{equation}\label{eq:asymptTransTransProb}
\bar{\fT}^{\epi(\fg)}_{\ft,\fx}=\lim_{\mathbf{T}\to\infty}\bar{\fT}^{\epi(\fg),\mathbf{T}}\fT_{\ft,\fx-\mathbf{T}}
\quad\text{with}\quad\bar{\fT}^{\epi(\fg),\mathbf{T}}(v,u)=\ee_{\fB(0)=v}\big[\fT_{0,\mathbf{T}-\ftau}(\fB(\ftau),u)\1{\ftau\leq\mathbf{T}}\big]
\end{equation}
and note that $\bar{\fT}^{\epi(\fg),\mathbf{T}}(v,u)$ is nothing but the transition probability for $\fB$ to go from $v$ at time 0 to $u$ at time $\mathbf{T}$ hitting $\epi(\fg)$ in $[0,\mathbf{T}]$.

\begin{definition}[Brownian scattering operator]\label{def:epi}
For $\fg\in \LC$, $\fx\in \R$ and $\ft>0$ we define
\begin{equation}
\fK^{\epi(\fg)}_{-\ft}= \fI - \bigl(\fT_{-\ft,\fx} - \bar{\fT}^{\epi(\fg^-_\fx)}_{-\ft,\fx} \bigr)^* \bar\P_{\fg(\fx)}
\bigl(\fT_{-\ft,-\fx}-\bar{\fT}^{\epi(\fg^+_\fx)}_{-\ft,-\fx}\bigr),
\end{equation}
where $\fg^+_\fx(\fy)=\fg(\fx+\fy)$ and $\fg^-_\fx(\fy)=\fg(\fx-\fy)$.

\begin{exercise} 
The projection $\bar\P_{\fg(\fx)}$ can be removed from the formula without changing its meaning.
\end{exercise}

\begin{exercise}$\!\!\!\!{}^*$
Show that the kernel $\fK^{\epi(\fg)}_{-\ft}$ does not depend on $\fx$.
\end{exercise}

\noindent There is another operator which uses $\fh\in \UC$, and hits `from above',
\begin{equation}\label{eq:Kepihypo}
\fK^{\hypo(\fh)}_{\ft}= \fI - \bigl(\fT_{\ft,\fx} - \bar{\fT}^{\hypo(\fh^-_\fx)}_{\ft,\fx} \bigr)^*\P_{\fh(\fx)} \bigl(\fT_{\ft,-\fx}-\bar{\fT}^{\hypo(\fh^+_\fx)}_{\ft,-\fx}\bigr),
\end{equation}
\end{definition}

\begin{exercise} Show that the following identity holds
\begin{equation}\label{skew}
\fK^{\hypo(\fh)}_{\ft}=\big(\varrho\,\fK^{\epi(-\varrho\fh)}_{-\ft}\,\varrho\big)^*,
\end{equation}
where $\varrho\fh(\fx) = \fh(-\fx)$.
\end{exercise}

As above, $\fT_{\ft,\fx} - \bar{\fT}^{\epi(\fg_\fx)}_{\ft,\fx}$ may be thought of as a sort of asymptotic transformed transition probability for a Brownian motion $\fB$, started at time 0, not to hit $\epi(\fg)$. Therefore $\fK^{\epi(\fg)}_{\ft}$ may be thought of as the same sort of asymptotic transformed transition probability for $\fB$, in this case hitting $\epi(\fg)$, which is built out of the product of left and right `no hit' operators.

\subsection{The KPZ fixed point formula.} 
\label{sec:KPZfp}

At this stage we are ready to state our main result which we prove in Section~\ref{sec:TASEPlimit}.

\begin{definition}[The KPZ fixed point formula] 
The KPZ fixed point is the Markov process on $\UC$ with transition probabilities 
\begin{equation}\label{eq:fpform}
 \pp_{\fh_0} \bigl( \fh(\ft, \fx) \le \fg(\fx),\,\fx\in\R\bigr)  = \det\left(\fI-\fK^{\hypo(\fh_0)}_{\ft/2} \fK^{\epi(\fg)}_{-\ft/2}\right)_{L^2(\R)},
\end{equation}
where $\fh_0\in\UC$ and $\fg\in \LC$. Here $ \pp_{\fh_0}$ means the process with initial data $\fh_0$.
\end{definition}

\begin{remark}
The fact that the Fredholm determinant in the formula is finite is a consequence of the fact that there is a (multiplication) operator $M$ such that 
the map $(\fh_0,\fg)\mapsto M\fK^{\hypo(\fh_0)}_{\ft/2} \fK^{\epi(\fg)}_{-\ft/2}M^{-1}$ is continuous from $\UC\times\LC$ into the trace class (see
\cite{KPZ}).  We will not get into such issues in these notes.
\end{remark}

\begin{exercise}[Finite dimensional distributions]  Let $\fh_0\in \UC$ and $\fx_1<\fx_2<\dotsm<\fx_M$. Show that 
\begin{align}
&\pp_{\fh_0} \bigl(\fh(\ft,\fx_1)\leq\fa_1,\ldots,\fh(\ft,\fx_M)\leq\fa_M\bigr)\nonumber\\
&\hspace{0.3in}=\det\left(\fI-\fK^{\hypo(\fh_0)}_{\ft,\fx_M}+\fK^{\hypo(\fh_0)}_{\ft,\fx_M} e^{(\fx_1-\fx_M)\p^2}\bar\P_{\fa_1}e^{(\fx_2-\fx_1)\p^2}\bar\P_{\fa_2}\dotsm e^{(\fx_M-\fx_{M-1})\p^2}\bar\P_{\fa_M}\right)_{L^2(\R)}.\label{eq:twosided-path}
\end{align}
\end{exercise}

\begin{remark}[Extended kernels]
The formula in the exercise can be rewritten as
\begin{equation}\label{eq:twosided-ext}
\pp_{\fh_0} \bigl(\fh(\ft,\fx_1)\leq\fa_1,\ldots,\fh(\ft,\fx_M)\leq\fa_M\bigr)=\det\left(\fI-\chi_{\fa}\fK^{\hypo(\fh_0)}_{\ft,\uptext{ext}}\chi_{\fa}\right)_{L^2(\{\fx_1,\ldots,\fx_M\}\times\R)},
\end{equation}
where
\begin{align}
\fK^{\hypo(\bh_0)}_{\ft,\uptext{ext}}(\fx_i,\cdot;\fx_j,\cdot)&=-e^{(\fx_j-\fx_i)\p^2}\uno{\fx_i<\fx_j}+e^{-\fx_i\p^2}\fK^{\hypo(\fh_0)}_\ft e^{\fx_j\p^2}.\label{eq:Khypo-ext}
\end{align}

The kernel in \eqref{eq:twosided-ext} is usually referred to as an \emph{extended kernel} (note that the Fredholm determinant is being computed on the `extended $L^2$ space' $L^2(\{\fx_1,\ldots,\fx_M\}\times\R)$).
The kernel appearing after the second hypo operator in \eqref{eq:twosided-path} is sometimes referred to as a \emph{path integral kernel} \cite{bcr}, and should be thought of as a discrete, pre-asymptotic version of the epi operators (on a finite interval).

The fact that $e^{-\fx\p^2}\fK^{\hypo(\fh_0)}_{\ft}e^{\fx\p^2}$ makes sense is not entirely obvious, but follows from the fact that
$\fK^{\hypo(\fh_0)}_{\ft}$ equals
\begin{equation}\label{eq:Kep-expansion}
 (\fT_{\ft,\fx} )^*\P_{\fg(\fx)}\fT_{\ft,-\fx} + ( \bar{\fT}^{\epi(\fg^-_\fx)}_{\ft,\fx} )^*\bar\P_{\fg(\fx)}
\fT_{\ft,-\fx} +  (\fT_{\ft,\fx} )^* \bar\P_{\fg(\fx)}
\bar{\fT}^{\epi(\fg^+_\fx)}_{\ft,-\fx}
 - ( \bar{\fT}^{\epi(\fg^-_\fx)}_{\ft,\fx} )^* \bar\P_{\fg(\fx)}
\bar{\fT}^{\epi(\fg^+_\fx)}_{\ft,-\fx},
\end{equation}
together with the group property \eqref{eq:groupS2}  and the definition of the hit operators.
\end{remark}

\begin{example}[Airy processes] For special initial data, at time $\ft=1$ we recover the known Airy$_1$, Airy$_2$ and Airy$_{2\to1}$ processes:  
\begin{enumerate}
\item Narrow wedge initial data leads to the \emph{Airy$_2$ process} \cite{prahoferSpohn,johansson}:
\[\fh(1,\fx;\mathfrak{d}_\fu) + (\fx-\fu)^2~=~ \aip_2(\fx);\]
\item \emph{Flat} initial data $\fh_0\equiv0$ leads to the \emph{Airy$_1$ process} \cite{sasamoto,borFerPrahSasam}:
\[\fh(1,\fx;0)~=~ \aip_1(\fx);\]
\item \emph{Wedge} or \emph{half-flat} initial data $\fh_{\text{h-f}}(\fx) = -\infty$ for $\fx<0$ and $\fh_{\text{h-f}}(\fx)=0$ for $\fx\geq0$, leads to the \emph{Airy$_{2\to1}$ process} \cite{bfs}:
\[\fh(1,\fx;\fh_{\text{h-f}})+\fx^2\uno{\fx<0}~=~ \Bt(\fx).\]
\end{enumerate}
We leave the derivation of the processes $\aip_1$ and $\aip_2$ as Exercise~\ref{eq:getAiry}. To get the formula in the case 3 we need to show that the finite dimensional distributions match, by computing the kernel on the right hand side of \eqref{eq:twosided-ext} with $\fh_0^-\equiv-\infty$ and $\fh_0^+(\fx)\equiv0$.  
 It is straightforward to check that $\bar{\fT}^{\hypo(\fh_0^-)}_{\ft,0}\equiv0$.
On the other hand, an application of the reflection principle based on \eqref{eq:asymptTransTransProb}  yields that, for $v\geq0$,
\[\bar{\fT}^{\hypo(\fh_0^+)}_{\ft,0}(v,u)=\int_0^\infty\!\pp_v(\tau_0\in d\fy)\fT_{\ft,-\fy}(0,u)=\fT_{\ft,0}(-v,u),\]
which gives, with $\varrho$ the reflection operator $\varrho f(x)=f(-x)$,
\[\fK^{\hypo(\fh_0)}_{\ft}=\fI-(\fT_{\ft,0})^*\P_0[\fT_{\ft,0}-\varrho\fT_{\ft,0}]=(\fT_{\ft,0})^*(\fI+\varrho)\bar \P_0\fT_{\ft,0}.\]
This yields, using \eqref{eq:Khypo-ext},
\begin{align}
\fK^{\hypo(\fh_{\text{h-f}})}_{\ft,\text{ext}}(\fx_i,\cdot;\fx_j,\cdot)&=-e^{(\fx_j-\fx_i)\p^2}\uno{\fx_i<\fx_j}+\fT_{0,-\fx_i}(\fT_{\ft,0})^*(\fI+\varrho)\bar \P_0\fT_{\ft,0}\fT_{0,\fx_i}\\
&=-e^{(\fx_j-\fx_i)\p^2}\uno{\fx_i<\fx_j}+(\fT_{\ft,-\fx_i})^*(\fI+\varrho)\bar \P_0\fT_{\ft/2,\fx_i}.
\end{align}
Choosing $\ft=1$ and using \eqref{eq:fTdef} yields that the second term on the right hand side equals
\begin{multline}
\fK^{\hypo(\fh_{\text{h-f}})}_{\ft,\text{ext}}(\fx_i,u;\fx_j,v)
=\int_{-\infty}^0d\lambda\,e^{-2\fx_i^3/3-\fx_i(u-\lambda)}\,\Ai(u-\lambda+\fx_i^2)\, e^{2\fx_j^3/3+\fx_j(v-\lambda)}\,\Ai(v-\lambda+\fx_j^2)\\
+\int_{-\infty}^0d\lambda\,e^{-2\fx_i^3/3-\fx_i(u+\lambda)}\,\Ai(u+\lambda+\fx_i^2)\, e^{2\fx_j^3/3+\fx_j(v-\lambda)}\,\Ai(v-\lambda+\fx_j^2)
\end{multline}
which, after a simple conjugation, gives the kernel for the Airy$_{2\to1}$ process.
\end{example}

\begin{exercise}  
Show that a conjugation gives the kernel for the Airy$_{2\to1}$ process (see \cite{bfs}).
\end{exercise}

\begin{example}[The Airy$_2$ process]  
Given $\fx_1,\dots,\fx_n\in\mathbb{R}$ and $\fa_1<\dots<\fa_n$ in $\mathbb{R}$, we have
\begin{equation}
  \mathbb{P}\bigl(\aip_2(\fx_1)\le \fa_1,\dots,\aip_2(\fx_n)\le \fa_n\bigr) =
  \det \bigl(I-\chi_{\fa}\K^{\mathrm{ext}} \chi_{\fa}\bigr)_{L^2(\{\fx_1,\dots,\fx_n\}\times\mathbb{R})},
\end{equation}
where the {\it extended Airy kernel} is defined by
\begin{equation}
  \K^\mathrm{ext}(\fx,u;\fx',u')=
  \begin{cases}
    \int_0^\infty d\lambda\,e^{-\lambda(\fx-\fx')}\Ai(u+\lambda)\Ai(u'+\lambda), &\text{if $\fx\ge \fx'$},\\
    -\int_{-\infty}^0 d\lambda\,e^{-\lambda(\fx-\fx')}\Ai(u+\lambda)\Ai(u'+\lambda), &\text{if $\fx<\fx'$}.
  \end{cases}\label{eq:extAiry}
\end{equation}
\end{example}

\begin{example}[The Airy$_1$ process] 
For $\fx_i$ and $\fa_i$ as in the previous example, we have the identity
\begin{equation}
  \mathbb{P}\bigl(\aip_1(\fx_1)\le \fa_1,\dots,\aip_1(\fx_n)\le \fa_n\bigr) = 
  \det\big(I-\chi_{\fa}K^{\mathrm{ext}}_1 \chi_{\fa}\big)_{L^2(\{\fx_1,\dots,\fx_n\}\times\mathbb{R})},
\end{equation}
with the kernel
\begin{multline}\label{eq:fExtAiry1}
  K^{\rm ext}_1(\fx,u;\fx',u')=-\frac{1}{\sqrt{4\pi
      (\fx'-\fx)}}\exp\!\left(-\frac{(u'-u)^2}{4 (\fx'-\fx)}\right)\uno{\fx'>\fx}\\
  +\Ai(u+u'+(\fx'-\fx)^2) \exp\!\left((\fx'-\fx)(u+u')+\tfrac23(\fx'-\fx)^3\right).
\end{multline}
\end{example}

\begin{exercise}\label{eq:getAiry}
 Obtain the kernels for the Airy$_2$ and Airy$_1$ processes from the formula \eqref{eq:fpform}.
\end{exercise}

\subsection{Symmetries and invariance.}

The KPZ fixed point inherits several nice properties as a scaling limit of TASEP. We will write $\fh(\ft, \fx;\fh_0)$ for the KPZ fixed point $\fh(\ft, \fx)$ started at $\fh_0$.

\begin{proposition}[Symmetries of $\fh$]\label{symmetries}
The KPZ fixed point $\fh$ has the following properties:
\begin{enumerate}[label=\emph{(\roman*)}]
\item (1:2:3 scaling invariance)\hskip0.1in
$\alpha\tts\fh(\alpha^{-3}\ft,\alpha^{-2}\fx;\alpha\fh_0(\alpha^{-2}\fx) )\stackrel{\uptext{dist}}{=} \fh(\ft, \fx;\fh_0), \quad \alpha>0$;
\item  (Skew time reversal) \hskip0.1in
$\pp\big(\fh(\ft,\fx; \fg)\le -\ff(\fx)\big) =\pp\big(\fh(\ft,\fx;\ff)\le -\fg(\fx)\big),\quad\ff,\fg\in\UC;$
\item (Shift invariance) \hskip0.1in$\fh(\ft,\fx+\fu; \fh_0(\fx+\fu))\stackrel{\uptext{dist}}{=} \fh(\ft,\fx; \fh_0);$
\item  (Reflection invariance) \hskip0.1in$\fh(\ft, -\fx;\fh_0(-\fx))\stackrel{\uptext{dist}}{=} \fh(\ft,\fx; \fh_0);$
\item (Affine invariance)  \hskip0.1in$\fh(\ft,\fx; \ff(\fx) + \fa + c\fx)\stackrel{\uptext{dist}}{=} \fh(\ft,\fx; \ff(\fx+\frac12c\ft)) + \fa + c\fx + \frac14c^2\ft;$
\item (Preservation of max) \hskip0.1in$\fh(\ft,\fx;\ff_1\vee \ff_2)= \fh(\ft,\fx; \ff_1)\vee \fh(\ft,\fx; \ff_2)$.
\end{enumerate}
\end{proposition}

\begin{exercise}  Prove property (i) from the fixed point formula.  Prove that properties (ii), (iii),(iv) hold for TASEP, and therefore for the limiting
fixed point.  
\end{exercise} 

\noindent Properties (v) and  (vi) require coupling and we provide their proves in Section~\ref{sec:variational}.

\subsection{Markov property.}

The sets $A_{\mathfrak{g}} = \bigl\{ \fh\in \UC:  \fh(\fx) \le \fg(\fx),\,\fx\in\R\bigr\}$ with $\fg\in\LC$ form a generating family for the Borel sets $\mathcal{B}(\UC)$.
Hence from \eqref{eq:fpform} we can define the fixed point transition probabilities $p_{\fh_0}\,(\ft, A_{\mathfrak{g}})$, for which we have:

\begin{lemma}
For fixed $\fh_0\in\UC$ and $\ft>0$, the measure $p_{\fh_0}\,(\ft, \cdot)$ is a probability measure on $\UC$.
\end{lemma}

\begin{proof}[Sketch of the proof]
It is clear from the construction that $\pp_{\fh_0}( \fh(\ft, \fx_i) \le \fa_i, i=1,\ldots,n)$ is non-decreasing in each $\fa_i$ and is in $[0,1]$.
We need to show then that this quantity goes to $1$ as all $\fa_i$'s go to infinity, and to $0$ if any $\fa_i$ goes to $-\infty$.
The first one is standard, and relies on the inequality $\left|\det(\fI-\fK) - 1\right|\leq\|\fK\|_1e^{\|\fK\|_1+1}$ (with $\|\cdot\|_1$ denoting trace norm).
The second limit is in general very hard to show for a formula given in terms of a Fredholm determinant, but it turns out to be rather easy in our case, because the multipoint probability is trivially bounded by $\pp_{\fh_0}( \fh(\ft, \fx_i) \le \fa_i)$, for any $i$.  By the skew time reversal symmetry this becomes the probability that the Airy$_2$ process minus a parabola is bounded everywhere by $-\fh_0+\fa_i$, which clearly
goes to $0$ as $\fa_i$ goes to $-\infty$.
\end{proof}

\begin{theorem}\label{thm:markovprop}
  The KPZ fixed point $\big(\fh(\ft,\cdot)\big)_{\ft>0}$ is a (Feller) Markov process taking values in $\UC$.
\end{theorem}

\noindent The proof is based on the fact that $\fh(\ft,\fx)$ is the limit of $\fh^\ep(\ft,\fx)$, which is Markovian.
To derive from this the Markov property of the limit requires some compactness, which in our case is provided by Theorem~\ref{reg} below.

\subsection{Regularity and local Brownian behavior.}

Up to this point we only know that the fixed point is in $\UC$, but by the smoothing mechanism inherent to models in the KPZ class one should expect $\fh(\ft,\cdot)$
to at least be continuous for each fixed $\ft>0$.
The next result shows that for every $M>0$, $\fh(\ft,\cdot)\big|_{[-M,M]}$ is H\"older-$\beta$ for any $\beta<1/2$ with probability 1.

\begin{definition}[Local H\"older spaces]
Let us define the space 
\[
\mathscr C= \bigl\{ \fh\!:\R \to [-\infty,\infty) ~\text{continuous with}~\fh(\fx)\le C(1+|\fx|)~\text{for some}~ C<\infty\bigr\}.
\] 
For $M > 0$ we define the local H\"older norm
 \begin{equation}
 \| \fh \|_{\beta, [-M,M]} = \sup_{\fx_1\neq \fx_2 \in [-M,M]} \frac{ |\fh(\fx_2)-\fh(\fx_1)|}{|\fx_2-\fx_1|^\beta}
 \end{equation}
 and the local H\"older spaces
\[\mathscr C^\beta= \bigcap_{M \in \N} \bigl\{ \fh\in \mathscr C : \| \fh \|_{\beta, [-M,M]} <\infty\bigr\}.\]
The topology on $\UC$, when restricted to $\mathscr C$, is the topology of uniform convergence on compact sets.
$\UC$ is a Polish space and the spaces $\mathscr C^\beta$ are compact in $\UC$.
\end{definition} 

Then we can get spatial regularity of the KPZ fixed point.

\begin{theorem}[Space regularity]\label{reg}
Fix $\ft>0$, $\fh_0\in \UC$ and initial data $h_0^{(\ep)}$ for the TASEP height function such that $\fh_0^\ep\rightarrow\fh_0$ locally in $\UC$ as $\ep\to0$.
Then for each $\beta\in (0,1/2)$ and $M > 0$ we have
\begin{equation}
\lim_{A\to \infty} \limsup_{\ep\to 0} \pp \bigl( \| \fh^\ep(\ft)\|_{\beta, [-M,M]}\ge A\bigr) = \lim_{A\to \infty} \pp\bigl( \| \fh\|_{\beta, [-M,M]}\ge A\bigr) =0.
\end{equation}
\end{theorem}

\noindent 
The proof proceeds through an application of the Kolmogorov continuity theorem, which reduces regularity to two-point functions, and depends heavily on the representation \eqref{eq:twosided-path} for the two-point function in terms of path integral kernels. We prefer to skip the details.

\begin{remark}
Since the theorem shows that this regularity holds uniformly (in $\ep>0$) for the approximating $\fh^{\ep}(\ft,\cdot)$'s, we get the compactness needed for the proof of the Markov property.
\end{remark}

\begin{theorem}[Local Brownian behavior]
For any initial condition $\fh_0\in\UC$ the KPZ fixed point $\fh$ is locally Brownian in space in the sense that for each $\fy\in\R$, the finite dimensional distributions of 
\[\mathfrak{b}^{\pm}_\ep(\fx)= \ep^{-1/2} \bigl(\fh(\ft,\fy \pm \ep \fx)-\fh(\ft,\fy) \bigr)\]
converge, as $\ep\to 0$, to those of Brownian motions with diffusion coefficient $2$.
\end{theorem}

\begin{proof}[A very brief sketch of the proof]
The proof is based again on the arguments of \cite{quastelRemAiry1}.
One uses \eqref{eq:twosided-path} and Brownian scale invariance to show that
\[
   \pp \bigl(\fh(\ft,\ep\fx_1)\leq \fu+\sqrt{\ep}\fa_1 , \ldots, \fh(\ft,\ep\fx_n)\leq \fu+\sqrt{\ep}\fa_n\, \big|\,\fh(\ft,0)=\fu\bigr) =\ee\left(\uno{\fB(\fx_i)\le\fa_i, i=1,\ldots,n}\,\phi^\ep_{\fx,\fa}(\fu,\fB(\fx_n))\right),
\]
for some explicit function $\phi^\ep_{\fx,\fa}(\fu,\mathbf{b})$.
The Brownian motion appears from the product of heat kernels in \eqref{eq:twosided-path}, while $\phi^\ep_{\fx,\fa}$ contains the dependence on everything else in the formula (the Fredholm determinant structure and $\fh_0$ through the hypo operator $\fK^{\hypo(\fh_0)}_\ft$).
Then one shows that $\phi^\ep_{\fx,\fa}(\fu,\mathbf{b})$ goes to 1 in a suitable sense as $\ep\to0$.
\end{proof}

\begin{proposition}[Time regularity]\label{holder}
Fix $\fx_0\in \R$ and initial data $\fh_0\in \UC$.  For $\ft>0$, $ \fh(\ft,\fx_0)$ is locally H\"older $\alpha$ in $\ft$ for any $\alpha<1/3$.  
\end{proposition} 

\noindent The proof uses the variational formula for the fixed point, we sketch it in the next section.

\subsection{Variational formulas and the Airy sheet}\label{sec:variational}

\begin{definition}[Airy sheet]  The two parameter process 
\[\aip(\fx,\fy) ~ = ~ \fh(1,\fy; \mathfrak{d}_\fx)+ (\fx-\fy)^2 \] 
is called the \emph{Airy sheet} \cite{cqrFixedPt}.  
Fixing either one of the variables, it is an Airy$_2$ process in the other.  
We also write
\[\hat\aip(\fx,\fy) = \aip(\fx,\fy)-  (\fx-\fy)^2.\]
\end{definition} 

\begin{remark}
The KPZ fixed point inherits from TASEP a canonical coupling between the processes started with different initial data (using the same `noise').
It is this the property that allows us to define the two-parameter Airy sheet.
An annoying difficulty is that we cannot prove that this process is unique.
More precisely, the construction of the Airy sheet in \cite{fixedpt} goes through using tightness of the coupled processes at the TASEP level and taking subsequential limits, and at the present time there seems to be no way to assure that the limit points are unique. 
This means that we have actually constructed `an' Airy sheet, and the statements below should really be interpreted as about any such
limit.

\noindent It is natural to wonder whether the fixed point formulas at our disposal determine the joint probabilities $\pp(\aip(\fx_i,\fy_i)\le \fa_i, i=1,\ldots,M)$ for the Airy sheet.
Unfortunately, this is not the case.
In fact, the most we can compute using our formulas is \[\pp\bigl(\hat\aip(\fx,\fy)  \le \ff(\fx)+\fg(\fy),\,\fx,\fy\in\R\bigr)  = \det\left(\fI-\fK^{\hypo(-\fg)}_{1} \fK^{\epi(\ff)}_{-1}\right).\]
Suppose we want to compute the two-point distribution for the Airy sheet $\pp(\hat\aip(\fx_i,\fy_i)  \le \fa_i,\,i=1,2)$ from this.
We would need to choose $\ff$ and $\fg$ taking two non-infinite values, which yields a formula for $\pp(\hat\aip(\fx_i,\fy_j)  \le \ff(\fx_i)+\fg(\fy_j),\,i,j=1,2)$, and thus we need to take $\ff(\fx_1)+\fg(\fy_1)=\fa_1$, $\ff(\fx_2)+\fg(\fy_2)=\fa_2$ and $\ff(\fx_1)+\fg(\fy_2)=\ff(\fx_2)+\fg(\fy_1)=L$ with $L\to\infty$.
But $\{\ff(\fx_i)+\fg(\fy_j),\,i,j=1,2\}$ only spans a 3-dimensional linear subspace of $\R^4$, so this is not possible.
\end{remark}

The preservation of max property allows us to write a  variational formula for the KPZ fixed point in terms of the Airy sheet.
\begin{theorem}[Airy sheet variational formula]\label{thm:airyvar}
One has the identities
\begin{equation}\label{eq:var}
\fh(\ft,\fx;\fh_0) = \sup_\fy\big\{ \fh(\ft,\fx;\mathfrak{d}_\fy) + \fh_0(\fy)\big\}  \stackrel{\uptext{dist}}{=}  \sup_\fy\Big\{ \ft^{1/3}\hat\aip(\ft^{-2/3} \fx,\ft^{-2/3} \fy) + \fh_0(\fy)\Big\}.
\end{equation}
In particular, the Airy sheet satisfies the \emph{semi-group property}: If $\hat\aip^1$ and $\hat\aip^2$ are independent copies and $\ft_1+\ft_2=\ft$ are all positive, then
\begin{equation*}
 \sup_\fz\left\{ \ft_1^{1/3}\hat\aip^1(\ft_1^{-2/3} \fx,\ft_1^{-2/3} \fz) +  \ft_2^{1/3}\hat\aip^2(\ft_2^{-2/3} \fz,\ft_2^{-2/3} \fy) \right\} \stackrel{\uptext{dist}}{=} \ft^{1/3}\hat\aip^1(\ft^{-2/3} \fx,\ft^{-2/3} \fy).
\end{equation*}
\end{theorem}

\begin{proof}
Let $\fh_0^{n}$ be a sequence of initial conditions taking finite values $\fh^n_0(\fy^n_i)$ at $\fy^n_i$, $i=1,\ldots, k_n$, and $-\infty$ everywhere else, which converges to $\fh_0$ in $\UC$ as $n\to\infty$.
By repeated application of Proposition~\ref{symmetries}(v) (and the easy fact that $\fh(\ft,\fx;\fh_0+\fa)=\fh(\ft,\fx;\fh_0)+\fa$ for $\fa\in\R$) we get
\[\fh(\ft,\fx;\fh^n_0)=\sup_{i=1,\ldots,k_n}\big\{\fh(\ft,\fx;\mathfrak{d}_{\fy^n_i})+\fh^n_0(\fy^n_i)\big\},\]
and taking $n\to\infty$ yields the result (the second equality in \eqref{eq:var} follows from scaling invariance, Proposition~\ref{symmetries}).
\end{proof}

One of the interests in this variational formula is that it leads to proofs of properties of the fixed point, as we already mentioned in earlier sections.

\begin{proof}[Proof of Proposition~\ref{symmetries}(iv)]
The fact that the fixed point is invariant under translations of the initial data is straightforward, so we may assume $\fa=0$.
By Theorem~\ref{thm:airyvar} we have
\begin{align*}
\fh(\ft,\fx;\fh_0+c\fx)&\stackrel{\uptext{dist}}{=}\sup_\fy\Big\{ \ft^{1/3}\aip(\ft^{-2/3} \fx,\ft^{-2/3} \fy)-\ft^{-1}(\fx-\fy)^2 + \fh_0(\fy)+c\fy\Big\}\\
&\stackrel{\hphantom{\uptext{dist}}}{=}\sup_\fy\Big\{ \ft^{1/3}\aip(\ft^{-2/3} \fx,\ft^{-2/3}(\fy+c\ft/2))-\ft^{-1}(\fx-\fy)^2 + \fh_0(\fy+c\ft/2)+c\fx+c^2\ft/4\Big\}\\
&\stackrel{\uptext{dist}}{=}\sup_\fy\Big\{ \ft^{1/3}\hat\aip(\ft^{-2/3} \fx,\ft^{-2/3}\fy)+\fh_0(\fy+c\ft/2)+c\fx+c^2\ft/4\Big\}\\
&\stackrel{\hphantom{\uptext{dist}}}{=}\fh(\ft,\fx;\fh_0(\fx+c\ft/2))+c\fx+c^2\ft/4.\qedhere
\end{align*}
\end{proof}

\begin{proof}[Sketch of the proof of Theorem~\ref{holder}]
Fix $\alpha<1/3$ and choose $\beta<1/2$ so that $\beta/(2-\beta)=\alpha$.
By the Markov property it is enough to assume that $\fh_0\in \mathscr C^\beta$ and check the H\"older-$\alpha$ regularity at time 0.
By space regularity of the Airy$_2$ process (proved in \cite{quastelRemAiry1}, but which also follows from Theorem~\ref{reg}) there is an $R<\infty$ a.s. such that $|\aip_2(\fx)|\le R(1+ |\fx|^{\beta})$, and making $R$ larger if necessary we may also assume $| \fh_0(\fx)- \fh_0( \fx_0)| \le R( |\fx-\fx_0|^{\beta} + |\fx-\fx_0|)$.
From the variational formula \eqref{eq:var}, $|\fh(\ft, \fx_0) - \fh(0,\fx_0)|$ is then bounded by
\begin{equation*}
\sup_{\fx\in\R}\Big(R  ( |\fx-\fx_0|^{\beta} + |\fx-\fx_0| + \ft^{1/3} +  \ft^{(1 - 2\beta)/3 }|\fx|^\beta)  - \tfrac1\ft (\fx_0-\fx)^2\Big).
\end{equation*}
The supremum is attained roughly at $x-x_0=\ft^{-\eta}$ with $\eta$ such that $|\fx-\fx_0|^{\beta}\sim\tfrac1\ft (\fx_0-\fx)^2$.
Then $\eta=1/(2-\beta)$ and the supremum is bounded by a constant multiple of $\ft^{\beta/(2-\beta)}=\ft^\alpha$, as desired.
\end{proof}

\section{The 1:2:3 scaling limit of TASEP}
\label{sec:TASEPlimit}

In this section we will prove that for a large class of initial data the growth process of TASEP converges to the KPZ fixed point defined in Section~\ref{sec:KPZfp}. To this end we consider the TASEP particles to be distributed with a density close to $\frac{1}{2}$, and take the following scaling of the height function $h$ from Section~\ref{sec:growth}:
\begin{equation}\label{eq:hep}
	\fh^{\ep}(\ft,\fx) = \ep^{1/2}\!\left[h_{2\ep^{-3/2}\ft}(2\ep^{-1}\fx) + \ep^{-3/2}\ft\right].
\end{equation}
We will always consider the linear interpolation of $\fh^{\ep}$ to make it a continuous function of $\fx\in \R$. Suppose that we have initial data  $X_0^\ep$ chosen to depend on $\ep$ in such a way that
\begin{equation}\label{xplim}
\fh_0=\lim_{\ep\to 0} \fh^{\ep}(0,\cdot)
\end{equation}
in the $\UC$ topology. For fixed $\ft>0$, we will prove that the limit 
\begin{equation}\label{eq:height-cvgce}
\fh(\ft,\fx;\fh_0)=\lim_{\ep\to0}\fh^{\ep}(\ft,\fx)
\end{equation}
exists, and take it as our \emph{definition} of the KPZ fixed point $\fh(\ft,\fx; \fh_0)$. We will often omit $\fh_0$ from the notation when it is clear from the context.

\begin{exercise} 
For any $\fh_0\in \UC$, we can find  initial data $X^\ep_0$  so that \eqref{xplim} holds.
\end{exercise}

\noindent We have the following convergence result for TASEP:

\begin{theorem}\label{thm:fullfixedpt}
For $\fh_0\in\UC$, let $X_0^\ep$ be initial data for TASEP such that the corresponding rescaled height functions $\fh_0^\ep$ converge to $\fh_0$ in the $\UC$ topology as $\ep\to0$.
Then the limit \eqref{eq:height-cvgce} exists (in distribution) locally in $\UC$ and is the KPZ fixed point with initial value $\fh_0$.
\end{theorem}

In other words, under the 1:2:3 scaling, as long as the initial data for TASEP converges in $\UC$, the evolving TASEP height function converges to the KPZ fixed point.

We now sketch the proof.  Our goal is to  compute $\pp_{\fh_0} \bigl(\fh(\ft,\fx_i)\leq \fa_i,\;i=1,\ldots,M\bigr)$. We chose for simplicity the frame of reference
\begin{equation}\label{eq:x0conv} 
  \xx_0^{-1}(-1)=1,
\end{equation}
i.e. the particle labeled $1$ is initially the rightmost in $\Z_{<0}$. Then it follows from \eqref{defofh}, \eqref{eq:hep} and \eqref{eq:height-cvgce} that the required probability should coincide with the limit as $\ep\to 0$ of
\begin{equation}\label{eq:TASEPtofp}
	\pp_{X_0}\!\left(X_{2\ep^{-3/2}\ft}(\tfrac12\ep^{-3/2}\ft-\ep^{-1}\fx_i-\tfrac12\ep^{-1/2}\fa_i+1)>2\ep^{-1}\fx_i-2,\;i=1,\ldots,M\right).
\end{equation}
We therefore want to consider Theorem~\ref{thm:tasepformulas} with 
\begin{equation}\label{eq:KPZscaling}
t=2\ep^{-3/2}\ft,\qquad  n_i = \tfrac{1}{2}\ep^{-3/2}\ft-\ep^{-1}\fx_i-\tfrac12\ep^{-1/2}\fa_i+1,\qquad a_i=2\ep^{-1}\fx_i-2.
\end{equation}

\begin{remark} One might worry that the initial data  \eqref{eq:Kt-2} is assumed to be right finite.  In fact, one can obtain a formula without this condition,
but it is awkward.  On the other hand, one could always cut off the TASEP data far to the right, take the limit, and then remove the cutoff.
If we call the macroscopic position of the cutoff $L$, this means 
 the cutoff data is $X_0^{\ep,L}(n) = X^\ep_0(n)$ if 
$n  > -\lfloor\ep^{-1}L\rfloor$  and $X_0^{\ep,L}(n) = \infty$ if $n\le -\lfloor\ep^{-1}L\rfloor$.  This corresponds to replacing $\fh^{\ep}_0(\fx)$ by $\fh^{\ep,L}_0(\fx)$ with a straight line with slope $-2\ep^{-1/2}$ to the right of $\ep X^\ep_0(-\lfloor\ep^{-1}L\rfloor)\sim 2L$.    The question is 
whether one can justify the exchange of limits $L\to\infty$ and $\ep\to 0$.  It turns out not to be a problem because one can use the exact formula
to get uniform bound (in $\ep$, and over initial data in $\UC$ with bound $C(1+|x|)$) that the difference of \eqref{eq:TASEPtofp} computed with initial data $X^\ep_0$ and with initial data $X_0^{\ep,L}$ is bounded by $ C e^{-cL^3}$. 
\end{remark}

\begin{lemma}\label{lem:KernelLimit1}
Under the scaling \eqref{eq:KPZscaling} (dropping the $i$ subscripts) and assuming that \eqref{xplim} holds, if we set $z=2\ep^{-1}\fx+\ep^{-1/2}(u+\fa)-2$ and  $y'=\ep^{-1/2}v$, then we have for $\ft>0$, as $\eps \to 0$,
\begin{align}\label{eq:QRcvgce}
\fT^\ep_{-\ft,\fx}(v,u)&:=\ep^{-1/2}\SM_{-t,-n}(y',z)\longrightarrow{} \fT_{-\ft,\fx}(v,u),\\\label{eq:QRcvgce2}
\bar\fT^\ep_{-\ft,-\fx}(v,u)&:=\ep^{-1/2}\SN_{-t,n}(y',z)\longrightarrow {} \fT_{-\ft,-\fx}(v,u),\\\label{eq:QRcvgce3}
\bar{\fT}^{\ep,\epi(-\fh_0^-)}_{-\ft,-\fx}(v,u)&:=\ep^{-1/2} {\SN}^{\oepi(X_0)}_{-t,n}(y',z)\longrightarrow {} \bar{\fT}^{\epi(-\fh_0^-)}_{-\ft,-\fx}(v,u)
\end{align} 
pointwise, where $\fh_0^-(x)=\fh_0(-x)$ for $x\geq0$.  Here $\SM_{-t,-n}$ and $\SN_{-t,n}$ are defined in \eqref{def:sm} and \eqref{def:sn}. 
\end{lemma}

\begin{proof}
Note that from  \eqref{def:sm},\eqref{def:sn}, $ \SN_{-t,n} (z_1,z_2)=  \SM_{-t,-n+1-z_1+z_2}(z_2,z_1)$, so \eqref{eq:QRcvgce2}
follows from  \eqref{eq:QRcvgce}.  
By changing variables $w\mapsto\frac12(1-\ep^{1/2}\tilde w)$ in \eqref{def:sm},  and using the scaling \eqref{eq:KPZscaling}, we have
\begin{align}
&\fT^\ep_{-\ft,\fx}(u)= \frac{1}{2\pi\I} \oint_{\tts C_\ep}d\tilde w\, e^{\ep^{-3/2}\ft F_\ep(\ep^{1/2} \tilde w,\ep^{1/2}\fx_\ep/\ft,\ep u_\ep/\ft ) } ,
\label{eq:ft1} \\\label{eq:ft2}
&F( w,x,u)= 
(  \arctanh w-w)-
 x\log(1- w^2) - u  \arctanh w
\end{align}
where $\fx_\ep= \fx+\ep^{1/2} (\fa-u)/2+\ep/2  $ and $u_\ep = u+\ep^{1/2} $  and $C_\ep $ is a circle of radius $\ep^{-1/2}$ centred at $\ep^{-1/2}$
and $\arctanh w = \tfrac12[\log(1+w)-\log(1-w)]$.  Note that 
\begin{equation}\label{eq:ft3}
\partial_w  F( w,x,u) =   ( w-w_+)(w-w_-)(1- w^2)^{-1},\qquad w_\pm = x \pm \sqrt{ x^2+ u},
\end{equation}
From \eqref{eq:ft3} it is easy to see that as $\ep\searrow 0$, $\ep^{-3/2}\ft F_\ep(\ep^{1/2} \tilde w,\ep^{1/2}\fx_\ep/\ft,\ep u_\ep/\ft )$  converges to the corresponding exponent in \eqref{eq:fTdef} (keeping in mind that $
\fT_{-\ft,\fx}=(\fT_{\ft,\fx})^*$).  Alternatively, one can just use \eqref{eq:ft2} and that for small $w$, $  \arctanh w-w\sim w^3/3$,
$-\log(1-w^2)\sim w^2$ and $\arctanh w\sim w$. Deform $C_\ep$ to the contour $\langle_{{}_\ep}~\cup~C^{\pi/3}_\ep$  where $\langle_{{}_\ep} $ is the part of the Airy contour $\langle$ within the ball of radius $\ep^{-1/2}$ centred at $\ep^{-1/2}$, and $C^{\pi/3}_\ep$ is the part of $C_\ep$ to the right of $\langle$.  
  As $\ep\searrow 0$, $\langle_{{}_\ep} \to \langle$, so it only remains to show that the integral over $C^{\pi/3}_\ep$ converges to $0$.
To see this note that the real part of the exponent of the integral over $\Gamma_0$ in \eqref{def:sm}  is given by $ \ep^{-3/2}\tfrac{\ft}2 [ \cos\theta -1 + (\tfrac18- c\ep^{1/2})\log( 1-4(\cos\theta -1)] $ where $w=\tfrac12 e^{i\theta}$and $c= \fx/\ft + \fa/2\ep^{1/2} +\ep$.  Using $\log(1+x) \le x$ for $x\ge 0$, this is less than or equal
to  $
\ep^{-3/2}\tfrac{\ft}8 [ \cos\theta -1] 
$ for sufficiently small $\ep$.  The  $\tilde{w}\in C^{\pi/3}_\ep$ corresponds to $\theta\ge \pi/3$, so the exponent there is less than $-\ep^{-3/2}\kappa\ft$ for some $\kappa>0$.  Hence this part of the integral vanishes.

Now define the scaled walk $\fB^\ep(\fx) = \ep^{1/2}\left(B_{\ep^{-1}\fx} + 2\ep^{-1}\fx-1\right)$ for $\fx\in \ep\Z_{\geq0}$, interpolated linearly in between, and let $\ftau^\ep$ be the hitting time by $\fB^\ep$ of $\epi(-\fh^{\ep}(0,\cdot)^-)$.
By Donsker's invariance principle~\cite{billingsley}, $\fB^\ep$ converges locally uniformly in distribution to a Brownian motion $\fB(\fx)$ with diffusion coefficient $2$, and therefore (using convergence of the initial values of TASEP) the hitting time $\ftau^\ep$ converges to $\ftau$ as well.
\end{proof}

We will compute next the limit of \eqref{eq:TASEPtofp} using \eqref{eq:extKernelProb} under the scaling \eqref{eq:KPZscaling}.
To this end we change variables in the kernel as in Lemma~\ref{lem:KernelLimit1}, so that for $z_i=2\ep^{-1}\fx_i+\ep^{-1/2}(u_i+\fa_i)-2$ we need to compute the limit of $\ep^{-1/2}\big(\bar\chi_{2\ep^{-1}\fx-2}K_t\bar\chi_{2\ep^{-1}\fx-2}\big)(z_i,z_j)$.
Note that the change of variables turns $\bar\chi_{2\ep^{-1}\fx-2}(z)$ into $\bar\chi_{-\fa}(u)$.
We have $n_i<n_j$ for small $\ep$ if and only if $\fx_j<\fx_i$ and in this case we have, under our scaling,
\[\ep^{-1/2}Q^{n_j-n_i}(z_i,z_j)\longrightarrow e^{(\fx_i-\fx_j)\p^2}(u_i,u_j),\]
as $\ep\to 0$. For the second term in \eqref{eq:Kt-2} we have
\begin{multline}
\ep^{-1/2}(\SM_{t,-{n}_i})^*\SN_{t,n_j}(z_i,z_j)=\ep^{-1}\int_{-\infty}^{\infty}dv\,(\fT_{-\ft,\fx_i}^\ep)^*(u_i,\ep^{-1/2}v)\bar\fT^{\ep,\epi(-\fh_0^-)}_{-\ft,-\fx_j}(\ep^{-1/2}v,u_j)\\
\xrightarrow[\ep\to0]{}(\fT_{-\ft,\fx_i})^*\fT_{-\ft,-\fx_j}(u_i,u_j)
\end{multline}
(modulo suitable decay of the integrand).
Thus we obtain a limiting kernel
\begin{equation}\label{def:firstker}
\fK_{\lim}(\fx_i,u_i;\fx_j,u_j)=-e^{(\fx_i-\fx_j)\p^2}(u_i,u_j)\1{\fx_i>\fx_j}+(\fT_{-\ft,\fx_i})^*\bar \fT^{\epi(-\fh_0^-)}_{-\ft,-\fx_j}(u_i,u_j),
\end{equation}
surrounded by projections $\bar\chi_{-\fa}$.
Our computations here only give pointwise convergence of the kernels, but they can be upgraded to trace class convergence (see \cite{KPZ}), which thus yields convergence of the Fredholm determinants.

We prefer the projections $\bar\chi_{-\fa}$ which surround \eqref{def:firstker} to read $\chi_{\fa}$, so we change variables $u_i\longmapsto-u_i$ and replace the Fredholm determinant of the kernel by that of its adjoint to get 
\[\det\left(\fI-\chi_{a}\fK^{\hypo(\fh_0)}_{\ft,\uptext{ext}}\chi_{a}\right)\qquad\text{with}\qquad
\fK^{\hypo(\fh_0)}_{\ft,\uptext{ext}}(u_i,u_j)=\fK_{\lim}(\fx_j,-u_j;\fx_i,-u_i).\]
The choice of superscript $\hypo(\fh_0)$ in the resulting kernel comes from the fact 
\[
\bar \fT^{\epi(-\fh_0^-)}_{-\ft,\fx}(v,-u)= \bigl(\bar\fT^{\hypo(\fh_0^-)}_{\ft,\fx}\bigr)^*(u,-v), 
\]
which together with $\fT_{-\ft,\fx}(-u,v)=(\fT_{\ft,\fx})^*(-v,u)$ yield 
\begin{equation}\label{eq:Kexthalf}
\fK^{\hypo(\fh_0)}_{\ft,\uptext{ext}}=-e^{(\fx_j-\fx_i)\p^2}\1{\fx_i<\fx_j}+(\bar \fT^{\hypo(\fh_0^-)}_{\ft,-\fx_i})^*\fT_{\ft,\fx_j}.
\end{equation}
This gives the following one-sided fixed point formula for the limit $\fh$:  

\begin{theorem}[One-sided fixed point formula]\label{thm:Kfixedpthalf}
Let $\fh_0\in \UC$ with  $\fh_0(\fx) = -\infty$ for $\fx>0$.
	Then given $\fx_1<\fx_2<\dotsm<\fx_M$ and $\fa_1,\ldots,\fa_M\in\R$, we have
	\begin{align}
		&\pp_{\fh_0} \bigl(\fh(\ft,\fx_1)\leq \fa_1,\ldots,\fh(\ft,\fx_M)\leq \fa_M\bigr) =\det \left(\fI-\chi_{\fa} \fK^{\hypo(\fh_0)}_{\ft,\uptext{ext}}\chi_{\fa}\right)_{L^2(\{\fx_1,\ldots,\fx_M\}\times\R)}\\
    		&\hspace{0.275in}=\det \left(\fI-\fK^{\hypo(\fh_0)}_{\ft,\fx_M}+\fK^{\hypo(\fh_0)}_{\ft,\fx_M}e^{(\fx_1-\fx_M)\p^2}\bar \P_{\fa_1}e^{(\fx_2-\fx_1)\p^2}\bar \P_{\fa_2}\dotsm e^{(\fx_M-\fx_{M-1})\p^2}\bar \P_{\fa_M}\right)_{L^2(\R)},\label{eq:onesidepath}
	\end{align}
	with the kernel
	\begin{equation}
	    \fK^{\hypo(\fh_0)}_{\ft,\fx}(\cdot,\cdot) = \fK^{\hypo(\fh_0)}_{\ft,\uptext{ext}}(\fx,\cdot;\fx,\cdot),
	\end{equation}
	where the latter is defined in \eqref{eq:Kexthalf}.
\end{theorem}

\noindent The second identity in \eqref{eq:onesidepath} can be obtained similarly to \eqref{eq:path-int-kernel-TASEPgem} for the discrete kernels.

Our next goal is to take a continuum limit in the $\fa_i$'s of the path-integral formula \eqref{eq:onesidepath} on an interval $[-L,L]$ and then take $L\to\infty$. For this we take $\fx_1,\dotsc,\fx_M$ to be a partition of $[-L,L]$ and take $\fa_i=\fg(\fx_i)$. Then taking the limit $M \to \infty$ we get as in \cite{flat} (and actually dating back to \cite{cqr})
\begin{equation}
\fT_{\ft/2,-L}\bar\P_{\fg(\fx_1)}e^{(\fx_2-\fx_1)\p^2}\bar\P_{\fg(\fx_2)}\dotsm e^{(\fx_M-\fx_{M-1})\p^2}\bar\P_{\fg(\fx_M)}(\fT_{\ft/2,-L})^*
\longrightarrow\fT_{\ft/2,-L}\wt\Theta^{\fg}_{-L,L}(\fT_{\ft/2,-L})^*,\label{eq:continuumLimit}
\end{equation}
where 
\[
\wt\Theta^g_{\ell_1,\ell_2}(u_1,u_2)=\pp_{\fB(\ell_1)=u_1}\big(\fB(s)\leq\fg(s)~\forall\,s\in[\ell_1,\ell_2],\,\fB(\ell_2)\in du_2\big)/du_2.
\]
When we pass now to the limit $L \to \infty$, one can see (at least roughly) that we obtain 
\[
\fT_{\ft/2,-L}\wt\Theta^{\fg}_{-L,L}(\fT_{\ft/2,-L})^*\longrightarrow \fI-\fK^{\epi(\fg)}_{-\ft/2}.
\]
One can find a rigorous proof of these results in \cite{KPZ}. After taking these limits, the Fredholm determinant from \eqref{eq:onesidepath} thus converges to
\begin{equation}
\det\left(\fI-\fK^{\hypo(\fh_0)}_{\ft/2}+\fK^{\hypo(\fh_0)}_{\ft/2} \bigl(\fI-\fK^{\epi(\fg)}_{-\ft/2}\bigr)\right)
=\det\left(\fI-\fK^{\hypo(\fh_0)}_{\ft/2}\fK^{\epi(\fg)}_{-\ft/2}\right),
\end{equation}
which is exactly the content of Theorem~\ref{thm:fullfixedpt}. 

As in the TASEP case, the kernel in \eqref{eq:Kexthalf} can be rewritten (thanks to the analog of \eqref{eq:epiabove}) as
\begin{equation}\label{eq:Kexthalf-alt}
    \fK^{\hypo(\fh_0)}_{\ft,\uptext{ext}}(\fx_i,\cdot;\fx_j,\cdot)=-e^{(\fx_j-\fx_i)\partial^2}\1{\fx_i<\fx_j}
    +(\fT_{\ft,-\fx_i})^*\bar \P_{\fh_0(0)}\fT_{\ft,\fx_j}+(\bar{\fT}^{\hypo(\fh_0^-)}_{\ft,-\fx_i})^* \P_{\fh_0(0)}\fT_{\ft,\fx_j}.
\end{equation}

\subsection{From one-sided to two-sided formulas.}

Now we derive the formula for the KPZ fixed point with general initial data $\fh_0$  
as the $L\to\infty$ limit of the formula with initial data \[\fh_0^L(\fx)=\fh_0(\fx)\1{\fx\leq L}-\infty\cdot\1{\fx>L},\] which can be obtained from the previous theorem by translation invariance. 
We then take, in the next subsection, a continuum limit of the operator $e^{(\fx_1-\fx_M)\p^2}\bar \P_{a_1}e^{(\fx_2-\fx_1)\p^2}\bar \P_{\fa_2}\dotsm e^{(\fx_M-\fx_{M-1})\p^2}\bar \P_{\fa_M}$ on the right side of \eqref{eq:onesidepath} to obtain a ``hit'' operator for the final data as well.  The result of all this is the same as if we started with two-sided data for TASEP.  

The shift invariance of TASEP, tells us that $\fh(\ft, \fx; \fh_0^L)\stackrel{\text{dist}}{=} \fh(\ft,\fx-L;\theta_L\fh_0^L)$, where $\theta_L$ is the shift operator.
Our goal then is to take $L\to\infty$ in the formula given in Theorem~\ref{thm:Kfixedpthalf} for $\fh(\ft,\fx-L;\theta_L\fh_0^L)$.
We get
\[\pp_{\theta_L\fh_0}\!\left(\fh(\ft,\fx_1-L)\leq \fa_1,\ldots,\fh(\ft,\fx_M-L)\leq \fa_M\right)
=\det\left(\fI-\chi_{a}\wt\fK^{\theta_L\fh_0^L}_L\chi_{a}\right)_{L^2(\{\fx_1,\ldots,\fx_M\}\times\R)}\]
with
\begin{equation}\label{eq:onesidedforlimit}
\wt\fK^{\theta_L\fh_0^L}_L(\fx_i,\cdot ;\fx_j,\cdot)=e^{(\fx_j-\fx_i)\p^2}\1{\fx_i<\fx_j}+e^{(\fx_j-\fx_i)\p^2}\bigl(\bar{\fT}^{\hypo((\theta_L\fh_0^L)^-_0)}_{\ft,-\fx_j+L}\bigr)^*\fT_{\ft,\fx_j-L}.
\end{equation}
Since $(\theta_{L}\fh_0^L)_0^+\equiv-\infty$, we may rewrite $\bigl(\bar{\fT}^{\hypo((\theta_L\fh_0^L)^-_0)}_{\ft,-\fx_j+L}\bigr)^*\fT_{\ft,\fx_j-L}$ as\footnote{At first glance it may look as if the product $e^{-\fx\p^2}\fK^{\hypo(\fh)}_\ft e^{\fx\p^2}$ makes no sense, because $\fK^{\hypo(\fh)}_{\ft}$ is given in \eqref{eq:Kepihypo} as the identity minus a certain kernel, and applying $e^{\fx\p^2}$ to $\fI$ is ill-defined for $\fx<0$.
However, thanks to the analog of \eqref{eq:Kep-expansion} below for $\fK^{\hypo(\fh)}_\ft$, the action of $e^{\fx\p^2}$ on this kernel on the left and right is well defined for any $\fx\in\R$.
This also justifies the identity in \eqref{eq:heatkerout}.}
\[
\fI- \bigl(\fT_{\ft,-\fx_j+L}-\bar{\fT}^{\hypo((\theta_L\fh_0^L)^-_0)}_{\ft,-\fx_j+L}\bigr)^*\bigl(\fT_{\ft,\fx_j-L}-\bar{\fT}^{\hypo((\theta_L\fh_0^L)^+_0)}_{\ft,\fx_j-L}\bigr)
=e^{-\fx_j\p^2}\fK^{\hypo(\fh_0^L)}_\ft e^{\fx_j\p^2},\label{eq:heatkerout}
\]
where $\fK^{\hypo(\fh_0)}_\ft$ is the kernel defined in \eqref{eq:Kepihypo}.
Note the crucial fact that the right hand side depends on $L$ only through $\fh_0^L$ (the various shifts by $L$ on the left hand side of \eqref{eq:heatkerout} play no role).
It was shown in~\cite{flat} that $\fK^{\hypo(\fh_0^L)}_\ft\longrightarrow\fK^{\hypo(\fh_0)}_\ft$ as $L\to\infty$.
This tells us that the second term on the right hand side of \eqref{eq:onesidedforlimit} equals $e^{-\fx_i\p^2}\fK^{\hypo(\fh_0^L)}_\ft e^{\fx_j\p^2}$, which converges to $e^{-\fx_i\p^2}\fK^{\hypo(\fh_0)}_\ft e^{\fx_j\p^2}$ as $L\to\infty$, and leads to

\begin{theorem}[Two-sided fixed point formula]\label{thm:two-sided-lim-extended}
Let $\fh_0\in \UC$  and $\fx_1<\fx_2<\dotsm<\fx_M$.
Then for $\fh(\ft,\fx)$ given as in \eqref{eq:height-cvgce},
\begin{align}
&\pp_{\fh_0}\bigl(\fh(\ft,\fx_1)\leq\fa_1,\ldots,\fh(\ft,\fx_M)\leq\fa_M\bigr) =\det\left(\fI-\chi_{\fa}\fK^{\hypo(\fh_0)}_{\ft,\uptext{ext}}\chi_{\fa}\right)_{L^2(\{\fx_1,\ldots,\fx_M\}\times\R)}\label{eq:twosided-ext}\\
&\hspace{0.3in}=\det\left(\fI-\fK^{\hypo(\fh_0)}_{\ft,\fx_M}+\fK^{\hypo(\fh_0)}_{\ft,\fx_M} e^{(\fx_1-\fx_M)\p^2}\bar\P_{\fa_1}e^{(\fx_2-\fx_1)\p^2}\bar\P_{\fa_2}\dotsm e^{(\fx_M-\fx_{M-1})\p^2}\bar\P_{\fa_M}\right)_{L^2(\R)}\label{eq:twosided-path}
\end{align}
where the kernels are as in Theorem~\ref{thm:Kfixedpthalf}.
\end{theorem}

\subsection{Continuum limit.}\label{sec:continuumfixedpt}

We turn now to the continuum limit in the $\fa_i$'s of the path-integral formula \eqref{eq:twosided-path} on an interval $[-R,R]$ (we will take $R\to\infty$ later on).
To this end we conjugate the kernel inside the determinant by $\fT_{\ft/2,\fx_M}$, leading to
\[\wt\fK^{\fh_0}_{\ft,\fx_M}-\wt\fK^{\fh_0}_{\ft,\fx_M}\big[\fT_{\ft/2,\fx_1}\bar\P_{\fa_1}e^{(\fx_2-\fx_1)\p^2}\bar\P_{\fa_2}\dotsm e^{(\fx_M-\fx_{M-1})\p^2}\bar\P_{\fa_M}(\fT_{\ft/2,-\fx_M})^*\big]\]
with $\wt\fK^{\fh_0}_{\ft,\fx_M}=\fT_{\ft/2,\fx_M}\fK^{\hypo(\fh_0)}_{\ft,\fx_M}(\fT_{\ft/2,-\fx_M})^*=\fK^{\hypo(\fh_0)}_{\ft/2}$ (the second equality follows from \eqref{eq:Kepihypo}).
Now we take the limit of term in brackets, letting $\fx_1,\ldots,\fx_M$ be a partition of $[-R,R]$ and taking $M\to\infty$ with $\fa_i=\fg(\fx_i)$.
As in~\cite{cqr} we have
\[
\fT_{\ft/2,-R}\,\bar\P_{\fg(\fx_1)}e^{(\fx_2-\fx_1)\p^2}\bar\P_{\fg(\fx_2)}\dotsm e^{(\fx_M-\fx_{M-1})\p^2}\bar\P_{\fg(\fx_M)}(\fT_{\ft/2,-R})^*
\longrightarrow\fT_{\ft/2,-R}\,\wt\Theta^{\fg}_{-R,R}(\fT_{\ft/2,-R})^*,\label{eq:continuumLimit}
\]
where $\wt\Theta^g_{\ell_1,\ell_2}(u_1,u_2)=\pp_{\fB(\ell_1)=u_1}\big(\fB(s)\leq\fg(s)~\forall\,s\in[\ell_1,\ell_2],\,\fB(\ell_2)\in du_2\big)/du_2$,

Next we rewrite the probability as
\begin{multline*}
  \pp_{\ell_1, x_1;\ell_2, x_2}\!\left(B(t)\leq g(t)\text{ on }[\ell_1,\ell_2]\right)
  = \int_{-\infty}^{g(\alpha)}\pp_{\ell_1, x_1; \ell_2, x_2} ( B(\alpha)\in dy)\\
  \times \pp_{\ell_1, x_1; \alpha, y} (B(t)<g(t)\text{ on }[\ell_1,\alpha])\pp_{\alpha, y;\ell_2, x_2}(
  B(t)<g(t)\text{ on } [\alpha,\ell_2]).
\end{multline*}
The last probability in the above integral can be rewritten as
\begin{equation*}
  \pp_{0, y;\ell_2-\alpha, x_2}(B(t)<g_\alpha(t)\text{ on } [0,\ell_2-\alpha])
  =1-\int_0^{\ell_2-\alpha}\pp_y(\tau_{g_\alpha^+}\in dt)\tfrac{p(\ell_2-\alpha-t,x_2-g_\alpha(t))}{p(\ell_2-\alpha,x_2-y)}.
\end{equation*}
A similar identity can be written for $\pp_{\ell_1, x_1; \alpha, y}(B(t)<g(t)\text{ on }[\ell_1,\alpha])$, now using $\tau_{g_a^-}$, which we take to be independent of $\tau_{g_\alpha^+}$, and going backwards from time $\alpha$ to time
$\ell_1$.
Using this and writing $\pp_{\ell_1, x_1; \ell_2, x_2}(B(\alpha)\in dy)$ explicitly we find that
\begin{align*}
  \pp_{\ell_1, x_1;\ell_2, x_2}\!\left(B(t)<g(t)\text{ on }[\ell_1,\ell_2]\right)
  &=\int_{-\infty}^{g(\alpha)}dy\sqrt{\tfrac{\ell_2-\ell_1}{4\pi(a-\ell_1)(\ell_2-\alpha)}}
  e^{-\frac{((\ell_2-\alpha)x_1+(\alpha-\ell_1)x_2+(\ell_1-\ell_2)y)^2}{4(\alpha-\ell_1)(\ell_2-\alpha)(\ell_2-\ell_1)}}\\
  &\hspace{0.35in}\times\left(1-\int_0^{\alpha-\ell_1}\pp_y(\tau_{g_\alpha^-}\in dt_1)\tfrac{p(
      \alpha-\ell_1 - t_1, x_1-g_\alpha(-t_1))}{ p(\alpha-\ell_1 , x_1-y) }\right)\\
  &\hspace{0.35in}\times\left(1-\int_{0}^{\ell_2-\alpha}\pp_y(\tau_{g_\alpha^+}\in dt_2)\tfrac{
      p(\ell_2-\alpha - t_2, x_2-g_\alpha(t_2))}{p(\ell_2-\alpha , x_2-y) } \right)
\end{align*}
Recalling  that in the formula for
$\wt\Theta^g_{\ell_1,\ell_2}(x_2-x_1)$ this probability is premultiplied by
$p(\ell_2-\ell_1,x_2-x_1+\ell_2^2-\ell_1^2)$ and observing that
\[\tfrac{p(\ell_2-\ell_1,x_2-x_1+\ell_2^2-\ell_1^2)}{p(a-\ell_1,x_1-y)p(\ell_2-a,x_2-y)}\sqrt{\tfrac{\ell_2-\ell_1}{4\pi(a-\ell_1)(\ell_2-a)}}
e^{-\frac{((\ell_2-a)x_1+(a-\ell_1)x_2+(\ell_1-\ell_2)y)^2}{4(a-\ell_1)(\ell_2-a)(\ell_2-\ell_1)}}
=e^{\frac14(\ell_1^2-\ell_2^2+2x_1-2x_2)(\ell_1+\ell_2)}\] we deduce that
\begin{multline*}
  \wt\Theta^g_{\ell_1,\ell_2}(x_1,x_2)=e^{\frac14(\ell_1^2-\ell_2^2+2x_1-2x_2)(\ell_1+\ell_2)}\\
  \times\int_{-\infty}^{g(\alpha)}dy\left[p(\alpha-\ell_1 , x_1-y)-\int_0^{\alpha-\ell_1}dt_1\,\pp_{y}(\tau_{g_\alpha^-}\in dt_1)p(\alpha-\ell_1 - t_1, x_1-g_\alpha(-t_1))\right]\\
  \times\left[p(\ell_2-\alpha , x_2-y)-\int_{0}^{\ell_2-\alpha}dt_2\,\pp_{y}(\tau_{g_\alpha^+}\in
    dt_2)p(\ell_2 -\alpha- t_2, x_2-g_\alpha(t_2))\right].
\end{multline*}
Taking $-\ell_1=\ell_2=L$, we have that 
  for any $\alpha\in(-L,L)$ we have
  \begin{multline*}
    A^*e^{-L\Delta}\wt\Theta^g_{-L,L}e^{-L\Delta}A(\lambda_1,\lambda_2)\\
    =\int_{-\infty}^{g(\alpha)}dy\left[A^*e^{\alpha \Delta}(\lambda_1,y)-\int_{0}^{\alpha+L}dt_1\,\pp_y(\tau_{g_\alpha^-}\in
      dt_1)A^*e^{(\alpha-t_1)\Delta}(\lambda_1,g_\alpha(-t_1))\right]\\\times
    \left[e^{-\alpha \Delta}A(y,\lambda_2)-\int_{0}^{L-\alpha}dt_2\,\pp_y(\tau_{g_\alpha^+}\in
      dt_2)e^{-(\alpha+t_2)\Delta}A(g_\alpha(t_2),\lambda_2)\right].
  \end{multline*}
where  the \emph{Airy transform} $A$, is defined by
\begin{equation}
  \label{eq:B0}
  A(x,\lambda)=\Ai(x-\lambda).
\end{equation}
Now we have  $\fT_{\ft/2,-R}\wt\Theta^{\fg}_{-R,R}(\fT_{\ft/2,-R})^*\longrightarrow \fI-\fK^{\epi(\fg)}_{-\ft/2}$ as $R\to\infty$.
Our Fredholm determinant is thus now given by
\[
\det\left(\fI-\fK^{\hypo(\fh_0)}_{\ft/2}+\fK^{\hypo(\fh_0)}_{\ft/2}(\fI-\fK^{\epi(\fg)}_{-\ft/2})\right)
=\det\left(\fI-\fK^{\hypo(\fh_0)}_{\ft/2}\fK^{\epi(\fg)}_{-\ft/2}\right).\label{eq:preThmfixedpt}
\]
which is the KPZ fixed point formula.

\begin{exercise} The Airy transform satisfies $AA^*=I$, so that $f(x)=\int_{-\infty}^\infty d\lambda\Ai(x-\lambda)Af(\lambda)$.
In other words, the shifted Airy functions $\{\Ai(x-\lambda)\}_{\lambda\in\R}$ (which are not in $L^2(\R)$) form a generalized orthonormal basis of $L^2(\R)$. Thus the \emph{Airy kernel}  $\K(x,y)=\int_{-\infty}^0 d\lambda\Ai(x-\lambda)\!\Ai(y-\lambda)$ is the projection onto the subspace spanned by $\{\Ai(x-\lambda)\}_{\lambda\leq0}$).  Show \begin{equation}\label{eq:K-AiryTr}
 	\K=A\bar \chi_0A^*,
\end{equation} 
\begin{equation}\label{eq:FGUE}
	F_{\rm GUE}(r)=\det\!\big(I-\chi_r\K \chi_r\big)=\det\!\big(I-\K \chi_r\K\big)
\end{equation}
and$^{*}$
\begin{equation}\label{eq:FGOE}
	F_{\rm GOE}(4^{1/3}r)=\det\!\big(I-\K\varrho_r\K\big)
\end{equation}
where $\varrho_r$ is the reflection operator
\begin{equation}\label{eq:varrho}
	\varrho_rf(x)=f(2r-x).
\end{equation}
\end{exercise}


\bibspread

\bibliography{biblio}

\end{document}